\documentclass[a4paper, preprint]{amsart}
\usepackage{amsmath,amssymb,amsthm}
\def\ocirc#1{\ifmmode\setbox0=\hbox{$#1$}\dimen0=\ht0
    \advance\dimen0 by1pt\rlap{\hbox to\wd0{\hss\raise\dimen0
    \hbox{\hskip.2em$\scriptscriptstyle\circ$}\hss}}#1\else
    {\accent"17 #1}\fi}
\usepackage{float}
\usepackage{color}
\usepackage{graphicx}
\usepackage{xcolor}
\usepackage{algorithm}
\usepackage{algorithmic}

\usepackage{dsfont}
\usepackage[colorlinks = false]{hyperref}

\numberwithin{equation}{section}

\theoremstyle{plain}
\newtheorem{theorem}{Theorem}[section]
\newtheorem{corollary}[theorem]{Corollary}
\newtheorem{lemma}[theorem]{Lemma}
\theoremstyle{remark}

\newtheorem{example}{Example}[section]
\newtheorem{definition}{Definition}[section]
\newtheorem{remark}{Remark}[section]

\newcommand{\coloneq}{\mathrel{\vcenter{\baselineskip0.5ex \lineskiplimit0pt
                     \hbox{\scriptsize.}\hbox{\scriptsize.}}}=}

\newcommand{\eqcolon}{=\mathrel{\vcenter{\baselineskip0.5ex \lineskiplimit0pt
                     \hbox{\scriptsize.}\hbox{\scriptsize.}}}}

\newcommand{\BV}[1]{\mathrm{BV}(#1)}

\newcommand{\Lip}[1]{\mathrm{Lip}_0(#1)}
\newcommand{\cI}{\mathcal{I}}

\newcommand{\set}[2]{\left\{ #1 \, ; \, #2\right\}}

\newcommand{\1}[1]{{\ensuremath{\mathbf 1_{#1}}}}
\newcommand{\cO}[1]{{\mathcal{O}}\!\left(#1\right)}
\newcommand{\tO}[1]{{\Theta}\!\left(#1\right)}

\newcommand{\abs}[1]{{\left| #1 \right|}}
\newcommand{\norm}[1]{{\left\| #1 \right\|}}

\newcommand{\prt}[1]{\left( #1 \right)}
\newcommand{\ceil}[1]{\left\lceil #1 \right\rceil}
\newcommand{\floor}[1]{\left\lfloor #1 \right\rfloor}
\newcommand{\seq}[1]{\left\{#1\right\}}

\DeclareMathOperator{\orm}{orm}

\newcommand{\Leb}[1]{\mathrm{Leb}\left(#1\right)}
\newcommand{\supp}[1]{\mathrm{supp}\left(#1\right)}

\DeclareMathOperator*{\sgn}{sign}
\newcommand{\sign}[1]{\sgn\left(#1\right)}

\newcommand{\bN}{\mathbb{N}}
\newcommand{\bP}{\mathbb{P}}
\newcommand{\bR}{\mathbb{R}}
\newcommand{\bT}{\mathbb{T}}
\newcommand{\bZ}{\mathbb{Z}}

\newcommand{\bbN}{\mathbb{N}}
\newcommand{\bbR}{\mathbb{R}}

\newcommand{\cD}{\ensuremath{\mathcal{D}}}

\newcommand{\cF}{\mathcal{F}}

\newcommand{\Prob}[1]{\mathbb{P}\mspace{-2mu}\left( #1 \right)}

\newcommand{\E}[1]{{\ensuremath{\mathbb{E}}\mspace{-2mu}\left[#1\right]}}

\newcommand{\cS}{\ensuremath{\mathcal{S}}}

\newcommand{\R}{\mathbb{R}}

\newcommand{\dt}{\,\mathrm{d}t}
\newcommand{\dx}{\,\mathrm{d}x}
\newcommand{\dy}{\,\mathrm{d}y}

\title[Conservation laws with rough flux] {Numerical methods for conservation laws \\ with rough flux}

\author[H.~Hoel]{H. Hoel}
\address[H{\aa}kon Hoel]
{\newline Mathematics Institute of Computational Science and Engineering
\newline \'Ecole polytechnique f\'ed\'erale de Lausanne
\newline  EPFL / SB / MATH-CSQI, MA C1 644 (B{\^a}timent MA), Station 8
\newline CH-1015, Lausanne 
\newline Switzerland}
\email[]{hakon.hoel@epfl.ch, hhakon@chalmers.se}

\author[K. H. Karlsen]{K. H. Karlsen}
\address[Kenneth Hvistendahl Karlsen]
{\newline Department of mathematics
\newline University of Oslo
\newline P.O. Box 1053,  Blindern
\newline N--0316 Oslo, Norway} 
\email[]{kennethk@math.uio.no}

\author[N. H. Risebro]{N. H. Risebro}
\address[Nils Henrik Risebro]
{\newline Department of mathematics
\newline University of Oslo
\newline P.O. Box 1053,  Blindern
\newline N--0316 Oslo, Norway} 
\email[]{nilsr@math.uio.no}

\author[E. B. Storr\o{}sten]{E. B. Storr\o{}sten}
\address[Erlend Briseid Storr\o{}sten]
{\newline Department of mathematics
\newline University of Oslo
\newline P.O. Box 1053, Blindern
\newline N--0316 Oslo, Norway} 
\email[]{erlenbs@math.uio.no}

\date{\today}



\thanks{This work received supported by the Research Council of Norway 
through the project Stochastic Conservation Laws (250674/F20) and by 
the KAUST CRG4 Award Ref:2584.}

\begin{document}

\begin{abstract}
Finite volume methods are proposed for computing 
approximate pathwise entropy/kinetic solutions 
to conservation laws with a rough path dependent flux function. 
For a convex flux, it is demonstrated that rough 
path oscillations may lead to ``cancellations'' 
in the solution. Making use of this property, we show that for 
$\alpha$-H{\"o}lder continuous rough paths the 
convergence rate of the numerical methods can improve 
from $\mathcal{O}(\text{COST}^{-\gamma})$,
for some $\gamma \in \left[\alpha/(12-8\alpha), \alpha/(10-6\alpha)\right]$, with 
$\alpha\in (0, 1)$, to $\mathcal{O}(\text{COST}^{-\min(1/4,\alpha/2)})$. 
Numerical examples support the theoretical results.
\end{abstract}

\maketitle 

\section{Introduction}\label{sec:Intro}
The inclusion of random effects is important for the development
of realistic models of physical phenomena.  Frequently  such models lead to
nonlinear stochastic partial differential equations (SPDEs), 
whose solutions may possess singularities, reflecting 
the appearance of shock waves, turbulence, or other physical features. 
Recently many researchers have targeted a wide range of 
questions relating to mathematical analysis and numerical methods 
for stochastic conservation laws and related SPDEs.  
Among the numerous questions addressed, we mention selection 
principles for singling out ``correct'' generalized solutions, 
theories of well-posedness (existence, uniqueness, 
stability of solutions), regularity and compactness properties (sometimes 
improved by the inclusion of noise), existence of 
invariant measures, and construction of convergent numerical methods.

Randomness can enter models in different ways, such as stochastic
forcing or uncertain system parameters as well as random initial and
boundary data. For example, a number of mathematical works
\cite{Bauzet:2012kx,Biswas:2014gd,Chen:2011fk,Debussche:2010fk,Debussche:2015aa,Debussche:2016aa,E:2000lq,Hofmanova:2013aa,Kim2003,Karlsen:2015ab,Feng:2008ul,Vallet:2009uq,Vallet:2000ys}
have studied the effect of It\^{o} stochastic forcing on conservation
laws,
\begin{equation}\label{eq:source}
	du +  \nabla\cdot f(u)  \dt=\sigma(u) \, dW(t),
\end{equation}
where $f,\sigma$ are nonlinear functions and $W(t)$ is a (finite or
infinite dimensional) Wiener process. Numerical methods, based on
operator splitting~\cite{Bauzet:2015aa,Holden:1997fk,Karlsen:2016aa}
or finite volume discretizations
\cite{Bauzet:2016ab,Bauzet:2016aa,Dotti:2016aa,Dotti:2016ab,Kroker:2012fk},
have been proposed and successfully analyzed for \eqref{eq:source} and
similar equations.

In another direction, several works 
\cite{Attanasio:2011fj,Flandoli:2010yq,Mohammed:2015aa,Neves:2015aa} 
have explored linear transport equations with low-regularity velocity 
coefficient $b(x)$ and ``transportation noise'',
\begin{equation}\label{eq:transport}
	du +  b(x)\cdot \nabla u  \dt + \nabla u \circ dW(t) =0,
\end{equation}
where $\circ$ refers to the Stratonovich differential (integral).

In this work we are interested in constructing numerical methods for a nonlinear variant 
of \eqref{eq:transport}, namely 
\begin{equation}\label{eq:stratonovich}
	du + \nabla\cdot f(u) \circ dW(t) =0.
\end{equation}
Nonlinear SPDEs like this were suggested and analyzed recently by 
Lions, Perthame, and Souganidis in a series of 
papers \cite{Lions:2013aa,Lions:2012fk,Lions:2014aa}, where  
a pathwise well-posedness theory was developed based on 
entropy/kinetic solutions. Informally, their notion of 
solution is based on writing the kinetic formulation 
of \eqref{eq:stratonovich}:
\[
d \chi + f'(\xi) \cdot \nabla \chi \circ dW(t)= \frac{\partial}{\partial \xi} m,
\]
for a bounded measure $m(t,x,\xi)\ge 0$ and a function $u(t,x)$ (entropy solution) such 
that $\chi=\chi(t,x,\xi):=\chi(\xi, u(t,x))$, where 
$$
\chi(\xi,u)=
\begin{cases}
+1 & \text{if $0\le \xi \le u$},\\
-1 & \text{if $u<\xi <0$},\\
0 & \text{otherwise}.
\end{cases}
$$
The next step is to use a ``transformation'' to remove the noise term. This 
can be achieved by the ``method of characteristics'' since the 
previous equation is linear in $\chi$. The result is that the function
$$
v=v(t,x,\xi):=\chi\left(t,x+f'(\xi)W(t),\xi\right)
$$
satisfies the following kinetic equation without drift term:
\begin{equation}\label{eq:transformed}
	\begin{split}
		\partial_t v  & = \left(\frac{\partial}{\partial \xi } m\right) \left(t,x+f'(\xi) W(t),\xi \right)
		\\ &=\frac{\partial}{\partial \xi} \Bigl(m(t,x+f'(\xi) W(t),\xi)\Bigr) 
		\\ & \qquad - f''(\xi)W(t)\, 
		\left(\frac{\partial}{\partial x} m\right)\left(t,x+f'(\xi) W(t),\xi\right),
	\end{split}
\end{equation}
where the right-hand side is ``nonstandard''. Informally, a weak 
solution to \eqref{eq:transformed} is taken as the definition 
of a pathwise (entropy/kinetic) solution to \eqref{eq:stratonovich}, 
since \eqref{eq:transformed} depends on the noise signal $W$ in a nice way ($dW/dt$ 
is not entering the equation). Various results concerning existence, uniqueness, and 
``continuous dependence on the data'' are found in the works 
\cite{Lions:2013aa,Lions:2014aa} by Lions, Perthame, and Souganidis (more on this below). 
The theory of pathwise solutions has been further developed by 
Gess and Souganidis in \cite{Gess:2014aa,Gess:2015aa}, see
also \cite{Hofmanova:2016aa} and \cite{Bailleul:2015aa,Deya:2016aa} 
for a framework of intrinsic weak solutions of 
PDEs driven by rough signals, without relying on ``transformation formulas'' to 
remove the rough terms.

As alluded to above, we are interested in numerical solutions to conservation laws 
with rough time-dependent flux \eqref{eq:stratonovich}. To the best of our knowledge, 
\cite{Gess:2016aa} is the only work addressing numerical 
aspects of \eqref{eq:stratonovich}. In that work, Gess, 
Perthame, and Souganidis prove convergence of a semi-discrete method based 
on Brenier's transport-collapse algorithm and ``rough path''
characteristics. 

The primary goal of our work is to develop and analyze 
fully discrete and thus computable finite volume 
methods for solving the problem
\begin{equation}\label{eq:sscl}
  \begin{split}
    du + \partial_xf(u) \circ dz = 0 \quad \text{in}\; (0,T] \times \bR, 
    \qquad u(0,\cdot)&=u_0,
  \end{split}
\end{equation}
where $0<T< \infty$ is some fixed final time, $z : [0,T] \to \bR$ is
an $\alpha$-H\"older continuous rough path with $\alpha \in [0,1]$,
$f \in C^2(\bR)$, and $u_0\in (L^1\cap BV)(\R)$.  
The basic numerical methods that we develop for \eqref{eq:sscl}
consist of the following two steps: (1) Approximate the rough path $z$
by a piecewise linear interpolant $z^m$ on a mesh over $[0,T]$
with $m-1$ degrees of freedom. (2) Solve \eqref{eq:sscl} with driving signal $z^m$
using a traditional finite volume method for computing Kru{\v{z}}kov 
entropy solutions, e.g., the Lax-Friedrichs, Godunov, or 
Engquist-Osher scheme \cite{Kroner:1997lq}. 
The second step is justified by the observation that since $z$ is 
uniformly continuous, $z^m$ will be Lipschitz continuous for
any fixed $m \in \bN$, and for any Lipschitz path, 
classical and pathwise entropy solutions 
coincide, cf.~Lemma~\ref{lem:entropyPes} below. 
Several numerical examples are presented 
to illustrate the finite volume methods.

A continuous dependence estimate (cf.~Theorem~\ref{thm:wellPosednessPes} below ) 
can be used to derive a convergence rate for the numerical methods.  
The result is a surprisingly slow rate of convergence: for any
$\alpha \in (0,1]$ with $\|\dot{z}^m\|_\infty > c m^{1-\alpha}$ 
for some $c>0$, the final time numerical error (measured in 
the $L^1$-norm) is bounded by
\begin{equation}\label{eq:costPlain}
\cO{\text{COST}^{-\gamma}}, \quad 
\text{for some } 
\gamma \in \left[\frac{\alpha}{12-8\alpha}, \frac{\alpha}{10-6\alpha} \right].
\end{equation}
Here $\text{COST}$ denotes the computational cost of
solving~\eqref{eq:sscl} with temporal and spatial resolution 
parameters $\Delta t$ and $\Delta x$ (that again are linked by 
the regularity of the rough path through the CFL condition); in other words, if the 
problem is solved numerically over 
the domain $[0,T]\times [a,b]$, then
\[
\text{COST}(\Delta t, \Delta x) = \cO{\frac{T}{\Delta t}\times \frac{b-a}{\Delta x}}.
\]

A conceptually helpful way of seeing why the convergence rate
deteriorates so quickly as $\alpha$ decreases, justified by the CFL
condition applied to the flux $\dot{z}^m f$, is to think of
\eqref{eq:sscl} as being integrated along the path $z^m$ rather than
along time $t$. By that viewpoint the numerical error accumulates
along the full path length of $z^m$ and leads to the replacement of
the factor $T$ by $\abs{z^m}_{BV([0,T])}$ in the standard error estimates
for numerical methods for conservation laws
\cite{HoldenRisebro2015,Kroner:1997lq} (see also Section
\ref{sec:prelim}).

For strictly convex flux functions, the theory of
generalized characteristics and Oleinik estimates can be used to
derive a cancellation property due to rough path
oscillations. We show that for any pathwise entropy 
solution $u(T)$ with piecewise linear
path $z$, there exists a pathwise entropy solution $\tilde u(T)$ 
with a constructively defined ``less oscillatory'' path $\tilde z$ 
which is equivalent to $z$ in the sense that $u(T) = \tilde u(T)$, provided 
$u(0) = \tilde u(0)=u_0$.

The total variation of a rough path enters as a factor in 
the error estimate for the numerical methods (for details, see Section \ref{sec:prelim}). 
In an effort to improve efficiency, we develop a 
variant of the numerical methods which 
solves \eqref{eq:sscl} with $z$ replaced by the
equivalent smoother path $\tilde z$.  The theoretical efficiency gain
by doing so can be significant.  For instance, if the rough path is a 
realization of a standard Wiener process, then we 
show that the final time approximation error is bounded by 
\begin{equation}\label{eq:costOrm}
\cO{\prt{\frac{\log(\mathrm{COST})}{\text{COST}}}^{1/4}}.
\end{equation}
As sample paths of a standard Wiener process
almost surely are $\alpha$-H\"older continuous for any $\alpha<1/2$,
the improvement from~\eqref{eq:costPlain} 
with $\alpha \approx 1/2$ to~\eqref{eq:costOrm} 
is near-optimal in the sense that for conservation laws with 
$z(t)=t$, the optimal ``cost versus accuracy'' rate 
for finite volume methods is $-1/4$.

The cancellation property along with some of its theoretical
consequences are further investigated in the companion work
\cite{Hoel:2017}. Although this article studies
problem~\eqref{eq:sscl} from a numerical perspective and the companion
work~\cite{Hoel:2017} is more focused on theoretical aspects, there
are, in terms of results, some overlaps. Let us therefore point out a
few characteristic features of the approach taken in the companion
work \cite{Hoel:2017} in relation to the one taken herein. Let $u$ be
the pathwise entropy solution to \eqref{eq:sscl}. The article
\cite{Hoel:2017} has the equivalence relation induced by the map $z
\mapsto u(T)$ as its main object of study, and also as a fundamental
tool. Proofs via the mentioned equivalence relation makes continuous
paths the natural objects of ``manipulation''. In this work, a
somewhat different approach is taken. In the case of a piecewise
linear map, the solution map $u_0 \mapsto u(T)$ is factorized as a
product of solution operators, each associated to a straight line
segment of the path, cf.~\eqref{eq:uMSol} and \eqref{eq:solRep2}. What
amounts to ``manipulation'' of paths via equivalence relations in
\cite{Hoel:2017} is replaced by "manipulations" on the product of
solution operators. The equivalent ``less oscillatory'' path $\tilde
z$ is herein associated to an ``irreducible factorization" of solution
operators.

Although it is not a venue we will explore in this work, let us
mention that the equation~\eqref{eq:sscl} may be extended to
stochastic versions which are amenable to various forms of
uncertainty quantification studies. To exemplify,
let $(\Omega, \mathcal{F}, (\mathcal{F})_{t \in [0,T]}, \bP)$
denote a filtered probability space on which the standard Wiener process
is defined, and consider~\eqref{eq:sscl} with
the sampled rough path $z = W$. Then it follows from Theorem~\ref{thm:wellPosednessPes}
that $u \in C([0,T]; L^1(\bR))$, almost surely. For a given
functional $Q: L^1(\bR) \to \bR$, one may for instance seek to approximate
the quantity of interest (QoI)
\[
\E{ Q(u(T,\cdot))} = \int_{\Omega} Q(u(T, \cdot)) \bP(d\omega). 
\]
The numerical methods developed in this paper are directly applicable
to non-intrusive UQ methods for approximating QoIs of this kind, e.g.,
Monte Carlo and Multilevel Monte Carlo methods. We refer to~\cite{giles2015multilevel, haji2016multi,
abgrall2017uncertainty, mishra2012sparse, mishra2012multi, bayer2016rough,
risebro2016multilevel, babuvska2007stochastic,xiu2005high, cliffe2011multilevel}
for recent developments on
numerical methods for uncertainty quantification, and note that
the contributions of this work share similarities with pathwise adaptive
methods for conservation laws and stochastic differential
equations, cf.~\cite{johnson1995adaptive,
  hartmann2003adaptive, szepessy2001adaptive, hall2016computable,
  gaines1997variable, hoel2014implementation, hoel2016construction, giles2016non,
kelly2016adaptive, yaroslavtseva2017non}.

The remaining part of this paper is organized as follows: 
In Section~\ref{sec:prelim}, we collect some preliminary 
material, including a precise definition of pathwise solutions 
as well as relevant existence, uniqueness, and stability results. 
In Section \ref{sec:met1}, we present finite volume 
methods for solving \eqref{eq:sscl} with a general flux function $f$. 
In Section~\ref{sec:orm} we study properties of 
oscillatory cancellations (for convex fluxes $f$) which we use to
develop more efficient numerical methods. 
Section \ref{sec:conclude} wraps the paper up with 
some concluding remarks.

\section{Preliminary material}\label{sec:prelim}

If we assume that the path $z$ is Lipschitz continuous ($\alpha = 1$), then 
\eqref{eq:sscl} reduces to a standard conservation law of the form
\begin{equation}\label{eq:consClass}
    \partial_t u +   \partial_xf(u) \dot{z}= 0
    \quad \text{in} \; (0,T] \times \bR, 
    \qquad u(0,\cdot) = u_0,
\end{equation}
and, assuming for example that $u_0 \in (L^1\cap BV)(\bR)$, 
well-posedness within the framework of Kru\v{z}kov entropy solutions 
is a well-known result \cite{Dafermos:2010fk}. 
Furthermore, entropy solutions are equivalent 
to kinetic solutions \cite{Perthame:2002qy}.

However, if $z$ is merely $\alpha$-H\"older continuous,
for some $\alpha <1$, the well-posedness of entropy/kinetic solutions does not follow 
from standard arguments. This very fact motivates the following notion of solution 
\cite{Gess:2015aa,Lions:2013aa,Lions:2014aa}, which can be 
viewed as a suitable weak formulation of \eqref{eq:transformed}.

  \begin{definition}\label{def:pes}
  Assume $z \in C([0,T])$, $f \in C^2(\bR)$. Then $u\in
  (L^1\cap L^\infty)([0,T] \times \bR)$ is a pathwise entropy solution to
  equation~\eqref{eq:sscl} provided there exists a non-negative,
  bounded measure $m$ on $\bR\times \bR\times [0,T]$ such that for all
  $\rho_0 \in C^\infty_0(\bR^2)$ and $\rho$ given by
  \[
  \rho(x,y,\xi,\eta,t) := \rho_0(y-x +f'(\xi)z(t), \xi-\eta),
  \]
  and all $\phi \in C_0^\infty([0,T])$,
    \begin{multline}
      \int_0^T \partial_t \phi(r) (\rho*\chi)(y,\eta,r)dr
  +\phi(0)(\rho*\chi)(y,\eta,0)- \phi(T)(\rho*\chi)(y,\eta,T)\\ = \int_0^T
  \phi(r) \partial_\xi\rho(x,y,\chi,\eta,r)m(x,\xi,r) dx d\xi dr,
  \end{multline}
  where the ``convolution along characteristics" term 
  $\rho*\chi$ is defined
  by
  \[
  \rho*\chi(y,\eta,r) := 
  \int \rho(x,y,\xi,\eta,r)\chi(x,\xi,r)dx d\xi.
  \]
  \end{definition}
  We note that for a continuous, piecewise Lipschitz path $z(t)$,
  the notions of entropy and pathwise entropy solutions coincide.  
  We recall the following existence, uniqueness, and stability results for
  pathwise entropy solutions~\cite[Theorem 3.2]{Lions:2013aa}.

\begin{theorem}\label{thm:wellPosednessPes}
  Let $u_0\in (L^1\cap L^\infty)(\bR)$ and assume $z \in C([0,T])$ 
  and $f \in C^2(\bR)$.  Then there exists a unique pathwise entropy solution $u
  \in C([0,T]; L^1(\bR))$ which satisfies the following inequality
  for all $p \in [1,\infty]$:
  $$
  \sup_{t\in [0,T]} \|u(t)\|_{L^p(\bR)} \leq \|u_0\|_{L^p(\bR)}.
  $$
  Furthermore, if $u$ and $\bar u$ represent the pathwise entropy
  solution with respective paths $z$ and $\bar z$, then there exists a
  uniform constant $C>0$ such that for all $t\in[0,T]$,
  \begin{equation}\label{eq:stabilityPes}
    \begin{split}
      \|u(t) - \bar u(t)\|_1 &\leq \|u_0 - \bar u_0\|_{1} 
      +C\Big[ \|f'\|(\|u_0\|_{BV} + \|\bar u_0\|_{BV})|(z-\bar z)(t)|\\
      &+\sqrt{\sup_{s\in(0,t)}|(z -
        \bar z)(s)|\|f''\|(\|u_0\|_2+\|\bar u_0\|_2)}\Big].
    \end{split}
  \end{equation}
\end{theorem}

\begin{remark}
Theorem~\ref{thm:wellPosednessPes} is proved in two steps. First, the
result is verified for smooth paths $z \in C^1([0,T])$. Thereafter, 
the result is extended to $z\in C([0,T])$ by utilizing an approximation sequence
$\{z^m\}_m \subset C^1([0,T])$ such that $z^m \to z$ in $C([0,T])$,
and using that solutions of~\eqref{eq:consClass} depend continuously on the driving path. 
\end{remark}

\begin{remark}
According to Theorem~\ref{thm:wellPosednessPes}, 
the pathwise entropy solution of \eqref{eq:consClass} 
depends continuously on the rough 
path $z(t)$ in the supremum norm. It is also possible to prove a 
variant of \eqref{eq:stabilityPes} that includes continuous 
dependence with respect to the flux $f$. Such estimates 
are relevant for some numerical methods \cite{HoldenRisebro2015,Lucier:1986px}. 
Suppose $\bar u$ is the pathwise entropy solution of \eqref{eq:consClass} 
with the ``data" $(u_0, z,f)$ replaced by $(\bar u_0,\bar z,\bar f)$. 
Then the ``continuous dependence" 
estimate \eqref{eq:stabilityPes} is replaced by
\begin{equation}\label{eq:cd1}
	\begin{split}
		\norm{u(t,\cdot)-\bar u(t,\cdot)}_{L^1(\R)}
		& \le \norm{u_0-\bar u_0}_{L^1(\R)}
		\\ & \qquad
		+C \Bigl(\,  
		t\norm{f' z-\bar f' \bar z}_{L^\infty}
		+\sqrt{\norm{f'' z-\bar f'' \bar z}_{L^\infty}}\,
		\Bigr),
	\end{split}
\end{equation}
for some constant $C$ depending on $\norm{(u_0,\bar u_0)}_{L^2\cap BV}$. 
We omit the (lengthy) proof since the arguments are very 
similar to those found in \cite{Lions:2013aa}. Earlier ``deterministic" 
continuous dependence estimates can be found, 
e.g., in \cite{Chen:2005wf,Karlsen:2003za,Lucier:1986px}.

The numerical methods presented later are based on 
replacing the rough path $z(t)$ by a piecewise linear, Lipschitz continuous 
approximation $\bar z(t)$. Suppose for the moment that 
both paths $z(t),\bar z(t)$ are Lipschitz continuous. 
Then, adapting the arguments in 
\cite{Chen:2005wf,Karlsen:2003za,Lucier:1986px}, 
one can prove the following stability estimate:
\begin{equation}\label{eq:cd2}
	\begin{split}
		\norm{u(t,\cdot)-\bar u(t,\cdot)}_{L^1(\R)}
		& \le \norm{u_0-\bar u_0}_{L^1(\R)}
		\\ & \qquad
		+C \Bigl(\,  
		\norm{f'-\bar f'}_{L^\infty}
		+\norm{\dot z- \dot{\bar{z}}}_{L^1((0,t))}\,
		\Bigr),
	\end{split}
\end{equation}
where the constant $C$ depends on the data as follows:
$$
C=C\left(
\norm{\left(u_0,\bar u_0\right)}_{BV}, 
\norm{\left(f',\bar f'\right)}_{L^\infty},
\norm{\left(\dot z, \dot{\bar{z}} \right)}_{L^1((0,t))}
\right).
$$
At variance with \eqref{eq:cd1}, note that the estimate 
\eqref{eq:cd2} does not depend on the second derivative of the flux, 
but it does depend on the derivative of the path (actually the 
total variation of the path). Consequently, there is a trade-off between 
the regularity of the nonlinear flux function and the regularity of the path.
\end{remark}

\section{The first numerical method}\label{sec:met1}

In this section we describe numerical methods for \eqref{eq:sscl}. 
Convergence rates are derived and a few numerical examples are presented 
to illustrate the qualitative behavior of solutions.

Since solutions to \eqref{eq:sscl} depend on the differential of the
driving path $z$, but not on its initial value $z(0)$, we may
without loss of generality restrict ourselves to
driving paths in the function space
\[
C_0([0,T]) \coloneq \{g \in C([0,T]) \mid g(0) = 0\}.
\]
Denote the set of $\alpha$-H\"older continuous functions on $[0,T]$
that are zero-valued at $t=0$ by
\[
C^{0,\alpha}_0([0,T]) \coloneq \left\{z \in C_0([0,T]) \, \Bigg\vert\, \sup_{s\neq t
    \in [0,T]} \frac{|z(s)-z(t)|}{|t-s|^\alpha} < \infty  \right\}, \qquad
\alpha \in (0,1].
\]
The set of Lipschitz continuous functions on $[0,T]$
that are zero-valued at $t=0$ are denoted by 
\[
\Lip{[0,T]} \coloneq C^{0,1}_0([0,T]).
\]
Given a mesh 
\[
0= \tau_0 < \tau_1 < \cdots < \tau_m = T, \qquad m \ge 2,
\]
we introduce the set of functions which are Lipschitz continuous over 
$[0,T]$ and linear over each interval $[\tau_k ,\tau_{k+1}]$, i.e.,
\[
\begin{split}
I_0([0,T]; \{\tau_j\}_{j=0}^m) \! \coloneq \! \!\Big\{g \in \Lip{[0,T]} \, 
\Big\vert\, g|_{[\tau_k, \tau_{k+1}]}(t)\! =&
g(\tau_k) \!+ \! \frac{t-\tau_k}{\Delta \tau} (g(\tau_{k+1}) - g(\tau_k)) \\
& \quad \text{ for all } k=0,1,\ldots,m-1\Big\}.
\end{split}
\]
We also introduce the operator 
$\cI[\cdot](\cdot; \{\tau_j\}_{j=0}^m):C_0([0,T]) \to I_0([0,T]; \{\tau_j\}_{j=0}^m)$
defined by
\begin{equation}\label{eq:interpolationDef}
\begin{split}
\cI[g](t; \{\tau_j\}_{j=0}^m) & \coloneq \1{[\tau_0, \tau_{1}]}(t)
\left(g(\tau_0) +  \frac{t-\tau_0}{\Delta \tau} (g(\tau_{1}) -
  g(\tau_0)) \right) \\
&\quad + \sum_{k=1}^{m-1}  \1{(\tau_k, \tau_{k+1}]}(t)
\left(g(\tau_k) +  \frac{t-\tau_k}{\Delta \tau} (g(\tau_{k+1}) -
  g(\tau_k)) \right),
\end{split}
\end{equation}
for $g \in C_0([0,T])$.
On some occasions we use the shorthand notations
$I^m_0([0,T]) = I_0([0,T]; \{\tau_j\}_{j=0}^m)$ and 
$\cI^m[\cdot] = \cI[\cdot](\cdot; \{\tau_j\}_{j=0}^m)$.

\subsection{Framework for numerical solvers}\label{sec:framework}

We propose the following numerical method for solving~\eqref{eq:sscl}:

\begin{enumerate}

\item[(i)] For an appropriately chosen 
  mesh $\{\tau_j\}_{j=0}^m$, approximate the rough path $z \in C_0([0,T])$ 
  by the piecewise linear interpolant $z^m \coloneq \cI^m[z]$.

\item[(ii)] Solve~\eqref{eq:consClass}
  with the rough path $z$ replaced by the Lipschitz path $z^m$, using a 
  consistent, conservative and monotone finite volume method (for entropy solutions).

\end{enumerate}

With the purpose of studying properties of the entropy solution
of~\eqref{eq:sscl} with path $z^m$, we introduce 
the solution operator $\cS^{\cdot}(\cdot) \cdot$ mapping 
$\bR \times [0,\infty) \times (L^1 \cap BV)(\bR)$ 
into $(L^1 \cap BV)(\bR)$. For $\kappa \in \bR$, $s \ge 0$ and $v \in  (L^1
\cap BV)(\bR)$, $\cS^{\kappa}(s) v$ denotes the solution at time $t=s$ of
\begin{equation}\label{eq:map1}
  \partial_t u +  \kappa \partial_x f(u) = 0
  \quad \text{in}\; (0,\infty)\times\bR,
  \qquad u(0,\cdot)  = v.
\end{equation}
Using the convention that for any $k=0,1,\ldots,m-1$, 
$\dot{z}^{m}_k \coloneq \lim_{t \downarrow \tau_k} \dot{z}^m(t)$
and denoting $\Delta \tau_k = \tau_{k+1} - \tau_k$,
we define 
\begin{equation}\label{eq:uMSol}
  u^m(t, \cdot) = 
\begin{cases}
u_0 & \text{if } t = 0,\\
\cS^{\dot z_0^m} (t) u_0 & \text{if } t \in
(\tau_0, \tau_{1}], \\
\cS^{\dot z_1^m} (t-\tau_1) \cS^{\dot z_0^m} (\Delta \tau_0) u_0 & \text{if } t \in
(\tau_1, \tau_{2}], \\
\vdots \\
\cS^{\dot{z}_{m-1}^m}( t - \tau_{m-1}) \cS^{\dot z_{m-2}^m} (\Delta \tau_{m-2})
\ldots \cS^{\dot z_0^m} (\Delta \tau_0) u_0 & \text{if } t \in
(\tau_{m-1}, T].
\end{cases}
\end{equation}

To justify step (ii) of the above algorithm, let us verify that $u^m$ is
a Kru{\v{z}}kov entropy solution as well as a pathwise entropy solution. 
\begin{lemma}\label{lem:entropyPes}
  Assume that $u_0 \in (L^1 \cap BV)(\bR)$, $f \in C^2(\bR)$,
  $z\in C_0([0,T])$ and $z^m \in \cI^m[z]$. Then the function $u^m$
  defined by equation~\eqref{eq:uMSol} is a Kru\v{z}kov entropy
  solution of~\eqref{eq:consClass} with driving path $z^m$, such that
\begin{equation}\label{eq:um-prop}
  u^m \in C([0,T];L^1(\bR)) \text{ and }
  |u^m(t)|_{\BV{\bbR}} \le  |u_0|_{\BV{\bbR}} \; \forall t \in [0,T]. 
\end{equation}
  Moreover, $u^m$ is also a
  pathwise entropy solution of~\eqref{eq:sscl} with driving path $z^m$.
\end{lemma}
\begin{proof}
It is enough to remark that 
$u^m$ is a Kru{\v{z}}kov entropy solution of \eqref{eq:consClass}.
By \cite{Lions:2013aa} it is then also 
a pathwise entropy solution of \eqref{eq:sscl}, as 
the total variation of $z^m$ is finite. 
Indeed, by construction, on each 
time interval $\left[\tau_k,\tau_{k+1}\right]$, $u^m$ satisfies 
the following local Kru\v{z}kov entropy condition:
\begin{align*}
    & \int_{\tau_k}^{\tau_{k+1}} \int_{\bR} \abs{u^m -c} \phi_t +
    \sign{u-c}\dot{z}^m(f(u^m)-f(c))\phi_x\,  \dx\dt 
    \\
    & + \int_{\bR} \abs{u^m(\tau_{k}+,x) -c } \phi(\tau_{k}, x) -
    \abs{u^m(\tau_{k+1}-,x) -c} \phi(\tau_{k+1}, x) \dx \ge 0,
\end{align*}
for all $c \in \bR$ and test functions $\phi \ge 0$.  
Since $u^m \in C_0([0,T];L^1(\bR))$,
we have that $u(\tau_k+)=u(\tau_k-)$ 
in $L^1(\bbR)$-sense, and summing over $k=0,\ldots,m-1$ gives 
that $u^m$ is a Kru\v{z}kov entropy solution on $[0,T]$. 
See~\cite{HoldenRisebro2015} for verification of \eqref{eq:um-prop}.
\end{proof}

\subsection{Numerical schemes}
Let $U$  denote a finite volume method approximation
of $u^m$ with uniform spatial and temporal mesh parameters $\Delta x$ and 
$\Delta t$ such that
\[
U^n_j \approx \frac{1}{\Delta x} \int_{x_j-\Delta x/2}^{x_j+\Delta x/2} u^m(t_n,y) \dy,
\qquad t_n = n\Delta t, \, x_j = j\Delta x.
\]
Although theoretical results will be stated in more generality, 
we have in the numerical implementations 
restricted ourselves to two numerical methods:
the Lax-Friedrichs scheme
\[
U^{n+1}_j = \frac{U^n_{j+1} + U^n_{j-1} }{2} - \frac{\dot{z}^m(t_n)
  \Delta t}{2\Delta x} \Big(f(U^n_{j+1})
- f(U^n_{j-1})\Big),
\]
and the Engquist--Osher scheme
\[
\begin{split}
U^{n+1}_j &= U^n_j - \frac{ \dot{z}^m(t_n) \Delta t}{2\Delta x}
\Bigg( f(U^n_{j+1}) - f(U^n_{j-1}) \\
& \qquad \qquad \qquad \qquad \qquad 
- \sign{\dot{z}^m(t_n)} \Big(  \int^{U^n_{j+1}}_{U^n_j} |f'(s)| ds - \int^{U^n_{j}}_{U^n_{j-1}}
 |f'(s)| ds \Big) \Bigg).
\end{split}
\]
If volume averages of $u_0$ are computable, both schemes 
are initialized by setting
\begin{equation}\label{eq:u0Approx}
U_j^0:= \frac{1}{\Delta x}\int_{x_j-\Delta x/2}^{x_j+\Delta x/2} u_0(y)dy,
\end{equation}
otherwise each volume average of $u_0$ is approximated using a finite number
of quadrature points evaluating the c\`{a}dl\`{a}g modification of $u_0$ over each volume.

For a consistent treatment of $\dot{z}^m$ at
(possible) discontinuity points $\{\tau_j\}$, we will
always assume that $\{\tau_j\} \subset \{t_n\}$, i.e., the 
interpolation points of $z^m$ constitute a subset of 
the temporal mesh points used in the finite volume scheme. 

We refer to \cite{HoldenRisebro2015,Kroner:1997lq} for 
background material on numerical methods for conservation laws.

\subsection{Resolution balancing and convergence rates}

Assuming that the mesh $\{\tau_j\} \subset [0,T]$ consists of uniformly spaced
points, the numerical solution $U^n_j$ 
defined above has three ``resolution parameters": the rough path 
interpolation step size $\Delta \tau$, and the temporal/spatial mesh sizes 
$\Delta t$ and $\Delta x$ of the finite volume method.
To construct an efficient and stable (convergent) numerical method, 
these parameters must be appropriately balanced. 
In this section, we derive a convergence rate expressed
in terms of the resolution parameters, and determine the optimal
balance for minimizing the error in terms of computational cost.

The next lemma contains our first convergence rate result.
\begin{lemma}
  Let $u\in C([0,T]; L^1(\bR))$ denote the unique pathwise entropy solution of \eqref{eq:sscl}
  for given $u_0 \in (L^1 \cap BV)(\bR)$, $f \in C^2(\bR)$, and $z \in C^{0,\alpha}_0([0,T])$
with $\alpha \in (0,1]$. Assume that $\{\tau_j\}_{j=0}^m\subset [0,T]$ are uniformly
spaced points with step size $\Delta \tau = T/m$, and that for
the numerical solution $U$ defined in Section~\ref{sec:framework}
the following global CFL condition is fulfilled: 
\begin{equation}\label{eq:cfl}
\frac{\Delta t  \|\dot{z}^m\|_\infty \|f'\|_\infty}{\Delta x} \leq C_{\mathrm{CFL}}, 
\end{equation}
where the constant $C_{\mathrm{CFL}}>0$ depends on the scheme used. 
Then
\begin{equation}\label{eq:rate1Met1}
\|u(T) - U(T)\|_1 \leq \|u_0 - U(0) \|_1 + C\prt{\Delta \tau^{\alpha/2} + \Delta \tau^{\alpha-1}
(\sqrt{\Delta t} + \sqrt{\Delta x})},
\end{equation}
where $C>0$ is independent of the resolution parameters.
\end{lemma}
\begin{proof}
  Recall that for the given initial data $u_0$ and flux $f$,
  $u^m$ denotes the pathwise entropy solution
  with driving signal $z^m$, and $U$ denotes the corresponding
  numerical solution with path $z^m$. By the triangle inequality 
and \eqref{eq:stabilityPes},
\begin{equation}
\begin{aligned}
	\|u(T) - U(T)\|_1 &\leq \|u(T)-u^m(T)\|_1 + \|u^m(T) -U(T)\|_1 \\
	& \le C \sqrt{\|z-z^m\|_\infty } +\|u^m(T) - U(T)\|_1.
\end{aligned}\label{eq:LSstar}
\end{equation}
The error can thus be bounded by the sum of the path
approximation error and the finite volume discretization error. 
Since $z \in C^{0,\alpha}_0([0,T])$ and $z^m$ uses $m+1$ uniformly spaced
interpolation points with step size $\Delta \tau$,
$$
\|z-z^m\|_\infty = \cO{\Delta \tau^{\alpha}} \quad \text{and} \quad
\|\dot{z}^m\|_\infty = \cO{\Delta \tau^{\alpha-1}}.
$$
To bound the second term, we repeat the proof of Kuznetsov's
lemma (see~e.g.~\cite{HoldenRisebro2015}, with $f$ 
replaced by $\dot{z}^m f$) to derive that for some constant $C$, depending
on $\|f'\|_\infty$ and $|u_0|_{BV(\bR)}$, the following error estimate holds 
for any consistent, conservative and monotone finite volume approximation:
\begin{equation}\label{eq:globErr1}
	\begin{split}
		\|u^m(T) - U(T)\|_1 &\leq \|u^m_0 - U(0)\|_1 
		+ C\|\dot{z}^m\|_\infty (\sqrt{\Delta t} + \sqrt{\Delta x})
		\\ & \le \|u^m_0 - U(0)\|_1 
		+ C\Delta \tau^{\alpha-1}(\sqrt{\Delta t} + \sqrt{\Delta x}).
\end{split}
\end{equation}
\end{proof}

Having obtained a convergence rate expressed in
terms of the resolution parameters, we next seek to optimally balance
these parameters for the purpose of minimizing computational cost
versus accuracy. Let us first dicuss briefly how the spatial support
of the numerical solutions grows in time.

For any $y \in \bR$, let $\lceil y \rceil$ denote the smallest $n \in
\bZ$ such that $n \ge y$. For two functions $g_1$ and $g_2$ we use the
notation $g_1(m) = \Theta(g_2(m))$ to signify that there exists two
positive constants $C_1$ and $C_2$ such that $C_1 g_1(m)\le g_2(m) \le
C_2 g_1(m)$ for all $m$, in particular $g_1(m)=\Theta(g_2(m))$ implies that
$g_1(m) = \cO{g_2(m)}$ and $g_2(m) = \cO{g_1(m)}$. Let $N= \lceil
T/\Delta t \rceil$ denote the number of timesteps used in the finite
volume method (expressing that number by $N$ rather than $\lceil
T/\Delta t \rceil$ simplifies the transition to non-uniform timesteps
$\Delta t_n$ later on).

Suppose that at some time $t_n \in [0,T-\Delta t]$, we have 
\[
-\infty < a =\inf\{x \in \bR \mid U(t_n,x) \neq 0\}
\quad
\text{and} \quad
\sup\{x \in \bR \mid U(t_n,x) \neq 0\} = b < \infty.
\]
Computing $U(t_{n+1})$ from $U(t_{n})$
by a $k$-stencil numerical scheme yields 
\[
-\infty < a - k\Delta x \le \inf\{x \in \bR \mid U(t_{n+1},x) \neq 0\}
\]
and
\[
\sup\{x \in \bR \mid U(t_{n+1},x) \neq 0\} \le  b + k\Delta x < \infty.
\]
Let $\text{Leb}(\cdot)$ denote the Lebesgue measure on $(\bR,\mathcal{B})$
and for any $g \in L^1(\bR)$ let $\supp{g}$ denote the essential
support of $g$. 
Based on the above observations, we will in the sequel assume that
for any $u_0 \in (L^1 \cap BV)(\bR)$ with $\Leb{ \supp{u_0}} >0$,
$z^m = \cI^m[z]$ and $f \in C^2(\bR)$,
there exists constants $c_1, c_2>0$ such that 
\begin{equation}\label{eq:finiteSupp}
  c_1 \le \text{Leb}\prt{\bigcup_{k\in \{0,1,\ldots,N\} } \text{supp}(U(t_k)) }
  \le c_2 \prt{1 + N \Delta x}
\end{equation}
Note that in the classical setting $\alpha \ge 1$,
the CFL condition~\eqref{eq:cfl} allows for $\Delta t = \tO{\Delta x}$.
This yields $N = \tO{\Delta x^{-1}}$, and assumption~\eqref{eq:finiteSupp}
becomes
\[
\text{Leb}\prt{\bigcup_{k\in \{0,1,\ldots,N\} } \text{supp}(U(t_k)) } = \tO{1},
\]
indicating finite speed of propagation. 
If $\alpha <1$, however, then the CFL condition imposes
the constraint 
$N \Delta x \ge C \|\dot{z}^m\|_\infty$.
So whenever $\lim_{m\to \infty} \|\dot{z}^m\|_\infty=\infty$, a numerical solution
generated by a scheme with artificial diffusion may attain 
infinite speed of propagation in the limit as $m\to \infty$ (although this is not 
an issue with the numerical examples presented later).

\begin{theorem}\label{thm:met1Compl}
Let $u\in C([0,T]; L^1(\bR))$ be the unique pathwise 
entropy solution of \eqref{eq:sscl}, for given 
$u_0 \in (L^1 \cap BV)(\bR)$ with $\Leb{ \supp{u_0}}<\infty$,
$f \in C^2(\bR)$ with $\|f'\|_\infty>0$, and $z \in C^{0,\alpha}_0([0,T])$ 
with $\alpha \in (0,1]$. For any $m \ge 2$, let $\{\tau_j\}_{j=0}^m\subset [0,T]$ 
denote the uniform mesh with step size $\Delta \tau = T/m$ and 
assume the computational cost of generating the interpolant
$z^m = \cI^m[z]$ is $\Theta(m^\beta)$ for some $\beta \ge 1$,
and that there exists an $\check m \ge2$ such that
\[
\|\dot{z}^m\|_{\infty} > m^{(1-\alpha)/3} \qquad \forall m \ge \check m.
\]
Let $U$ denote a numerical solution linked to 
the two-step algorithm in Section~\ref{sec:framework},
satisfying the CFL condition
\begin{equation}\label{eq:cfl2}
  \Delta t = \frac{\Delta \tau}{ \max\prt{ \left \lceil \frac{ \max\prt{\{|\Delta z^m_k|\}_{k=0}^{m-1}} \norm{f'}_\infty}{C_{\mathrm{CFL}} \Delta x}\right \rceil, \, 1 }}.
\end{equation}
Assume that the spatial support of $U([0,T])$ is
covered by an interval $[a_m,b_m] \subset \bR$ that satisfies
\[
c_1 \le b_m-a_m \le c_2 (1+ N\Delta x),
\]
for some $c_1,c_2>0$, cf.~\eqref{eq:finiteSupp}.
  
Then the optimal balance of resolution parameters
for minimizing computational cost versus accuracy 
is
\[
\Delta x =\tO{\frac{\Delta \tau^\alpha}{\|\dot{z}^m\|_\infty^2 }} \quad \text{and} \quad
\Delta t = \tO{\frac{\Delta \tau^\alpha}{\|\dot{z}^m\|_\infty^3  }},
\]
and
\begin{equation}\label{eq:accuracy1}
\|u(T) - U(T)\|_1 = \cO{m^{-\alpha/2}}
\end{equation}
is achieved at the computational cost
\begin{equation}\label{eq:cost1}
  \hat c_1 \prt{ \|\dot{z}^m\|_{\infty}^{5} m^{2\alpha}  + m^{\beta}} 
  \le \mathrm{Cost}(U) \le
  \hat c_2 \prt{ \|\dot{z}^m\|_{\infty}^{6} m^{2\alpha}  + m^{\beta}},
\end{equation}
for some $\hat c_1, \hat c_2>0$.
\end{theorem}

\begin{proof}
Assume that $m \ge \check m$. Then the CFL condition~\eqref{eq:cfl2}
imposes the following constraint on
the timestep:
\[
\Delta t = \tO{\frac{\Delta x}{\|\dot{z}^m\|_\infty}}.
\]
Since $u^m(0) = u_0\in (L^1 \cap BV)(\bR)$, the approximation of the
initial data~\eqref{eq:u0Approx} yields that
\[
\|u^m_0 - U(0)\|_1 = \cO{\Delta x},
\]
and by~\eqref{eq:LSstar},
\[
\|u(T) - U(T)\|_1 
= \cO{\Delta \tau^{\alpha/2} 
+  \|\dot{z}^m\|_\infty \sqrt{\Delta x}}.
\]
The optimal balance of resolution parameters for minimizing the computational
cost versus accuracy is achieved through equilibration of error contributions:
\[
\Delta x = \tO{ \frac{\Delta \tau^{\alpha}}{ \|\dot{z}^m\|_\infty^{2}}}.
\]
Since $m \ge \check m$, 
\[
\frac{\Delta x}{\|\dot{z}^m\|_\infty} = \tO{\frac{\Delta \tau^\alpha}{\|\dot{z}^m\|_\infty^3}}
= \cO{\Delta \tau},
\]
and thus 
\[
\Delta t = \tO{\min\prt{\frac{\Delta x}{\|\dot{z}^m\|_\infty}, \Delta \tau}}
= \tO{\frac{\Delta \tau^\alpha}{\|\dot{z}^m\|_\infty^3}}.
\]
The computational cost of the numerical solution 
$U(T)$ is the sum of $\tO{m^\beta}$ for generating the
piecewise linear interpolant $z^m$, and
\[
\begin{split}
  \tO{ \frac{T}{\Delta t} \times \frac{b_m-a_m}{\Delta x}}  
\end{split}
\]
for solving $U$ over $[0,T]\times [a_m,b_m]$. 
\end{proof}

\begin{remark}
Theorem~\ref{thm:met1Compl} 
provides a surprisingly slow convergence rate. 
For instance, if $z \in C^{0,1/2}_0([0,T])$
and $\|\dot{z}^m\|_{\infty} = \tO{m^{1/2}}$,
then~\eqref{eq:accuracy1} implies that in order to achieve 
the accuracy $\|u(T) - U(T) \|_{L^1} = \cO{\epsilon}$,
one needs $m \ge c \epsilon^{-4}$ for some $c>0$.
By~\eqref{eq:cost1}, this results in the astounding 
$\tO{\epsilon^{-14} + \epsilon^{-4\beta}}$ lower bound
on the computational cost. 
In some numerical experiments, however, we observe a
better convergence rate than predicted by~\eqref{eq:accuracy1},
see Example~\ref{ex:fBM} in Section~\ref{subsec:met1NumEx}.
\end{remark}

\begin{remark}
In Theorem~\ref{thm:met1Compl}, we assume
the computational cost of 
generating/sampling the piecewise linear interpolant
$z^m$ is $\tO{m^\beta}$ for some $\beta \ge 1$. 
If $z$ is a realization of a Wiener process, for instance,
then $\beta=1$, but to cover the more
general H\"older continuous stochastic 
processes, we allow for $\beta \ge 1$.
\end{remark}

We next consider the use of an adaptive mesh 
$\{t_n\}_{n=0}^N \supset \{\tau_k\}_{k=0}^m$
that have uniform timesteps over each interpolation interval $[\tau_k,\tau_{k+1}]$
for $k=0,1,\ldots,m-1$.
That is, $t_0 = 0$ and given $t_n \in [\tau_k,\tau_{k+1})$, the next mesh point
  is set to
  \begin{equation}\label{eq:cflLoc}
    \begin{cases}
      t_{n+1} = t_n + \frac{\Delta \tau_k}{\overline n(k)}, \qquad  \text{where}\\
      \\
\overline n(k) \coloneq \max\seq{ \left \lceil \frac{|\Delta z^m_k|\norm{f'}_\infty}{C_{\mathrm{CFL}} \Delta x} \right \rceil, \, 1 }
&  \text{for } k=0,1,2, \ldots,m-1.
\end{cases}
\end{equation}
Here, the constant $C_{\mathrm{CFL}}>0$ depends on the scheme used.
We refer to~\eqref{eq:cflLoc} as the local CFL condition. 
The next theorem shows that adaptive timesteps can improve the efficiency 
of the numerical methods.

\begin{theorem}\label{thm:met1Compl2}
Let $u\in C([0,T]; L^1(\bR))$ denote the unique pathwise entropy
solution of \eqref{eq:sscl} for given $u_0 \in (L^1 \cap BV)(\bR)$
with $\Leb{\supp{u_0}}>0$, $f\in C^2(\bR)$
with $\|f'\|_\infty >0$, and $z \in C^{0,\alpha}_0([0,T])$ with $\alpha \in
(0,1]$.
For any $m \ge 2$, let $\{\tau_j\}_{j=0}^m\subset [0,T]$ denote 
the uniform mesh with step size $\Delta \tau = T/m$ and 
assume the computational cost of generating the interpolant
$z^m = \cI^m[z]$ is $\Theta(m^\beta)$ for some $\beta \ge 1$,
and that there exists an $\check m\ge 2$ and $c>0$ such that
\begin{equation}\label{eq:zBVAssmp}
\abs{z^m}_{\BV{[0,T]}} > c m^{(1-\alpha)/3} \qquad \forall m \ge \check m.
\end{equation}
Furthermore, let $U$ denote
a numerical solution of the method in Section~\ref{sec:framework}
satisfying the local CFL condition~\eqref{eq:cflLoc} and
assume that the spatial support of $U([0,T])$ is covered by
an interval $[a_m,b_m] \subset \bR$
that satisfies
\[
c_1 \le b_m-a_m \le c_2 (1+ N\Delta x),
\]
for some $c_1,c_2>0$, cf.~\eqref{eq:finiteSupp}.
Then, the optimal balance of the resolution parameters 
for minimizing computational cost versus accuracy is
\[
\Delta x = \tO{\frac{\Delta \tau^\alpha}{\abs{z^m}_{BV([0,T])}^2}}
\quad \text{and} \quad  
N = \sum_{k=0}^{m-1} \overline n(k) = \tO{ \frac{\abs{z^m}_{BV([0,T])}^3}{\Delta \tau^{\alpha}}},
\]
and
\[
\|u(T) - U(T)\|_1 = \cO{m^{-\alpha/2}},
\]
is achieved at the computational cost 
\[
\hat c_1 \prt{\abs{z^m}_{BV([0,T])}^{5}m^{2\alpha}  +m^\beta}
\le \mathrm{Cost}(U)
\le \hat c_2 \prt{\abs{z^m}_{BV([0,T])}^{6}m^{2\alpha}  +m^\beta},
\]
for some $\hat c_1, \hat c_2 >0$.
\end{theorem}

\begin{remark}
  If $\abs{z^m}_{BV([0,T])} =\tO{ \|\dot{z}^m\|_\infty}$,
  then the computational cost results in Theorems~\ref{thm:met1Compl} 
  and~\ref{thm:met1Compl2} are, up to constants, equivalent.
\end{remark}

\begin{proof}
  The local CFL condition~\eqref{eq:cflLoc} implies that
  all timesteps $\Delta t_n$ belonging to the same interpolation
  interval $[\tau_j, \tau_{j+1}]$ are of the equal size and 
  \begin{equation}\label{eq:dtNBound}
  \Delta t_n \le C_{\mathrm{CFL}} 
  \frac{\Delta \tau \Delta x}{|\Delta z_j^m| \|f'\|_\infty},
  \end{equation}
  for any $j \in \{0,1,\ldots,m-1\}$ and all $n \in \{0,1,\ldots,N\}$ such that
  $t_n \in [\tau_j, \tau_{j+1})$.
  By~\eqref{eq:dtNBound} and the proof of Kuznetsov's lemma
  (see~e.g.~\cite{HoldenRisebro2015}, with the flux $f$ replaced by
  $\dot{z}^m_j f$), the numerical error from one interpolation
  interval can be bounded by
\begin{align*}
    \|u^m(\tau_{j+1}) - U(\tau_{j+1})\|_1 &\leq \|u^m(\tau_j) -
    U(\tau_j)\|_1 \\
    &\qquad  + C|\dot{z}^m_j| \!\!\!\!\sum_{n \in \{ 0\le k \le N \mid t_k
      \in [\tau_j,\tau_{j+1}) \}}
      \!\!\!\!\!\Delta t_n (\sqrt{\Delta t_n} + \sqrt{\Delta x})\\
    & \leq \|u^m(\tau_j) - U(\tau_j)\|_1 
    + C\prt{|\Delta z^m_j|+\Delta \tau} \sqrt{\Delta x},
\end{align*}
for some $C>0$ that depends on $\|f'\|_\infty$, $|u_0|_{BV(\bR)}$.
Consequently, the error over $[0,T]$ is bounded by
\[
\|u^m(T) - U(T)\|_1 \leq \|u_0 - U(0)\|_1 + C(\abs{z^m}_{BV([0,T])}+T)\sqrt{\Delta x}. 
\]
By a similar argument as in the proof of the preceding theorem, 
we conclude that
\[
\|u(T) - U(T)\|_1 = \cO{\Delta \tau^{\alpha/2} + \abs{z^m}_{BV([0,T])} \sqrt{\Delta x}},
\]
where, by~\eqref{eq:zBVAssmp}, we have used that $\abs{z^m}_{BV([0,T])}+T = \cO{\abs{z^m}_{BV([0,T])}}$. 
The error contribution of the resolution parameters are balanced by
\[
\Delta x = \tO{\frac{\Delta \tau^\alpha}{\abs{z^m}_{BV([0,T])}^2}}.
\]
Assume that $m\ge \check m$. By~\eqref{eq:cflLoc}, 
\begin{equation}\label{eq:lineNBound}
\frac{\|f'\|_\infty}{C_{\mathrm{CFL}}} \frac{|\Delta z^m_k|}{\Delta x} \le  \overline n(k) \le \frac{\|f'\|_\infty}{C_{\mathrm{CFL}}} \frac{|\Delta z^m_k|}{\Delta x} + 1 \qquad \forall k \in \{0,1,\ldots,m-1\},
\end{equation}
and~\eqref{eq:zBVAssmp} implies that
\begin{equation}\label{eq:nBound2}
\begin{split}
N = \sum_{k=0}^{m-1} \overline n(k) &\le  \sum_{k=0}^{m-1} \prt{ \frac{\|f'\|_\infty}{C_{\mathrm{CFL}}} \frac{|\Delta z^m_k|}{\Delta x} + 1}\\
&\le \frac{\|f'\|_\infty}{C_{\mathrm{CFL}}}\frac{\abs{z^m}_{BV([0,T])}^3}{\Delta \tau^{\alpha}} + m\\
&= \cO{\frac{\abs{z^m}_{BV([0,T])}^3}{\Delta \tau^{\alpha}}}.
\end{split}
\end{equation}
From~\eqref{eq:lineNBound} and~\eqref{eq:nBound2} we conclude that
\[
N = \tO{\frac{\abs{z^m}_{BV([0,T])}^3}{\Delta \tau^{\alpha}}}.
\]
The computational cost of $U(T)$ is the sum of $\cO{m^\beta}$, for generating the
piecewise linear interpolant $z^m$, and
\[
\tO{ N \times \frac{b_m - a_m}{\Delta x}} 
\]
for solving $U$ over $[0,T]\times [a_m, b_m]$.
\end{proof}

\subsection{Numerical examples}\label{subsec:met1NumEx}
To simplify the spatial discretization in our numerical tests, we
consider the following version of~\eqref{eq:sscl} with 
periodic boundary conditions:
\begin{equation}\label{eq:ssclPer}
  \begin{split}
    & du + \partial_xf(u) \circ dz = 0 \quad \text{ in}\; (0,T] \times \bT, \\
    & u(0,\cdot) =u_0 \in (L^1\cap BV)(\bT).
    \end{split}
\end{equation}
Well-posedness and stability results for~\eqref{eq:ssclPer}
can be derived by a simple extension of~\cite{Lions:2013aa}.
Lemma~\ref{lem:entropyPes}, and the numerical framework of Section~\ref{sec:framework} extend trivially to the 
periodic setting using the solution operator
$$
\cS_\bT^{\cdot}(\cdot) \cdot: 
\bR \times [0,\infty) \times (L^1 \cap BV)(\bT) 
\to (L^1 \cap BV)(\bT),
$$
 where, for $c\in \bR$, $s \ge 0$ and 
$v \in (L^1 \cap BV)(\bT)$, $\cS_\bT^{c}(s)v$ denotes 
the solution at time $t=s$ of the conservation law
$$
 \partial_t u +  c \partial_x f(u) = 0 
 \quad \text{in}\; (0,\infty)\times \bT,
 \qquad u(0,\cdot) = v.
$$

All problems are solved with the adaptive 
timestep method of Theorem~\ref{thm:met1Compl2} 
with free/varying resolution parameter 
$m \ge 2$, linked parameters
\begin{equation}\label{eq:resolution}
\Delta \tau = m^{-1}, \qquad \Delta x 
= \frac{1}{ \left\lceil m^{\alpha}\abs{z^m}_{BV([0,T])}^2 \right\rceil},
\end{equation}
and $\Delta t_n$ determined by~\eqref{eq:cflLoc} 
with $C_{\mathrm{CFL}} =1$, for both the 
Lax--Friedrichs and the Engquist--Osher scheme.

\begin{example}\label{ex:zigzag}
We consider~\eqref{eq:ssclPer} with $T=1$, $f(u) = u^2/2$,
$u_0(x) = \1{[1/4,3/4]}(x)$, and the zigzag path $z \in C^{0,1}_0([0,1])$ 
generated by piecewise linear interpolation
of the points $\{(t_i,z(t_i))\}_{i=0}^8$ with $t_i =i/8$ and
$$
z(t_i) =
\begin{cases} 0 & i \text{ odd},\\
 \frac{(-1)^{i/2}}{4}& i \text{ even}.
\end{cases}
$$
Thanks to Lemma~\ref{lem:entropyPes}, the 
solution can be represented as
\begin{equation}\label{eq:uSolEx1}
  u(t, \cdot) = 
\begin{cases}
\1{[3/8,5/8]}(x) & \text{if } t = 0,\\
\cS^{2}_\bT (t) u_0 & \text{if } t \in (t_0, t_{1}], \\
\cS^{-2}_\bT (t-t_1) \cS_\bT^{2} (1/8) u_0 
& \text{if } t \in (t_1, t_{2}], \\
\vdots \\
\cS_\bT^{2}( t - t_{7}) \cS_\bT^{-2} (1/8)
\ldots \cS_\bT^{2} (1/8) u_0 & \text{if } t \in
(t_7, 1].
\end{cases}
\end{equation}
Moreover, using the method of characteristics
and the auxiliary function $\psi: A \subset\bT^4 \to \bR$ defined 
on the domain $A =\{ (a,b,c,d) \in \bT^4 \mid a \le b \le c \le d
\text{ and } a<d \}$ by 
\[
\psi(a,b,c,d) \coloneq
\begin{cases}
0 &  0\le x \le a,\\
(x-a)/(b-a) & a< x <b,\\
1 & b\le x \le c,\\
 (d-x)/(d-c) & c< x <d,\\
0 &  d \le x <1,
\end{cases}
\]
we obtain the exact solution
\begin{equation}\label{eq:exactSolEx1}
u(t,x) = \begin{cases}
\psi(3/8,3/8+2t,5/8+t,5/8+t) &  \!\! \! t \in [0, 1/8], \\
\psi(3/8, 5/8-2(t-1/8),6/8-2(t-1/8),6/8) & \!\! \! t \in (1/8, 1/4],\\
\psi(3/8-(t-1/4), 3/8-(t-1/4) ,4/8-2(t-1/4), 6/8) & \!\! \! t \in (1/4, 3/8],\\
\psi(2/8, 2/8+ 2(t-3/8) , 2/8+2(t-3/8),6/8 ) & \!\! \! t \in (3/8, 1/2],\\
\psi(2/8, 4/8+ 2(t-1/2) , 4/8+ 2(t-1/2),6/8 ) & \!\! \! t \in (1/2, 5/8],\\
\psi(2/8, 6/8-2(t-5/8) , 6/8-2(t-5/8),6/8) &  \!\! \! t \in (5/8, 3/4],\\
\psi(2/8, 4/8 -2(t-3/4) , 4/8- 2(t-3/4),6/8 ) & \!\! \! t \in (3/4, 7/8],\\
\psi(2/8, 2/8+2(t-7/8) , 2/8 + 2(t-7/8),6/8 ) &  \!\! \! t \in (7/8, 1].
\end{cases}
\end{equation}
Figure~\ref{fig:animBurgPos} shows snapshots of the exact solution of
$u$ for the above problem, and corresponding numerical solutions $U$
computed with the
Lax--Friedrichs and the Engquist--Osher scheme.  The free resolution
parameter is set to $m=2^3$ in first time series and $m=2^6$ in the second one.
Since $\alpha =1$, $\|f'\|_\infty =1$ and
$\abs{z^m}_{\BV{[0,1]}} = 2$, equations~\eqref{eq:cflLoc} and~\eqref{eq:resolution}
yield $\Delta x = 2^{-5}$, $N= 2^6$
when $m = 2^3$ and $\Delta x = 2^{-8}$, $N=2^9$ when $m= 2^6$.
As is to be expected from Theorem~\ref{thm:met1Compl2},
the numerical solutions converge towards the exact solution as $m$ increases,
and the Engquist--Osher approximations converges faster 
than the Lax--Friedrichs approximations.

Observe further from~\eqref{eq:exactSolEx1} and
Figure~\ref{fig:animBurgPos} that for any $s, t \in
[t_3, 1]$ such that $z(s) = z(t)$, it holds that $u(s
) = u(t)$. In the next section we will explain this property by
showing that certain ``oscillating cancellations" in $z$ lead to
corresponding cancellations in the solution $u$.

\begin{figure}[H]
\includegraphics[width=0.81\textwidth]{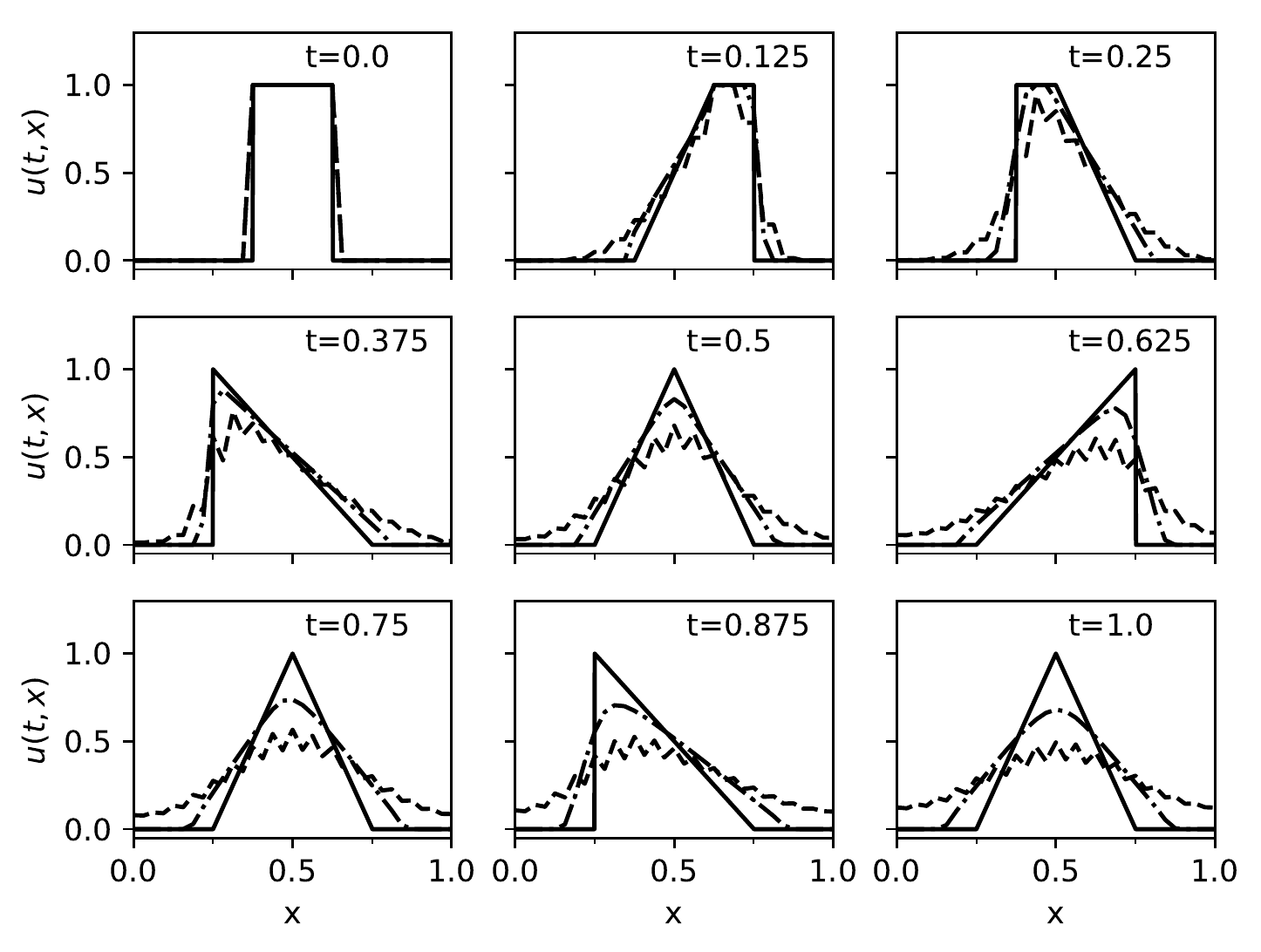}\\
\includegraphics[width=0.81\textwidth]{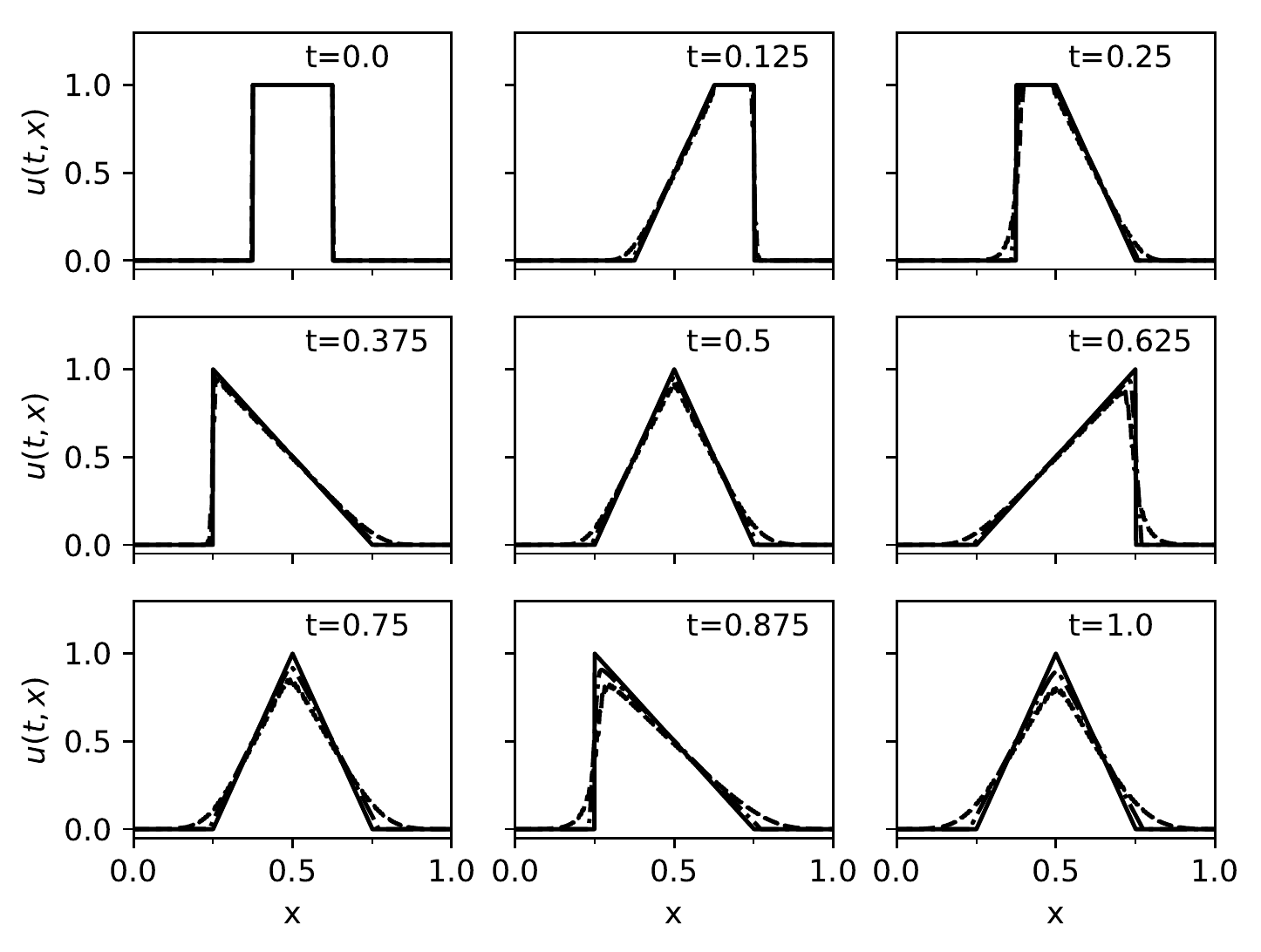}\\
\includegraphics[width=0.65\textwidth]{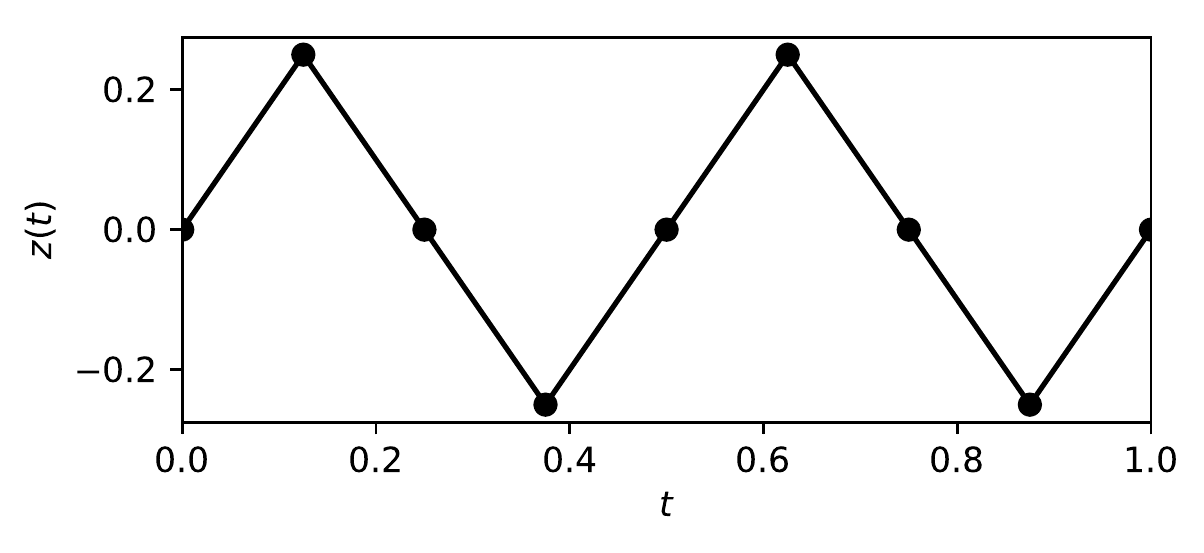}
\caption{Top and middle: Time series snapshots of 
the exact (solid line), the Lax--Friedrichs (dashed line)
and the Engquist--Osher (dash-dotted line)  solutions of
Example~\ref{ex:zigzag} with flux $f(u) = u^2/2$ 
and free resolution parameter $m=2^3$ (top $3\times3$ subfigure) 
and $m=2^6$ (middle $3\times3$ subfigure). 
Bottom: The zigzag driving path $z$. 
The dotted points correspond to the value of $z$ at the
respective time series snapshots.}
\label{fig:animBurgPos}
\end{figure}

Figure~\ref{fig:animCubicPos} shows snapshots of
numerical solutions of the above problem with the only difference being that
the flux function here is $f(u) =u^3/3$.
The free resolution parameter is set to $m=2^3$ for the
first time series and $m=2^6$ for the second one,
and an approximate reference solution is computed at resolution  
$m=2^{12}$ using the Engquist--Osher scheme. 
We observe that the numerical solutions
converge towards the reference solution, and a similar 
cancellation property as that for $f(u)=u^2/2$ 
seems to hold at the snapshot times $t_i, t_j\ge t_3$ displayed here as well.

\begin{figure}[H]
\includegraphics[width=0.81\textwidth]{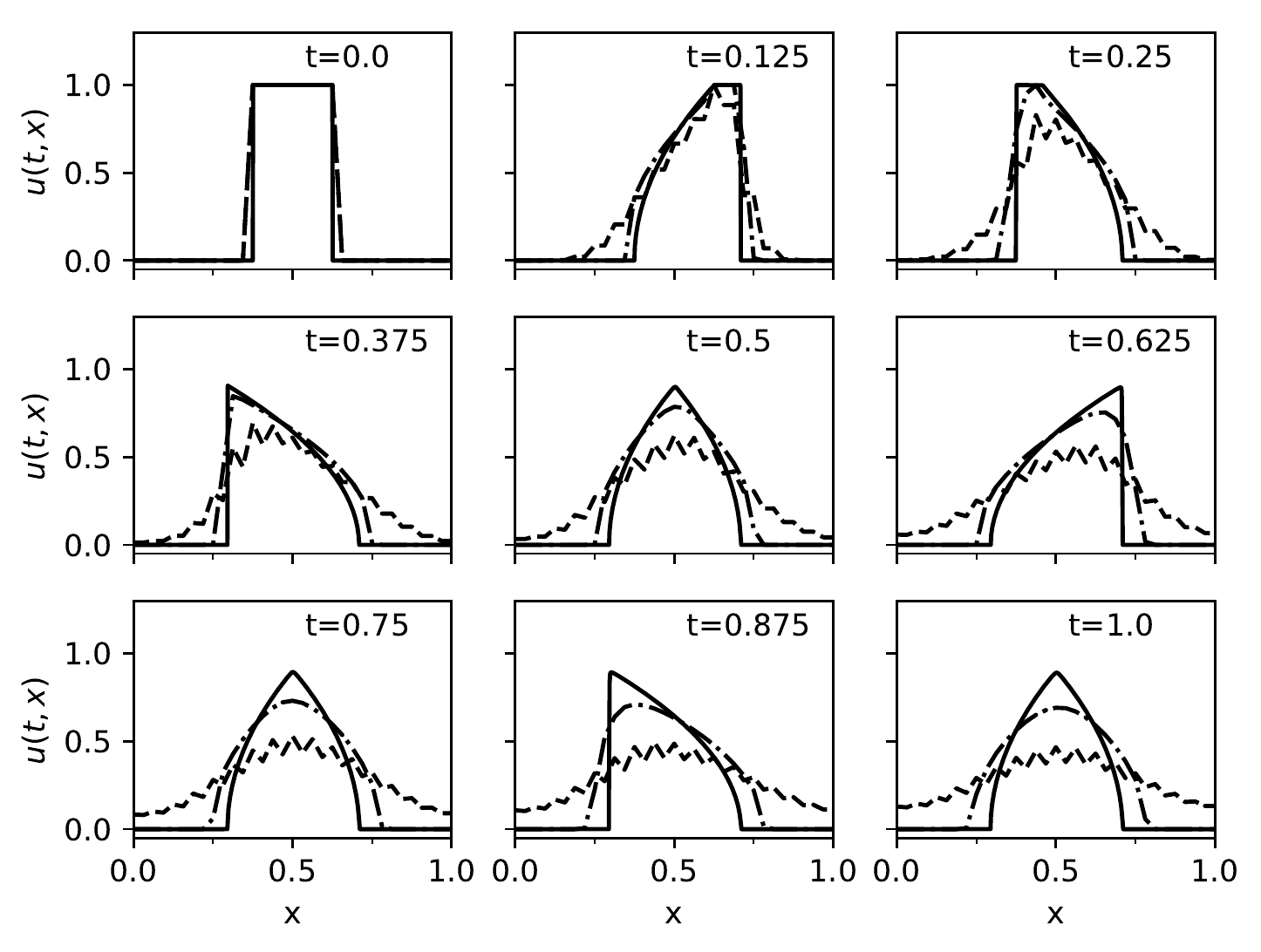}
\includegraphics[width=0.81\textwidth]{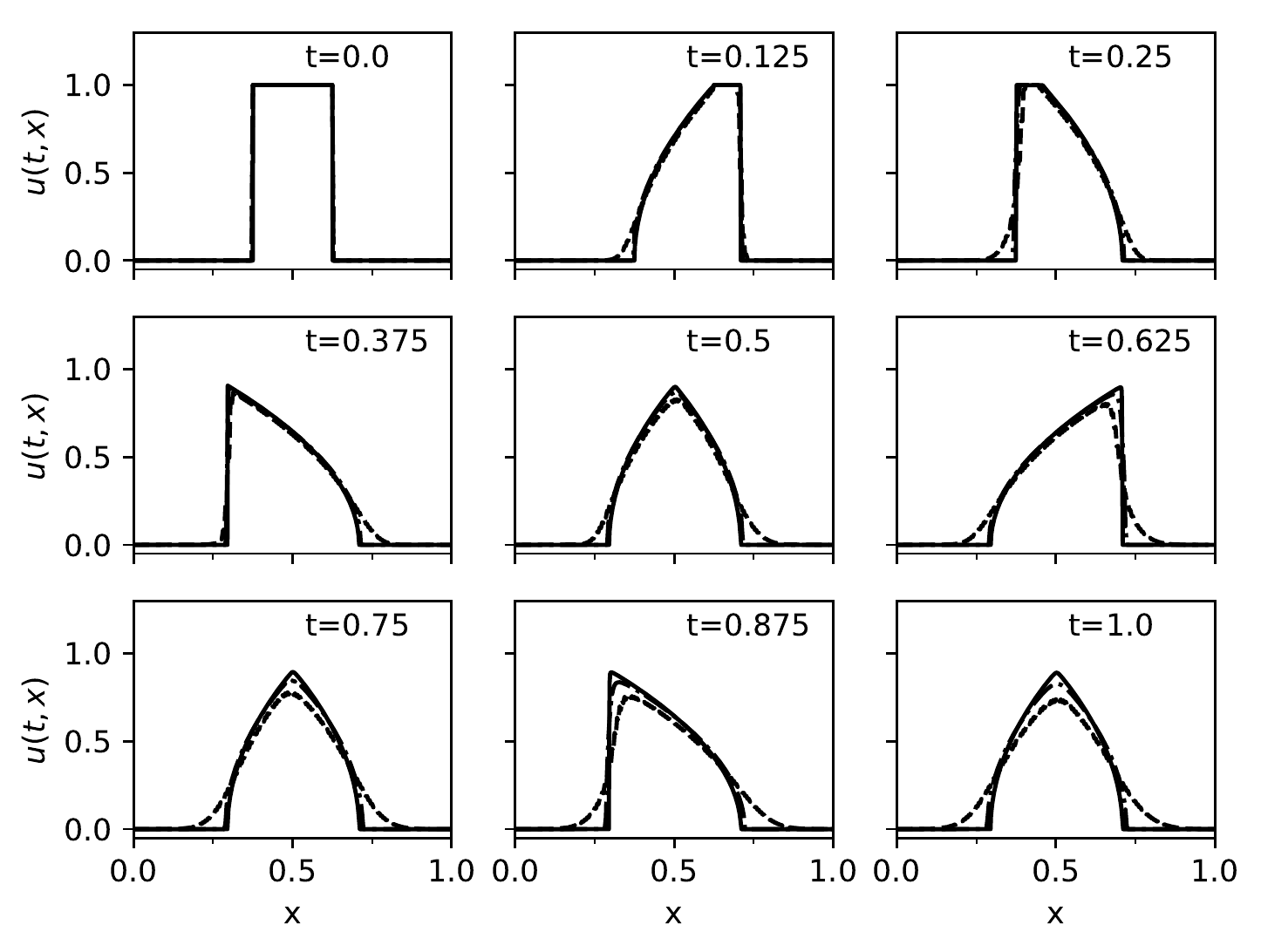}\\
\caption{Example~\ref{ex:zigzag} with flux $f(u) = u^3/3$. 
Approximated reference solution (solid), 
Lax--Friedrichs (dashed) and Engquist--Osher (dash-dotted).
The free resolution parameter is $m= 2^3$ in 
the top $3\times3$ subfigure, and $m=2^6$ in the bottom one.
}
\label{fig:animCubicPos}
\end{figure}

\end{example}

\begin{example}\label{ex:fBM}

We study the problem~\eqref{eq:ssclPer} with $T=1$, initial function 
$u_0(x) = \1{[3/8,5/8]}(x)$, and flux $f(u) = u^2/2$.
As driving path $z \in C^{0,\alpha}_0([0,1])$ we consider
realizations of a fractional Brownian motion (fBM) with Hurst
index\footnote{To be precise, a sample path of an fBM with Hurst index 
$\alpha \in (0,1)$ is almost surely $(\alpha-\delta)$-H\"older 
continuous for all $\delta>0$, but it is almost surely 
not $\alpha$-H\"older continuous.} $\alpha =
1/4, 1/2$ and $3/4$.  We refer to~\cite{Kroese15} for
details on fBM and the circulant embedding algorithm, which
we use here to generate realizations on a uniformly 
spaced mesh $\{\tau_k\}_{k=0}^m$. 
The cost of generating $z^m$ with this method 
is $\cO{m\log(m)}$.

\begin{figure}[H]
\includegraphics[width= 0.75\textwidth]{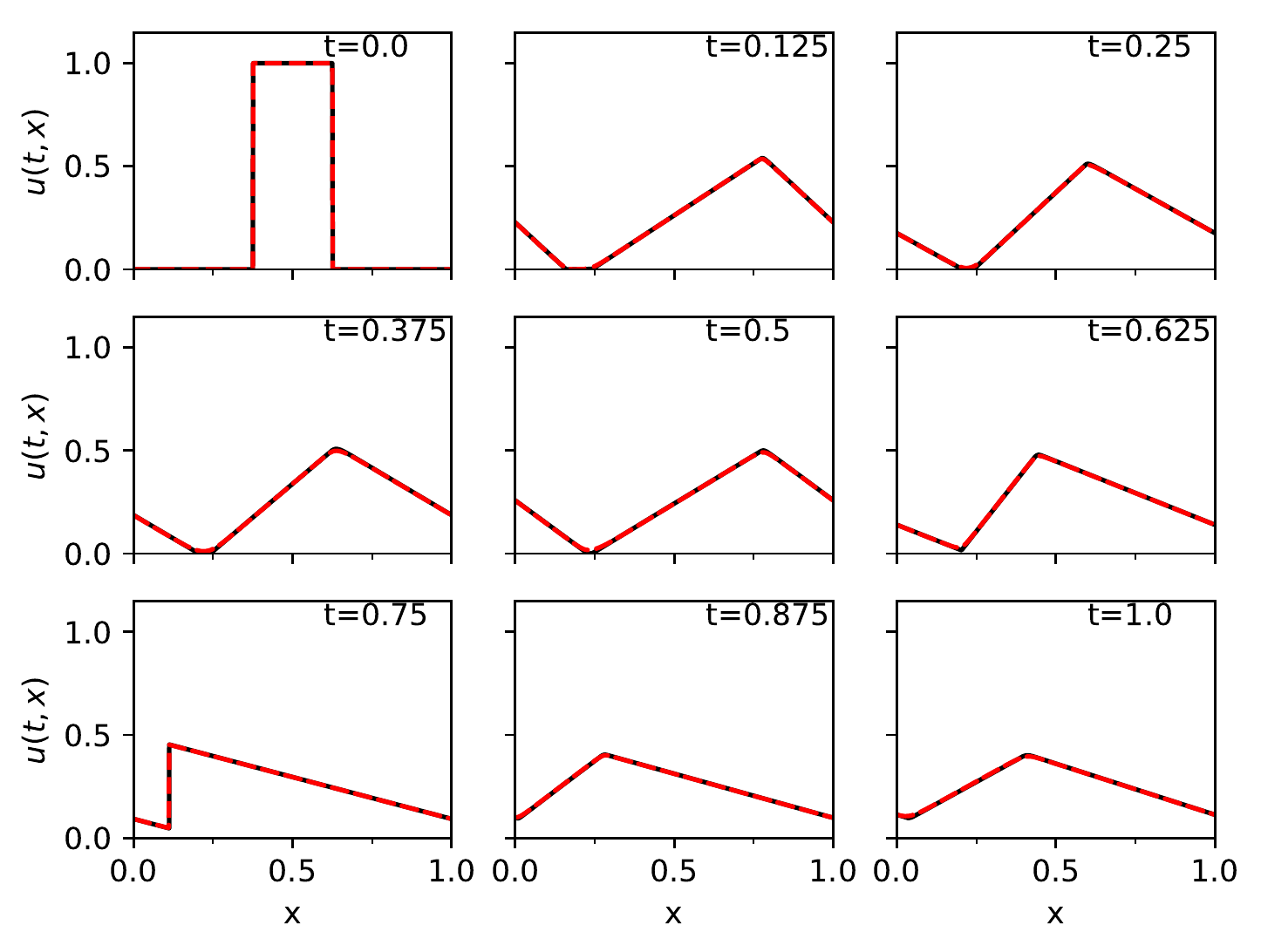}\\
\includegraphics[width= 0.67\textwidth]{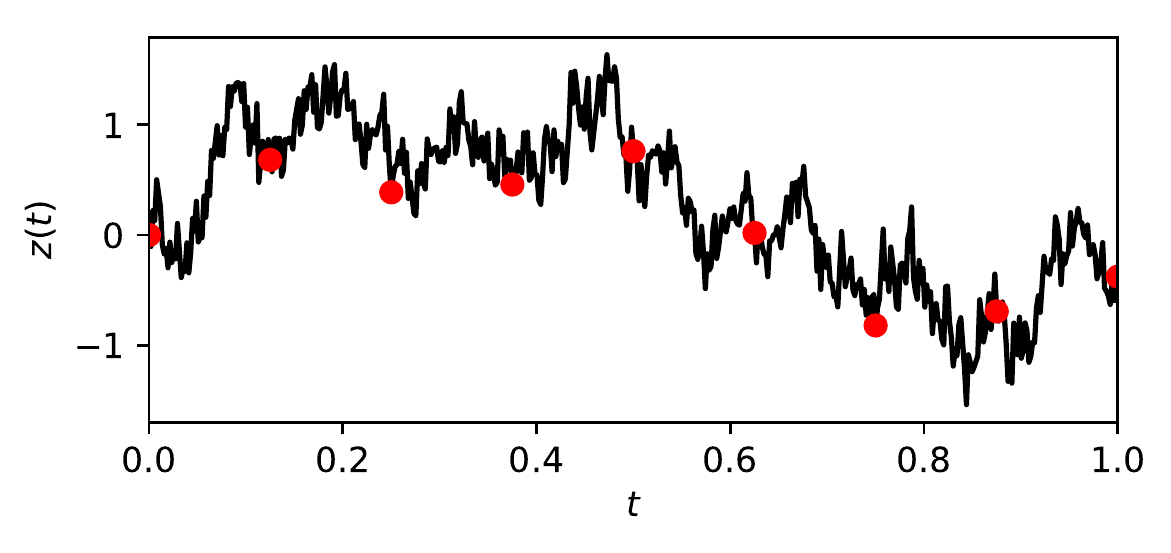}
\caption{Top:  Time series of 
numerical solutions with $\alpha = 0.25$ for Example~\ref{ex:fBM}.
Lax--Friedrichs (red dashed line) and Engquist--Osher (black line). 
Bottom: The rough path $z$. Red points correspond to the 
value of $z$ at the respective time series snapshots.}
\label{fig:ex2TSer25}
\end{figure}

Figures~\ref{fig:ex2TSer25}, ~\ref{fig:ex2TSer50},
and~\ref{fig:ex2TSer75} show time series of ``adaptive timestep"
numerical solutions for the respective Hurst indices $\alpha = 0.25,
0.5$ and $0.75$. The free resolution parameter is set to $m = 2^{8}$.
In Figure~\ref{fig:ex2FinalTime} we compare the final time numerical
solutions $U(1)$ computed with the respective resolutions $m=2^6$ and
$m=2^{8}$, and with an approximate reference solution computed with
resolution $m=2^{10}$ using the more accurate numerical method
developed in Section~\ref{sec:framework2}.  A link to the other
resolution parameters is obtained through~\eqref{eq:resolution} and
the following property for fBMs: $\E{\abs{z^m}_{\BV{[0,1]}}} =
\tO{m^{1-\alpha}}$. At resolution $m=2^{8}$, for instance, a typical
realization of an fBM sample path yields $\Delta x = 2^{-10}$ and
$N=2^{12}$ for $\alpha =3/4$; $\Delta x = 2^{-12}$ and $N=2^{16}$ for
$\alpha =1/2$; and $\Delta x = 2^{-14}$ and $N=2^{20}$ for $\alpha
=1/4$.  We observe that the Engquist--Osher scheme introduces less
artificial diffusion and therefore produces more accurate
approximations than the the Lax--Friedrichs scheme.  (Note that
solutions for different values of $\alpha$ are not directly comparable,
not least since they are generated from independent fBM sample paths.)

\begin{figure}[h]
\includegraphics[width= 0.75\textwidth]{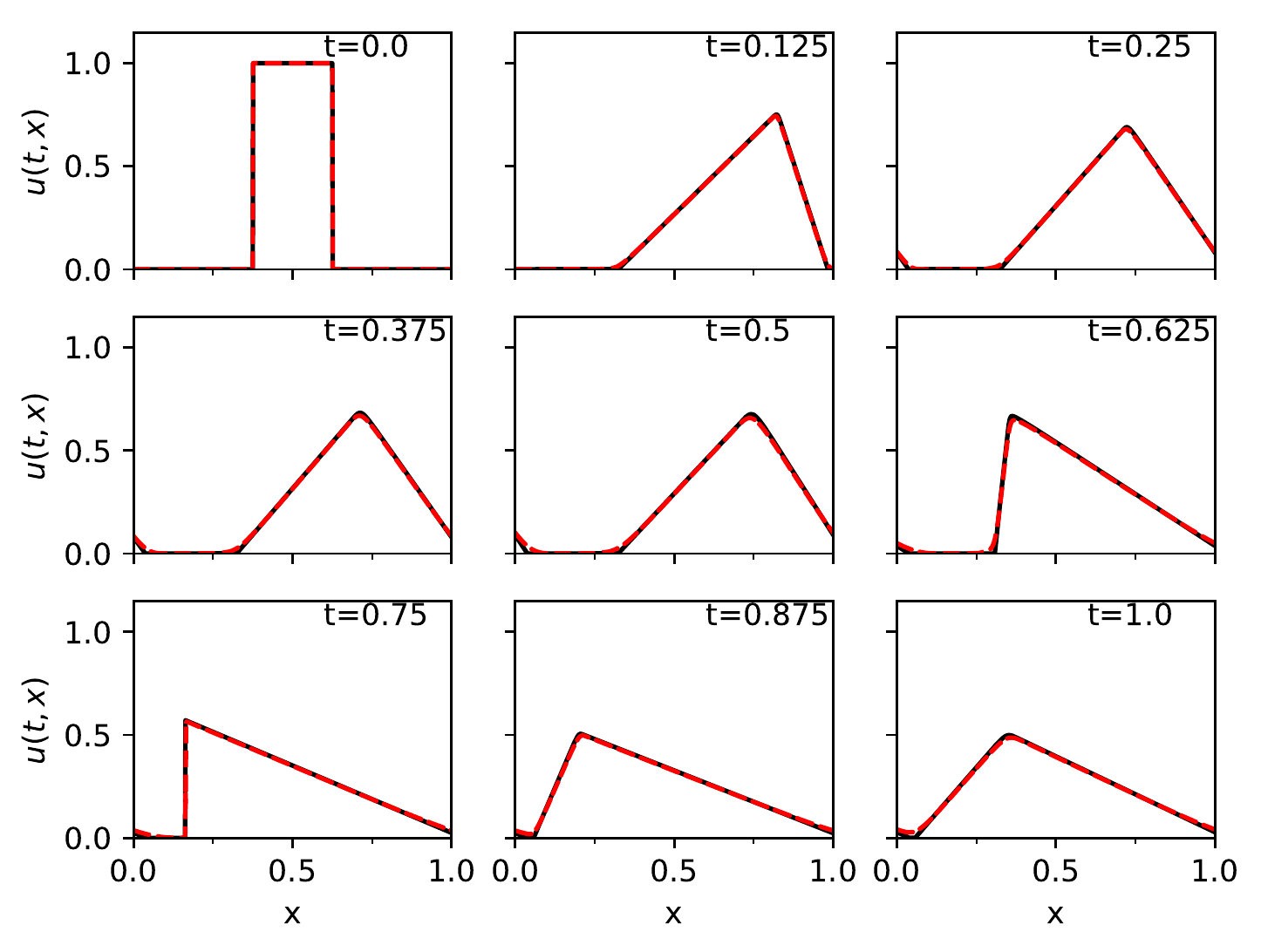}\\
\includegraphics[width= 0.67\textwidth]{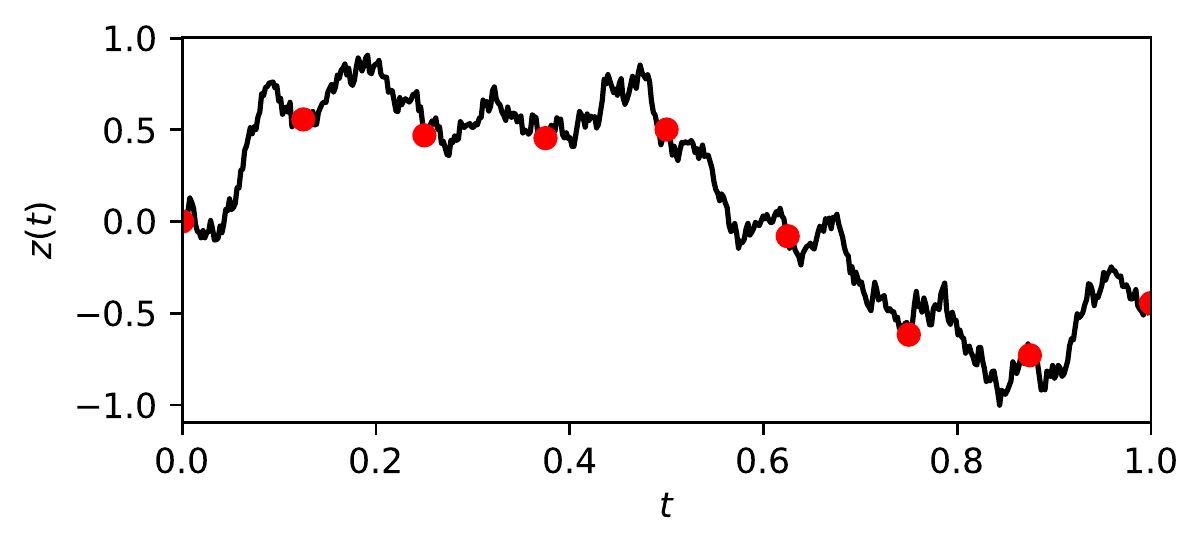}
\caption{Top: Time series of numerical solutions
with $\alpha = 0.5$ for Example~\ref{ex:fBM}. 
Lax--Friedrichs (red dashed line) and Engquist--Osher 
(black line). Bottom: The rough path $z$. 
Red points correspond to the value of $z$ at the
respective time series snapshots. }
\label{fig:ex2TSer50}
\end{figure}

\begin{figure}[h]
\includegraphics[width= 0.75\textwidth]{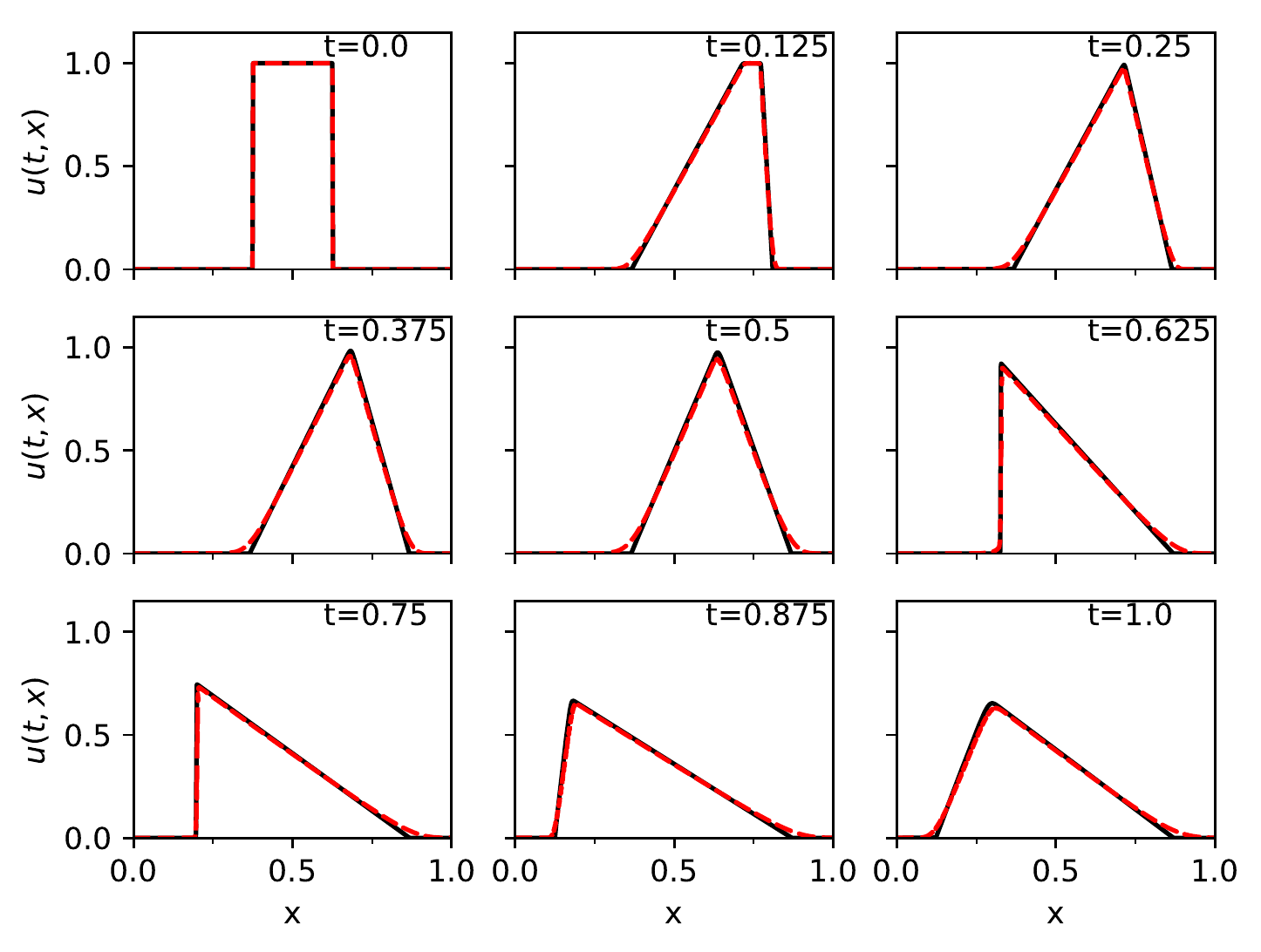}\\
\includegraphics[width= 0.67\textwidth]{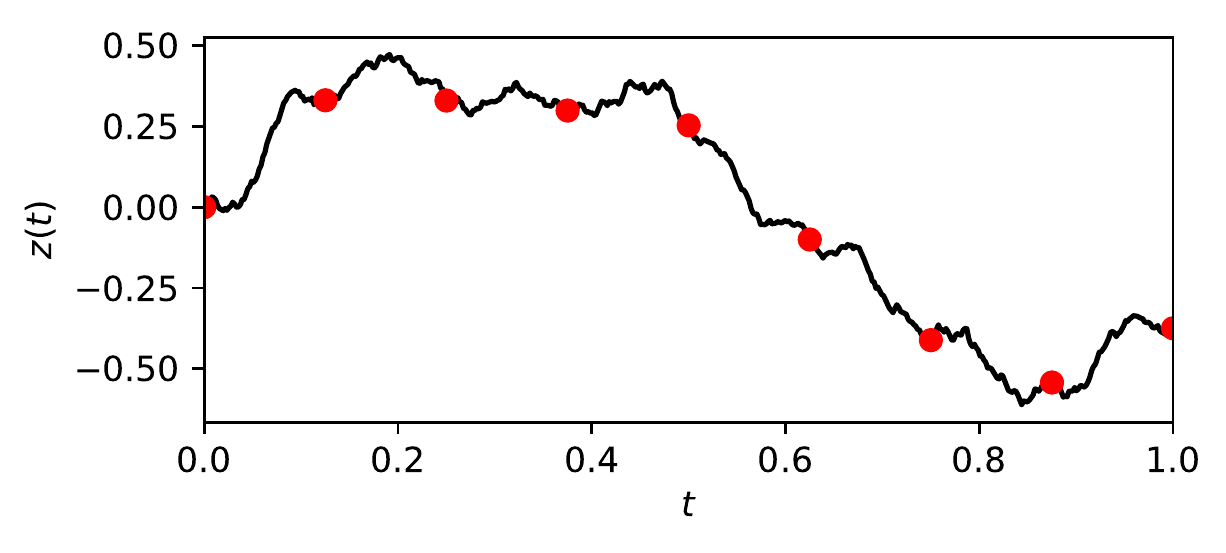}
\caption{Top: Time series of numerical solutions with $\alpha = 0.75$
  for Example~\ref{ex:fBM}.  Lax--Friedrichs (red dashed line) and
  Engquist--Osher (black line).  Bottom: The rough path $z$. Red
  points correspond to the value of $z$ at the respective time series
  snapshots.}
\label{fig:ex2TSer75}
\end{figure}

Figure~\ref{fig:ex2ConvRates} shows the final time approximation error
$\|U(1) -u(1)\|_1$ as a function of the resolution parameter $m$ for
both numerical schemes. The error is averaged over 10 fBM realizations
for each of the considered Hurst indices.  
The convergence rate decreases as $\alpha$ (and
thus the regularity of $z$) decreases, but it is consistently 
orders of magnitude faster than Theorem~\ref{thm:met1Compl2}'s 
possible worst case, $\cO{m^{-\alpha/2}}$.

\begin{figure}[!h]
\includegraphics[width= 0.65\textwidth]{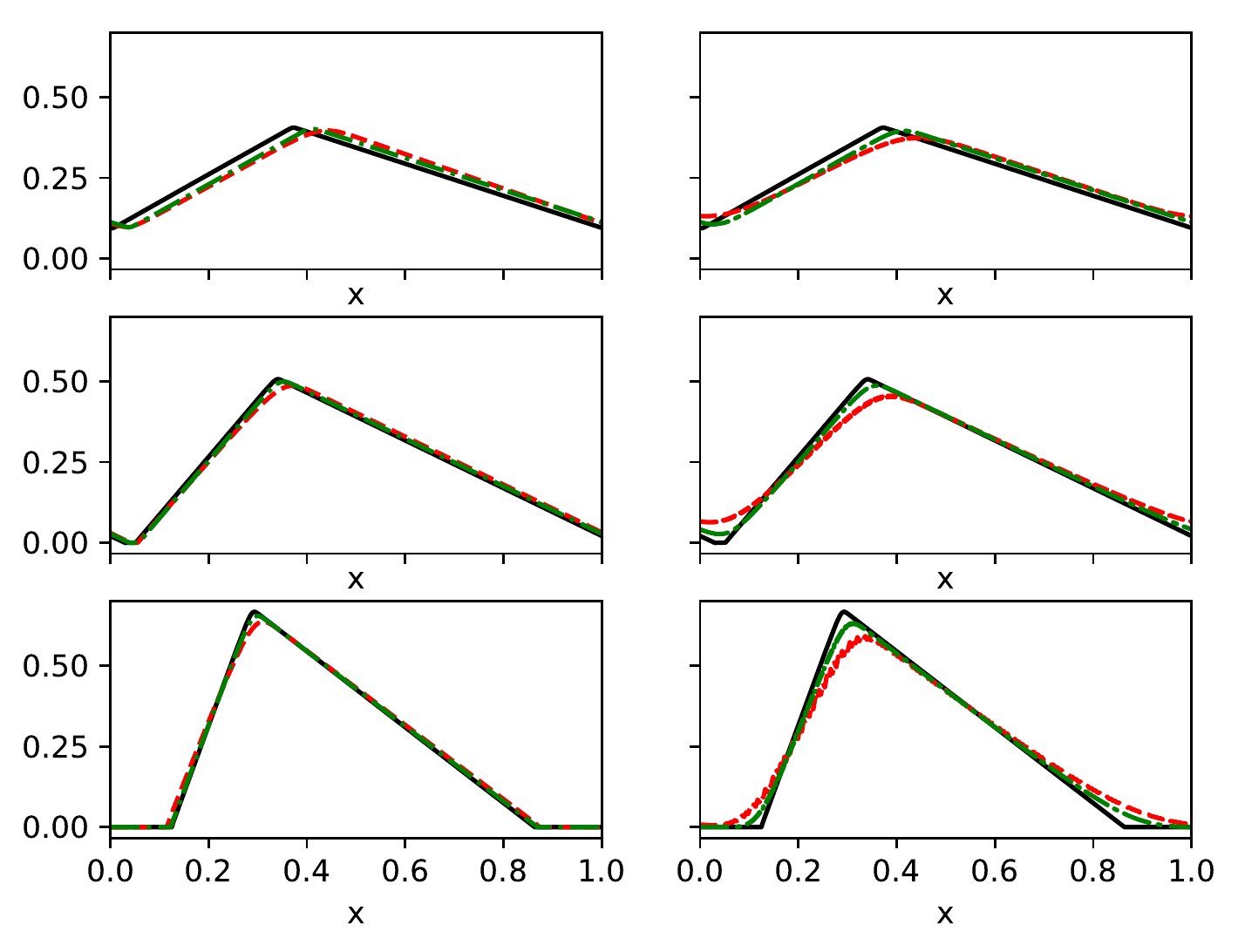}\\
\caption{Final time solutions for Example~\ref{ex:fBM}.
Reference solution (black line), numerical solutions 
at resolutions $m=2^6$ (red dashed line) and 
$m=2^{8}$ (green dash-dot line). 
Left column: Engquist--Osher scheme, Hurst index 
$\alpha =0.25, 0.5, 0.75$ from top to bottom, respectively.
Right column: Lax--Friedrichs scheme, Hurst index 
$\alpha =0.25, 0.5, 0.75$ from top to bottom, respectively.}
\label{fig:ex2FinalTime}
\end{figure}

\begin{figure}[!h]
\includegraphics[width= 0.94\textwidth]{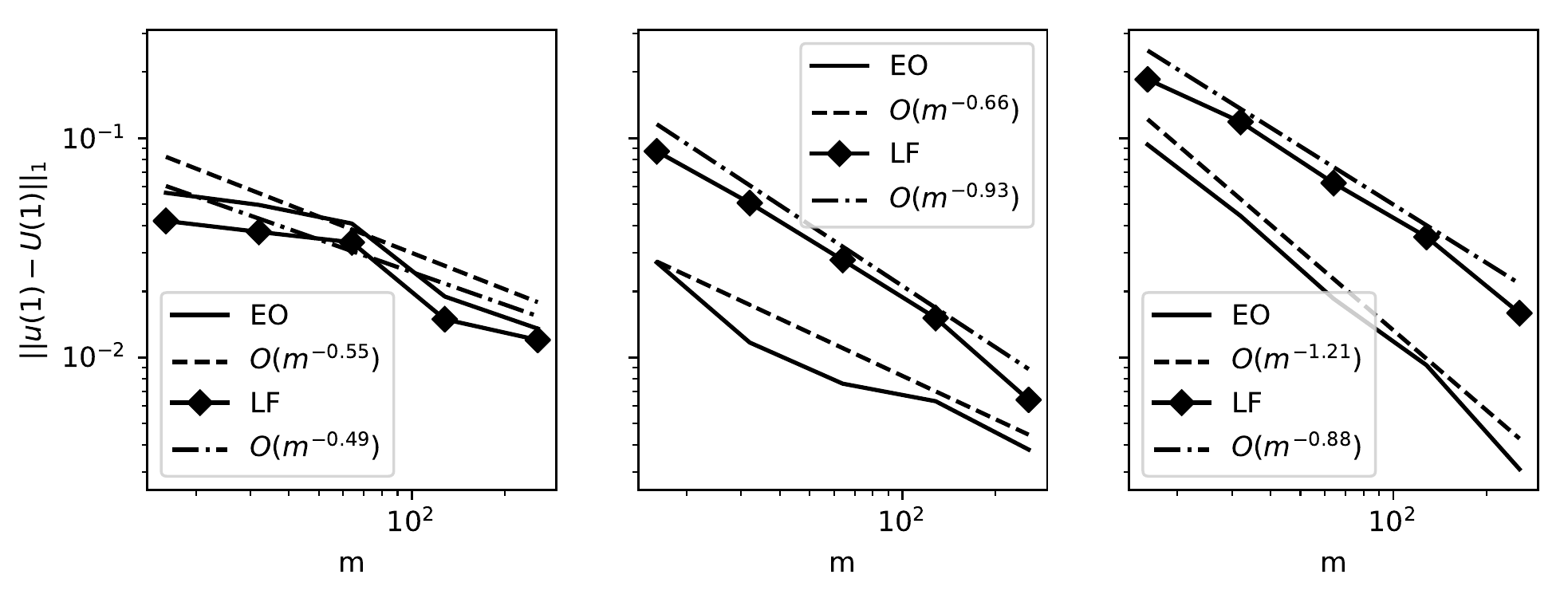}\\
\caption{Final time approximation error for Example~\ref{ex:fBM}.
Hurst indices $\alpha = 0.25, 0.5, 0.75$ from left to right, respectively. 
The abbreviations EO and LF denote respectively 
the Engquist--Osher and Lax--Friedrichs schemes.}
\label{fig:ex2ConvRates}
\end{figure}
\end{example}

\section{Cancellations and improved numerical methods}\label{sec:orm}

The numerical experiments in Example~\ref{ex:fBM} indicate that the
convergence rates obtained in Theorems~\ref{thm:met1Compl}
and~\ref{thm:met1Compl2} might not be sharp. In view of the more
precise adaptive timestep error analysis of the latter theorem, where
the factor $\abs{z^m}_{BV([0,T])}$ enters, one might suspect that the
error bounds could be improved if one were able to identify ``rough
path oscillations" resulting in ``cancellations" in the flux term
$\dot{z}^m \partial_xf(u)$.  In this section we identify such
oscillatory cancellations for strictly convex flux functions.  More
precisely, we show that if $f\in C^2(\bR)$ is strictly convex, then
the path $z^m \in I^m_0([0,T])$ can be replaced by a ``simpler" 
path $y^m \in I^m_0([0,T])$ (with smaller total variation) 
such that the solution of~\eqref{eq:consClass}
with path $y^m$ coincides with the corresponding solution
with path $z^m$ at final time $T$ (but not necessarily at earlier times).
An efficient numerical method that exploits this property
is developed in Section~\ref{sec:framework2}.

\subsection{Preliminaries}
Recall that for $\kappa \in \bR$, $s\ge 0$, and $v \in (L^1 \cap BV)(\bR)$,
$\cS^\kappa(s)v$ denotes the solution at time $t=s$ of
\[
\partial_tu +  \kappa \partial_xf(u) = 0 \quad
\text{in} \; (0,\infty)\times \bR,
\qquad u(0,\cdot) = v,
\]
so the entropy solution at time $t=T$ of~\eqref{eq:consClass}
with path $z^m \in I_0([0,T]; \{\tau_j\}_{j=0}^m)$ can be expressed by
\begin{equation}\label{eq:solRepTmp}
	u^m(T) = \cS^{\dot{z}_{m-1}^m}(\Delta \tau)
        \cS^{\dot{z}_{m-2}^m}(\Delta \tau)
	\ldots \cS^{\dot{z}_0^m}(\Delta \tau) u_0.
\end{equation}
By a change of variables, the solution mapping can be 
simplified to only depend on the path increments.
\begin{lemma}\label{lem:map2}
  For any $\Delta \tau > 0$, $\kappa \in \bR$ and $v \in (L^1 \cap BV)(\bR)$, 
  $\cS^{\kappa}(\Delta \tau)v$ coincides with 
  the entropy solution at time $t=|\kappa|\Delta \tau$ of
  \begin{equation}\label{eq:map2}
    \partial_t\tilde u + \sign{\kappa} \partial_xf(\tilde u) = 0
    \quad \text{in} \; (0,\infty)\times \bR,
    \qquad \tilde u(0) = v.
  \end{equation}
 \end{lemma}

\begin{proof}
For $\kappa=0$ the result trivially holds as $\cS^{0}(\cdot) = I$.
Otherwise, when $|\kappa|>0$, we verify the result by showing that 
$$
\bar u (t,x) := \tilde u(t |\kappa|,x)
$$
is an entropy solution of
$$
\partial_t \bar u + \kappa  \partial_xf(\bar u) = 0 
\quad \text{in} \; (0, \infty) \times \bR,
\qquad \bar u(0) = v.
$$
Let $\phi \in C^\infty_0( \bR\times \bR)$ be an
arbitrary nonnegative test function and set
\begin{equation}\label{eq:testDef}
\bar \phi(t,x):= \phi(t|\kappa|,x ).
\end{equation}
By construction, for all $c \in \bR$, we have
\[
\begin{split}
&\int_0^{|\kappa|\Delta \tau} \int_\bR |\tilde{u}-c| \partial_t \phi + \sign{\tilde u-c} 
\sign{\kappa}(f(\tilde u)-f(c)) \partial_x \phi \dx \dt \\
& \quad + \int_\bR |\tilde u(0,x) -c| \phi(0, x) - |\tilde u(|\kappa| \Delta
\tau,x) -c| \phi(|\kappa|\Delta \tau, x) \dx \ge 0.
\end{split}
\]
Making the change of variables $\bar t = t/|\kappa|$, we arrive at
\[
\begin{split}
	&\int_0^{\Delta \tau}\int_\bR |\bar{u}-c| (\partial_t \phi) \left(\bar t |\kappa|,x\right)
	\\ & \qquad \qquad
	+\sign{\bar u-c} \kappa (f(\bar u)-f(c)) 
	(\partial_x \phi)\left(\bar t |\kappa|,x\right) \dx \, \mathrm{d}\bar t \\
	& \quad +  \int_\bR |\bar u(0,x) -c| \bar \phi(0, x) - |\bar u(\Delta \tau , x) -c| \bar \phi(\Delta \tau, x) \dx \ge 0.
\end{split}
\]
Noting that
\[
\partial_t \phi \left(\bar t |\kappa|,x\right) =\partial_{\bar t} \phi \left(\bar t |\kappa|,x\right)  |\kappa|^{-1},
\]
it follows that
\[
\begin{split}
	&\int_0^{\Delta \tau}\int_\bR |\bar{u}-c| \partial_{\bar t}\bar \phi
	+\sign{\bar u-c}\kappa (f(\bar{u})-f(c)) 
	\partial_x \bar \phi \dx \, \mathrm{d}\bar t \\
	& \quad +  \int_\bR |\bar u(0,x) -c| \bar \phi(0, x) 
	- |\bar u(\Delta \tau , x) -c| \bar \phi(\Delta \tau, x) \dx \ge 0.
\end{split}
\]
In view of~\eqref{eq:testDef} and the invertibility of
the mapping $t\mapsto t |\kappa|$,
the above inequality holds for arbitrary
nonnegative $\bar \phi \in C^\infty_0( \bR\times \bR)$. 
\end{proof}

Let $\bar \cS(\cdot ) : \bR\times (L^1 \cap BV)(\bR) \to (L^1 \cap BV)(\bR)$ 
be the solution operator linked to 
\begin{equation}\label{eq:map3}
  \partial_t\tilde u + \sign{\Delta z} \partial_xf(\tilde u) = 0
  \quad  \text{in} \; (0,\infty)\times \bR,
    \qquad \tilde u(0) = v,
\end{equation}
that is, for $\Delta z \in \bR$ and $v \in (L^1 \cap BV)(\bR)$, 
$\bar \cS(\Delta z) v$ denotes the Kru{\v{z}}kov entropy solution 
at time $t=|\Delta z|$ of \eqref{eq:map3}.

Lemma~\ref{lem:map2} implies that for any 
$z^m \in I_0([0,T]; \{\tau_j\}_{j=0}^m)$,
$k \in \{0,1,\ldots, m-1\}$, $s \in [0,\Delta \tau_k]$ and $v \in (L^1 \cap BV)(\bR)$,  
\begin{equation}\label{eq:cSEquiv}
\cS^{\dot{z}^m_{k}}(s)v = \bar \cS(z^m(\tau_k+s)-z^m(\tau_k)) v.
\end{equation}
In view of~\eqref{eq:cSEquiv} and~\eqref{eq:uMSol}, the solution of~\eqref{eq:consClass}, for given $u_0 \in (L^1 \cap BV)(\bR)$, $f \in C^2(\bR)$ 
and driving path $z^m \in I^m_0([0,T])$, can be represented by 
\begin{equation}\label{eq:solRep2}
	u^m(t) = \begin{cases} 
          \bar\cS(z^m(t)-z^m(0)) u_0 & \text{if } t \in [0,\tau_1],\\
          \bar\cS(z^m(t)-z^m(\tau_1)) \bar\cS(\Delta z_0^m)  u_0 &
          \text{if } t \in (\tau_1, \tau_2],\\
          \cdots \\
          \bar\cS(z^m(t)-z^m(\tau_{m-1})) \bar\cS(\Delta {z}_{k-2}^m) 
          \cdots \bar \cS(\Delta {z}_0^m) u_0 & \text{if }
          t \in (\tau_{m-1}, T].
          \end{cases}
\end{equation}

To study how an entropy solution depends on the driving path, 
we introduce the notion ``oscillatory cancellations''.
\begin{definition}
For given $u_0 \in (L^1 \cap BV)(\bR)$, $f\in C^2(\bR)$ and $z^m \in
I^m_0([0,T])$, we say there are ``oscillatory cancellations'' over an interval
$[\tau_k, \tau] \subset [0,T]$, with $k \in \{0,1,\ldots,m-2\}$ and $\tau \in
(\tau_\ell, \tau_{\ell+1}]$ for some $k<\ell\le m-1$,
if it holds that
\begin{equation}\label{eq:semigroup}
  \begin{split}
    u^m(\tau) &= \bar \cS(z^m(\tau) - z^m(\tau_\ell)) \cdots \bar \cS(\Delta
    z_k^m) u^m(\tau_k)\\
    &= \bar \cS\prt{ z^m(\tau) - z^m(\tau_\ell) + \sum_{j=k}^{\ell-1}
      \Delta z_{j}^m} u^m(\tau_k)\\
    & =\bar \cS\prt{ z^m(\tau) - z^m(\tau_k) } u^m(\tau_k).
    \end{split}
\end{equation}
\end{definition}
Recall from~\cite[Theorem 2.15]{HoldenRisebro2015} 
that the solution operator $\bar \cS$
fulfills the following properties 
for all $u,v \in (L^1 \cap BV) (\bbR)$ and $s,t \in \bR$:
\begin{align}
\|\bar S(s)u-\bar S(s) v \|_{L^{1}(\bbR)} & \le \|u -
                                            v\|_{L^1(\bbR)}, \label{eq:stabOp} \\
\|\bar S(s)u-\bar S(t) u \|_{L^{1}(\bbR)} &\le
  \|f'\|_{L^\infty} |u|_{\BV{\bR}}
                                            |t-s|, \label{eq:stabOp2}\\
|\bar S(s)u |_{\BV{\bbR}} & \le |u_0|_{\BV{\bR}}. \label{eq:stabOp3}
\end{align}

\begin{definition}\label{def:runMaxMin}
For any $z \in C_0([0,T])$, let 
\begin{equation}\label{eq:runMaxMin}
	M^+[z](t) \coloneq \max_{s \in [0,t]} z(s) \quad 
	M^-[z](t) \coloneq \min_{s \in [0,t]} z(s), 
	\quad t\in [0,T],
\end{equation}
denote the running max/min functions of $z$.
\end{definition}

See Figure~\ref{fig:aPlusMin} for an example illustrating the 
running max/min functions.

To identify intervals with oscillatory cancellations 
we make use of the following regularity result:
\begin{corollary}\label{cor:regularity}
Let $u^m \in C_0([0,T]; L^1(\bR))$
denote the unique entropy solution 
of~\eqref{eq:consClass}, for given initial data 
$u_0 \in (L^1 \cap BV)(\bR)$, strictly convex 
flux $f\in C^2(\bR)$ and driving path $z^m \in I^m_0([0,T]) $. 
If for some $0\le s_1 < s_2 \le T$, 
\[
z^m(t) \in (M^{-}[z^m](s_1) , M^{+}[z^m](s_1)) \quad \forall t \in
(s_1, s_2),
\]
then for all $t \in (s_1, s_2)$, it holds that
$u^m(t) \in \mathrm{Lip}(\bbR)$ and  
\[
\sup_{x\neq y} \abs{\frac{u^m(t,y) - u^m(t,x)}{y-x}} \le
\frac{\|f''\|_{L^\infty}^{-1} }{\min\seq{z^m(t)-M^-[z^m](s_1), 
 \, M^+[z^m](s_1)- z^m(t)}},
\]
where (as usual) the $L^\infty$ is restricted to the interval 
$\left[-\|u_0\|_\infty, \|u_0\|_\infty\right]$. 
\end{corollary}
The corollary is a direct consequence of 
Lemma~\ref{lem:regLem}, $f\in C^2(\bbR)$, and the mean-value theorem. 
We refer to the companion work~\cite{Hoel:2017} for an 
in-depth theoretical treatment of regularity 
and cancellation properties for \eqref{eq:consClass}.

\begin{lemma}\label{lem:cancel}
For an $m\ge 2$, let $u^m \in C_0([0,T]; L^1(\bR))$
denote the unique entropy solution of~\eqref{eq:consClass} 
for some initial data $u_0 \in (L^1 \cap BV)(\bR)$, 
strictly convex $f\in C^2(\bR)$ and driving noise
$z^m \in I^m_0([0,T])$.
Then property~\eqref{eq:semigroup} holds
over an interval $[\tau_k,\tau] \subset [0,T]$
with $k\in \{0,1,\ldots,m-2\}$ and $\tau \in (\tau_\ell,
\tau_{\ell+1}]$ for some $k < \ell \le m-1$,
if at least one of the following conditions are met:
\begin{enumerate}
\item[(i)] $\dot{z}^m(s+) \coloneq \lim_{\delta \downarrow0} 
\dot{z}^m(s+\delta) \ge 0$ for all $s \in [\tau_k, \tau)$,

  \item[(ii)] $\dot{z}^m(s+) \le 0$ for all $s \in [\tau_{k}, \tau)$, 

  \item[(iii)] $z^m([\tau_k, \tau]) \subset [M^-[z^m](\tau_k),M^+[z^m](\tau_k)]$.
\end{enumerate}
\end{lemma}
\begin{remark}
Figure~\ref{fig:aPlusMin} exemplifies the running max/min functions 
for a piecewise linear function $z^{10} \in I_0([0,T]; \{\tau_j\}_{j=0}^{10})$, 
where $[\tau_5,\tau_7]$ is of type (i), $[\tau_7,\tau_{10}=T]$ is of type
(ii), and $[\tau_2,\tau_{5}]$ and $[\tau_7, \tau]$ with 
$\tau =\tau_9+2\Delta \tau/3$ are of type (iii).
\end{remark}

\begin{proof}
  
{\bf (i) \& (ii):} The condition $\dot{z}^m(\cdot +)|_{[\tau_k,\tau)} \ge 0$ 
implies that $\Delta z^m_j \ge 0$ for $j=k,k+1,\ldots,\ell-1$.
By the definition of $\bar \cS$, cf.~\eqref{eq:map3},
\[
\bar \cS(z^m(\tau)- z^m(\tau_\ell) ) \ldots \bar
\cS(\Delta z_k^m)u^m(\tau_k)
\]
is equal to the entropy solution at time
$t=z^m(\tau)- z^m(\tau_\ell)  
+ \sum_{j=k}^{\ell-1} \Delta z^m_j$ of 
\[
\partial_t \tilde u + \sign{z^m(\tau) - z^m(\tau_j) } \partial_x f(\tilde u) = 0,
\quad
\tilde u(0)= u^m(\tau_k).
\]
But, by definition, we also have that
\[
\bar \cS\prt{ z^m(\tau) - z^m(\tau_k)}u^m(\tau_k) = \tilde u(z^m(\tau) - z^m(\tau_k)),
\]
and it is clear that~\eqref{eq:semigroup} holds.
Part (ii) follows by a similar argument.

{\bf (iii):} We assume $M^+[z^m](\tau_k) -M^-[z^m](\tau_k)>0$, as 
otherwise $\dot{z}^m(\cdot+)|_{[\tau_k,\tau)} = 0$ 
and the cancellation property follows by (i) or (ii).

For a $\delta \in \prt{0,  (M^+[z^m](\tau_k) -M^-[z^m](\tau_k))/2}$ we define the
approximation path $z^{m,\delta} \in I_0([0,T]; \{\tau_j\}_{j=0}^m)$ by 
interpolating the values
\[
z^{m,\delta}(\tau_j) \coloneq \begin{cases} 
z^m(\tau_j) -\delta & \text{if } j \in \{k+1,\ldots,\ell\} \text{ and } z^m(\tau_j) = M^+[z^m](\tau_k),\\
z^m(\tau_j) +\delta & \text{if } j \in \{k+1,\ldots,\ell\} \text{ and } z^m(\tau_j) = M^-[z^m](\tau_k),\\
z^m(\tau_j) & \text{otherwise,}
\end{cases}
\]
over the set of interpolation points $\{\tau_j\}_{j=0}^m$.
By construction, it then holds that 
\[
z^{m,\delta}(s) \in (M^-[z^m](\tau_k), M^+[z^m](\tau_k)), \qquad
\forall  s \in (\tau_k, \tau),
\]
\begin{equation}\label{eq:zDeltaBound}
\Big| \Delta z^{m,\delta}_j  - \Delta z^m_j\Big|\le 2\delta, \quad \text{for
  all } \quad j \in \{k, k+1, \ldots, \ell-1\},
\end{equation}
where $\Delta z^{m,\delta}_j \coloneq z^{m,\delta}(\tau_{j+1}) -
z^{m,\delta}(\tau_j)$,
and
\begin{equation}\label{eq:zDeltaBound2}
\abs{z^{m,\delta}(\tau) - z^{m,\delta}(\tau_\ell) - (z^m(\tau) -
  z^m(\tau_\ell)) } \le 2\abs{ z^{m,\delta}(\tau_\ell) -  z^m(\tau_\ell)}
\le 2\delta.
\end{equation}

Let $u^{m,\delta}$ denote the solution
of~\eqref{eq:consClass} with driving path $z^{m,\delta}$  
and the same initial data $u_0$ and flux $f$ as the solution $u^m$. 
Since $z^{m,\delta}|_{[0,\tau_k]} = z^{m}|_{[0,\tau_k]}$, 
\eqref{eq:map3} implies
\[
u^{m,\delta}(t) = \begin{cases} u^m(t) & \text{if } t \le \tau_k,\\
\bar \cS(z^{m,\delta}(t) - z^m(\tau_k) )u^m(\tau_k) & \text{if } t \in
(\tau_k, \tau_{k+1}],\\
\cdots\\
\bar \cS(z^{m,\delta}(t)- z^{m,\delta}(\tau_{\ell}))\cdots 
\bar\cS(\Delta z^{m,\delta}_k ) u^m(\tau_k), & \text{if } t \in
(\tau_{\ell}, \tau].
\end{cases}
\]
By Corollary~\ref{cor:regularity} it holds that $u^{m,\delta}(s) \in
\mathrm{Lip}(\bbR)$ for all $s \in (\tau_k, \tau)$.
Consequently, $u^{m,\delta}|_{(\tau_k, \tau) \times \bbR}$ is a
classical solution that is time invertible, as it
can be obtained by the method of characteristics.
By~\eqref{eq:stabOp2}, we further obtain that 
\small
\[
\begin{split}
&\bar \cS\prt{z^{m,\delta}(\tau) - z^{m,\delta}(\tau_\ell) }
\cdots \bar\cS(\Delta z^{m,\delta}_k ) u^m(\tau_k)
\\ & =\lim_{\epsilon \downarrow 0}\Big[
\bar \cS\prt{z^{m,\delta}(\tau) - z^{m,\delta}(\tau-\epsilon) }
\bar \cS\prt{z^{m,\delta}(\tau-\epsilon)) - z^{m,\delta}(\tau_\ell) }\bar\cS(\Delta z^{m,\delta}_{\ell-1} )\cdots \\
&\quad \quad \bar\cS(\Delta z^{m,\delta}_{k+1})
\bar \cS\prt{z^{m,\delta}(\tau_{k+1}) - z^{m,\delta}(\tau_k+\epsilon)}
\bar \cS\prt{z^{m,\delta}(\tau_{k}+\epsilon) - z^{m,\delta}(\tau_k)}
u^m(\tau_k)\Big]\\
&= \lim_{\epsilon \downarrow 0}\Big[
\bar \cS\prt{z^{m,\delta}(\tau) - z^{m,\delta}(\tau-\epsilon) }
\bar \cS\prt{z^{m,\delta}(\tau-\epsilon)) - z^{m,\delta}(\tau_k+\epsilon)} \\
& \qquad \qquad \bar \cS\prt{z^{m,\delta}(\tau_{k}+\epsilon) - z^{m,\delta}(\tau_k)}
u^m(\tau_k)\Big]\\
&= \bar \cS\prt{z^{m,\delta}(\tau)) - z^{m,\delta}(\tau_k)}u^m(\tau_k).
\end{split}
\]
\normalsize
By \eqref{eq:stabOp}, \eqref{eq:stabOp2}, \eqref{eq:stabOp3}, 
\eqref{eq:zDeltaBound} and \eqref{eq:zDeltaBound2},
\small
\[
\begin{split}
& \norm{\bar \cS\prt{z^{m}(\tau) -
z^{m}(\tau_\ell) }\cdots 
\bar \cS(\Delta z_k^m) u^m(\tau_{k}) - \bar \cS\prt{ z^{m,\delta}(\tau) -
  z^{m,\delta}(\tau_k) } u^{m}(\tau_{k})}_{L^1} \\
& =  \norm{ \left[ \bar \cS\prt{z^{m}(\tau) - z^{m}(\tau_\ell) }
\cdots \bar \cS(\Delta z_k^m) 
- \bar \cS\prt{z^{m,\delta}(\tau) -z^{m,\delta}(\tau_\ell) }\cdots
      \bar\cS(\Delta z^{m,\delta}_k ) \right] u^m(\tau_k)}_{L^1}\\
  & \le \norm{\Big[ \bar \cS\prt{z^{m}(\tau) -z^{m}(\tau_\ell)}
      - \bar \cS\prt{z^{m,\delta}(\tau) -z^{m,\delta}(\tau_\ell)} \Big] 
      \bar \cS(\Delta z_{\ell-1}^m)\cdots \bar \cS(\Delta z_k^m) u^m(\tau_{k})}_{L^1} \\
&+ \norm{\bar \cS\prt{z^{m,\delta}(\tau) -
z^{m,\delta}(\tau_\ell)} \left[ \bar \cS(\Delta z_{\ell-1}^m) - \bar \cS(\Delta
    z_{\ell-1}^{m,\delta}) \right] \bar \cS(\Delta z_{\ell-2}^m)
    \cdots \bar \cS(\Delta z_k^m) u^m(\tau_{k})}_{L^1}\\
& + \cdots + \norm{\bar \cS\prt{z^{m,\delta}(\tau) -z^{m,\delta}(\tau_\ell)} 
\cdots \bar \cS(\Delta z_{k+1}^{m,\delta})  \left[\bar \cS(\Delta  z_{k}^m) -
    \bar \cS(\Delta z_k^{m,\delta}) \right] u^m(\tau_{k})}_{L^1}\\
& \leq 2\delta (\ell+1-k) \|f'\|_{\infty} |u_0|_{\mathrm{BV}}.
\end{split}
\]
\normalsize
Taking the limit $\delta \downarrow 0$, shows that ~\eqref{eq:semigroup}
holds in $L^1(\bR)$-sense.
\end{proof}

\subsection{The oscillating running max and min functions}

Lemma~\ref{lem:cancel} identifies a class of intervals over which  
an entropy solution with driving path $z^m \in I^m_0([0,T])$ 
experience ``oscillatory cancellations''.  In this section we construct 
an alternative driving path from $z^m$ that is free of 
Lemma~\ref{lem:cancel}'s type (iii) ``oscillatory cancellations'',
has lower total variation than $z^m$, and (under some assumptions) 
produces the same entropy solution at the final time as $z^m$ does.
The further removal of type (i), (ii) ``oscillatory cancellations'' 
is postponed to Section~\ref{sec:framework2}.

\begin{definition}\label{def:ormLDef}
For any $z \in C_0([0,T])$, 
we define the monotonically increasing 
functions $A^+[z],A^-[z] :[0,T] \to [0,T]$ by
 \begin{align*}
  A^+[z](t) &:= \min\{s \in [0,t]\, |\, z(s) = M^+[z](t)\},\\
  A^-[z](t)  &:= \min\{s \in [0,t] \,|\, z(s) = M^-[z](t)\},
\end{align*}
and the monotonically increasing c\`{a}dl\`{a}g 
functions $\bar A^+[z], \bar A^-[z] :[0,T] \to [0,T]$ by
\begin{equation}\label{eq:barADef}
\begin{split}
\bar A^+[z](t) &:= \begin{cases} 
\lim\limits_{\delta \downarrow 0} A^+[z](t+\delta)
  & \text{if} \quad t \in [0,T), \\
A^+[z](T) & \text{if} \quad  t = T,
\end{cases}\\
\bar A^-[z](t) &:= \begin{cases} 
\lim\limits_{\delta \downarrow 0} A^-[z](t+\delta)
  & \text{if}\quad  t \in [0,T), \\
 A^-[z](T) & \text{if} \quad  t = T.
\end{cases}
\end{split}
\end{equation}
For a given set of points 
\[
0=\tau_0 < \tau_1 < \ldots< \tau_m=T,
\]
and $z^m = \cI[z](\cdot; \{\tau_k\}_{k=0}^m)$, define
\begin{equation*}
A^\pm[z^m] := \bar
A^+[z^m](\{\tau_k\}_{k=0}^{m}) \cup \bar
A^-[z^m](\{\tau_k\}_{k=0}^{m}) \subset \{\tau_k\}_{j=0}^{m}.
\end{equation*}
The set inclusion $A^\pm[z^m] \subset \{\tau_k\}_{j=0}^{m}$
is verified in Lemma~\ref{lem:aPmProps}.

Note that $0 \in A^\pm[z^m]$, and let
$\{\tau_{j(k)}\}_{k=0}^{\bar m}  =A^\pm[z^m]\cup \{T\}$, with 
$0=j(0) < \cdots < j(\bar m) =m$, represent the subsequence 
of interpolation points in ascending order. The operator 
$\overline{\orm}: C_0([0,T]) \to I^m_0([0,T])$ is defined by 
\begin{equation}\label{eq:yDef}
\begin{split}
&\overline{\orm}[z](t;\{\tau_{j(k)}\}_{k=0}^{\bar m})\\ &= 
\sum_{k=0}^{\bar m -1} \1{t \in (\tau_{j(k)},
  \tau_{j(k+1)}]} \prt{ z(\tau_{j(k)})  + \prt{z(\tau_{j(k+1)})-
z(\tau_{j(k)}) }\frac{t- \tau_{j(k)}}{\tau_{j(k+1)}
-\tau_{j(k)}}}.
\end{split}
\end{equation}
We refer to $\overline{\orm}[z]\prt{\cdot;\{\tau_{j(k)}\}_{k=0}^{\bar m}}$ 
as the piecewise linear oscillating running
max/min (orm) function of $z$, and frequently 
use the shorthand $\overline{\orm}_m[z]
=\overline{\orm}[z]\prt{\cdot; \{\tau_{j(k)}\}_{k=0}^{\bar m}}$ (if no 
confusion is possible).
\end{definition}

See Figure~\ref{fig:aPlusMin} for an example illustrating 
$\bar{A}^{+}[z^m]$, $\bar{A}^{-}[z^m]$, and $\overline{\orm}_m[z]$.

\begin{figure}[H]
\includegraphics[width=0.75\textwidth]{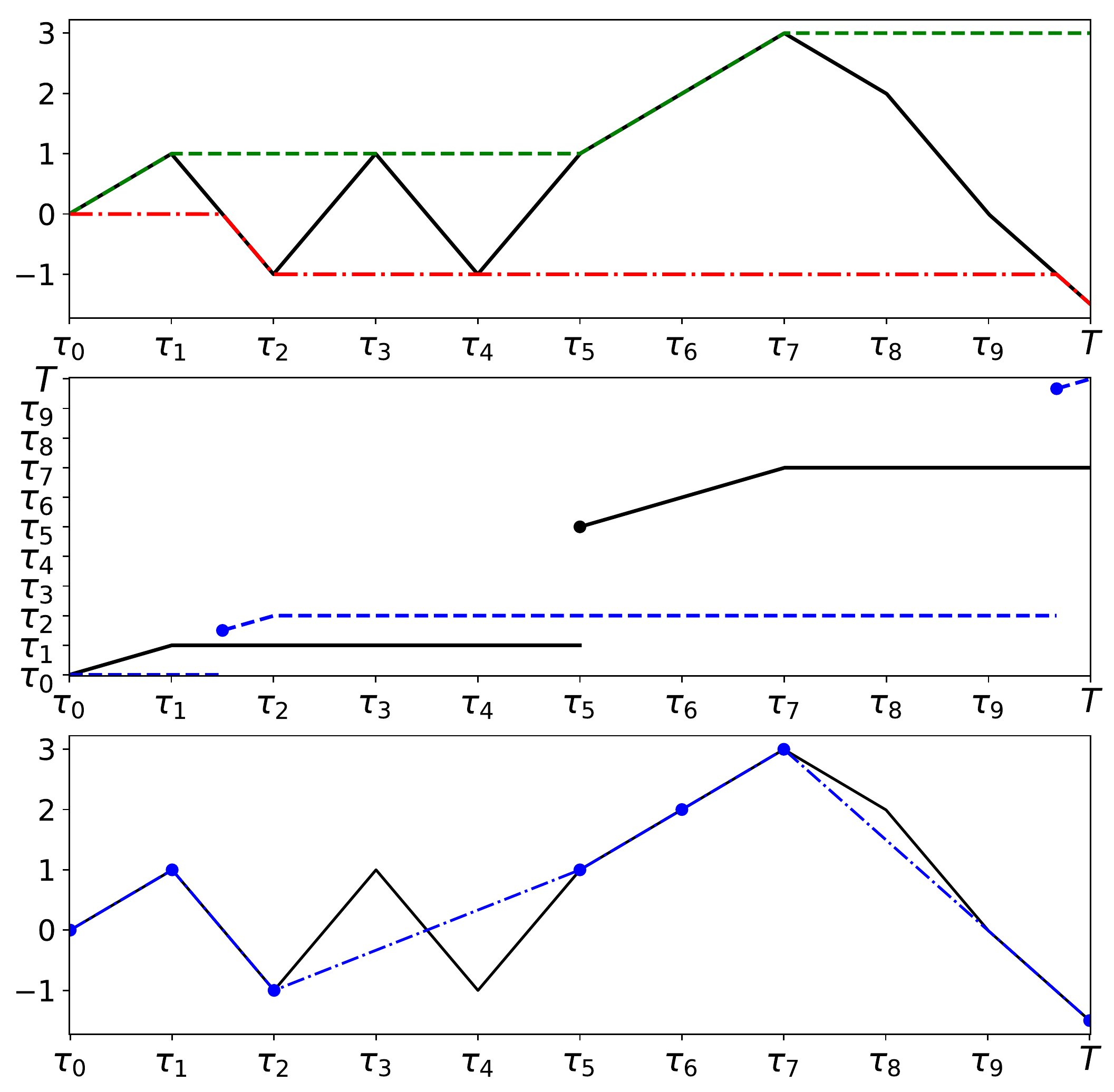}
\caption{Top: The piecewise linear path
  $z^{m} \in I_0([0,T]; \{\tau_j\}_{j=0}^{m})$ with $m=10$ (black line)
  and the associated running maximum $M^+[z^{m}]$ (green dashed line) and
  running minimum $M^-[z^m]$ (red dash-dotted line).  Middle: $\bar A^{+}[z^m]$
  (black line) and $\bar A^{-}[z^m]$ (blue dashed line).
  Black and blue dots illustrate that the respective functions
  are right-continuous at jump discontinuities. 
  Bottom: The piecewise linear path $z^m$ (black line) and the
  associated $\overline{\orm}_m[z]$ (blue dash-dotted line). Blue
  dots mark the value of $\overline{\orm}_m[z]$ at its interpolation points 
  $\{\tau_{j(k)}\}_{k=0}^m$.}
\label{fig:aPlusMin}
\end{figure}

For later reference we collect some properties of 
$\bar A^{\pm}[z]$ in a lemma.
\begin{lemma}\label{lem:aPmProps}
Assume that $z \in C_0([0,T])$. 
Then, for all $t \in [0,T]$,
\[
\max(\bar A^{+}[z], \bar A^-[z])(t) \le t,
\]
\begin{equation}\label{eq:MPmRelation}
\begin{split}
z(\bar A^+[z](t)) &=M^{+}[z](\bar A^+[z](t)) =M^{+}[z](t),   \\
z(\bar A^-[z](t)) &=M^{-}[z](\bar A^-[z](t)) = M^{-}[z](t), 
\end{split}
\end{equation}
\begin{equation}\label{eq:fixedPointsA}
\bar A^{+}[z](\bar A^+[z](t)) = \bar A^+[z](t),
\quad \bar A^{-}[z](\bar A^{-}[z](t)) = \bar A^{-}[z](t).
\end{equation}
Furthermore, for any set of points that satisfies
\[
0=\tau_0 < \tau_1 < \ldots< \tau_m=T, \qquad m \ge 2,
\]
and $z^m = \cI[z](\cdot; \{\tau_k\}_{k=0}^m)$,
it holds that
\begin{equation}\label{eq:setInclusion}
  A^{\pm}[z^m] \subset \{\tau_k\}_{k=0}^m
\end{equation}
and 
\begin{equation}\label{eq:fixedPointsAPm}
\tau = \max\seq{\bar A^{+}[z], \bar A^-[z]}(\tau) \quad \text{ for all
  $\tau \in A^{\pm}[z^m]$.}
\end{equation}
\end{lemma}
\begin{proof}
  By Definition~\ref{def:ormLDef},
  \[
  \max(\bar A^{+}[z], \bar A^-[z])(t) = \lim_{\delta \downarrow 0} \max(A^{+}[z], A^-[z])(t+\delta)
  \le \lim_{\delta \downarrow 0} (t+\delta) = t.
  \]
  
  To verify~\eqref{eq:MPmRelation}, $z \in C_0([0,T])$ implies that $M^+[z]\in C_0([0,T])$,
  \[
  \begin{split}
    z(\bar A^+[z](t)) &= z( \lim_{\delta \downarrow 0} A^+[z](t+\delta))\\
    &= \lim_{\delta \downarrow 0} z( A^+[z](t+\delta))\\
    & = \lim_{\delta \downarrow 0} M^{+}[z](t+\delta)\\
    &= M^{+}[z](t),
    \end{split}
  \]
  and
  \[
  \begin{split}
  M^{+}[z](\bar A^+[z](t)) &= M^{+}[z](\lim_{\delta \downarrow 0} A^+[z](t+\delta))\\
  & = \lim_{\delta \downarrow 0} M^{+}[z](A^+[z](t+\delta))\\
  & = \lim_{\delta \downarrow 0} M^{+}[z](t+\delta)\\
  & =M^{+}[z](t).
  \end{split}
  \]
The second line of~\eqref{eq:MPmRelation} follows by a similar argument.

To verify~\eqref{eq:fixedPointsA}, we begin by noting that 
for all $t\in [0,T]$,
\begin{equation}\label{eq:fixedAPlain}
\begin{split}
A^+[z](A^+[z](t)) &= \min\{s \in [0,A^+[z](t)] \mid z(s) = M^+[z](A^+[z](t))\}\\
& = \min\{s \in [0,A^+[z](t)] \mid z(s) = M^+[z](t)\}\\
& = A^{+}[z](t).
\end{split}
\end{equation}
By writing 
\[
B_1^+[z] = \{t \in [0,T) \mid  \lim_{\delta \downarrow 0} 
A^{+}[z](t+\delta) - A^{+}[z](t) =0\} \cup \{T\}
\]
and
\[
B_2^+[z] = \{t \in [0,T) \mid  \lim_{\delta \downarrow 0} 
A^{+}[z](t+ \delta) - A^{+}[z](t) >0\},
\]
we see that for all $t\in B_1^+[z]$, $\bar A^{+}[z](t) = A^+[z](t)$.
Hence,
\[
\bar A^{+}[z](\bar A^{+}[z](t) )\le \bar A^{+}[z](t) =  A^+[z](t)
\]
and since $\bar A^{+}[z] \ge A^{+}[z]$,
\[
\bar A^{+}[z](\bar A^{+}[z](t) )\ge  A^{+}[z](A^{+}[z](t))= A^+[z](t).
\]
We conclude that
\[
\bar A^{+}[z](\bar A^{+}[z](t) ) = A^+[z](t) = \bar A^{+}[z](t) 
\quad \forall t \in B_1^+[z].
\]
We claim that for all $t \in B_2^+[z]$ and $\delta \in (0,T-t)$, it 
holds that $A^{+}[z](t+\delta)>t$; supposing 
otherwise, \eqref{eq:fixedAPlain} and the monotonicity 
(increasing) of $A^+[z]$ leads to the following contradiction
\[
A^{+}[z](t+\delta) - A^+[z](t) = A^{+}[z](A^{+}[z](t+\delta)) - A^+[z](t) \le 0.
\]
Hence, for all $t \in B_2^+[z]$,
it follows by the preceding observation and $\bar A^{+}[z](t) \le t$ that 
$\bar A^{+}[z](\bar A^{+}[z](t)) =  \bar A^{+}[z](t)$.
The equality $\bar A^{-}[z](\bar A^{-}[z](t)) =  \bar A^{-}[z](t)$
can be verified by a similar argument.

To verify~\eqref{eq:setInclusion},
observe that since $z^m$ is piecewise linear, it holds that
$$
A^+[z^m](\{\tau_k\}_{k=0}^m) 
\cup A^-[z^m](\{\tau_k\}_{k=0}^m) 
\subset \{\tau_k\}_{r=0}^m,
$$
and for all $k \in \{0,1,\ldots, m\}$,
\[
\bar A^+[z^m](\tau_k) =
\begin{cases}
  A^+[z^m](\tau_k) & \text{if } \tau_k \in B_1^+[z^m],\\
  \tau_k & \text{if } \tau_k \in B_2^+[z^m],
  \end{cases}
\]
and
\[
\bar A^-[z^m](\tau_k) =
\begin{cases}
  A^-[z^m](\tau_k) & \text{if } \tau_k \in B_1^-[z^m],\\
  \tau_k & \text{if } \tau_k \in B_2^-[z^m],
  \end{cases}
\]
where
\[
B_1^-[z^m] = \{t \in [0,T) \mid  \lim_{\delta \downarrow 0} A^{-}[z^m](t+\delta) - A^{-}[z^m](t) =0\} \cup \{T\}
\]
and
\[
B_2^-[z] = \{t \in [0,T) \mid  \lim_{\delta \downarrow 0} A^{-}[z^m](t+ \delta) - A^{-}[z^m](t) >0\}.
\]

To verify~\eqref{eq:fixedPointsAPm}, equation~\eqref{eq:fixedPointsA} implies that
for any $\tau \in \bar A^{+}[z^m](\{\tau_k\}_{k=0}^m)$,
\[
\bar A^+[z^m](\tau) =\tau \quad \text{and} \quad  \bar A^-[z^m](\tau) \le \tau,
\]
and for any $\tau \in \bar A^{-}[z^m](\{\tau_k\}_{k=0}^m)$,
\[
\bar A^-[z^m](\tau) =\tau \quad \text{and} \quad  \bar A^+[z^m](\tau) \le \tau.
\]
\end{proof}
 
To extend the solution representation~\eqref{eq:solRep2}
to paths with jump discontinuities and to 
study asymptotic properties of $\overline{\orm}_m$ 
for $m\to \infty$, we introduce 
\begin{definition}
Let $\cD([0,T])$ denote the space of c\`{a}dl\`{a}g 
functions $g:[0,T] \to \bR$. For some $m \ge 2$, let 
$\{ \tau_{k}\}_{k=0}^{m}$ be a set of points with
$0= \tau_0 <  \tau_1 <  \ldots < \tau_{m} =T$.
Given $f \in C^2(\bR)$, $u_0 \in (L^1 \cap
BV)(\bR)$ and $y \in \cD([0,T])$, we define 
\begin{equation}\label{eq:solRep3}
v( \tau_k; y, \{\tau_{k}\}_{k=0}^{ m} ) \coloneq 
\begin{cases}
u_0 & \text{if } k =0\\
\bar\cS(\Delta y_{ k-1})
\bar\cS(\Delta {y}_{ k-2}) \ldots \bar \cS(\Delta {y}_0) u_0 &
\text{if } k \in \{1,\ldots,  m\},
\end{cases}
\end{equation}
where $\Delta y_k = y(\tau_{k+1}) - y(\tau_k)$.
\end{definition}

Note that for any
$f \in C^2(\bR)$, $u_0 \in (L^1 \cap BV)(\bR)$,
set of points
\[
0= \tau_0 <   \tau_1 <  \ldots <  \tau_{m} =T, \qquad m\ge 2,
\]
and $y \in \cD([0,T])$, it holds
by~\eqref{eq:stabOp},~\eqref{eq:stabOp3} and induction that
\[
\|v(\tau_k ; y, \{\tau_{r}\}_{r=0}^{m} ) \|_{L^1} \le
\|u_0\|_{L^1} \qquad 
\abs{v(\tau_k ; y, \{\tau_{r}\}_{r=0}^{m} )}_{\mathrm{BV}} 
\le |u_0|_{\mathrm{BV}} 
\quad \forall k \le m.
\]
Moreover, cf.~\eqref{eq:solRep2},
\[
v(T; z, \{\tau_{k}\}_{k=0}^{ m} ) = v(T; \cI^m[z],
\{\tau_{k}\}_{k=0}^{ m} ) = u^m(T), 
\quad z \in C_0([0,T]).
\]

The next theorem shows that if the flux is strictly convex, then the driving
paths $\cI^m[z]$ and $\overline{\orm}_m[z]$ produce the same entropy 
solution at final time.

\begin{theorem}\label{thm:ormLinEquiv}
Assume that $f\in C^2(\bR)$ is strictly convex, $u_0 \in (L^1\cap
BV)(\bR)$ and $z \in C_0([0,T])$.
For some $m \ge 2$, let $\{ \tau_{k}\}_{k=0}^{m}$ denote a set of 
points satisfying
$0= \tau_0 <  \tau_1 <  \ldots < \tau_{m} =T$,
$z^m = \cI^m[z]$, and let 
\[
0=\tau_{j(0)}< \tau_{j(1)} <\ldots< \tau_{j(\bar m)} =T
\]
denote the associated interpolation points of~$\,\overline{\orm}[z]\prt{\cdot; \{\tau_{j(k)}\}_{j=0}^{\bar m}}$,
cf.~Definition~\ref{def:ormLDef}.
Then
\[
v(T; \overline{\orm}_m[z], \{\tau_{j(k)}\}_{k=0}^{\bar m} )
= v(T; \overline{\orm}_m[z], \{\tau_{k}\}_{k=0}^{ m} ) =
v(T; z^m, \{\tau_{k}\}_{k=0}^{ m} ).
\]
\end{theorem}
\begin{proof}
Recall from the solution representations~\eqref{eq:solRep2}
and~\eqref{eq:solRep3} that
\[
v(T; z^m, \{\tau_{k}\}_{k=0}^{ m} ) = u^m(T) 
= \bar \cS(\Delta z^m_{m-1}) \cdots \bar \cS(\Delta z^m_0)u_0.
\]
By writing $y^m = \overline{\orm}_m[z]$ and 
introducing the shorthand 
\[
v^m(\tau_k) \coloneq v(\tau_k; y^m, \{\tau_{r}\}_{r=0}^{ m}), \quad
\forall k \in \{0,1,\ldots,m\},
\]
we have by~\eqref{eq:solRep3},
\[
v^m(T) = \bar \cS(\Delta y^m_{m-1}) \ldots \bar \cS(\Delta y^m_0)u_0. 
\]
For any $k \in\{0,\ldots,\bar m-1\}$,
$y^m|_{[\tau_{j(k)}, \tau_{j(k+1)}]}$ is a linear function.
Therefore, either
\[
\dot{y}^m(s+) \ge 0, \quad \forall s \in [\tau_{j(k)}, \tau_{j(k+1)})
\] 
or
\[
\dot{y}^m(s+) \le 0, \quad \forall s \in [\tau_{j(k)}, \tau_{j(k+1)}),
\]
for any $k \in\{0,\ldots,\bar m-1\}$,
and Lemma~\ref{lem:cancel} and $y^m|_{A_\tau^{\pm}[z^m] \cup \{T\}} 
=z^m|_{A_\tau^{\pm}[z^m]  \cup \{T\}}$, cf.~\eqref{eq:yDef}, yield
\[
\begin{split}
v^m(T) &= \bar \cS\prt{y^m(\tau_{j(\bar m)}) - y^m(\tau_{j(\bar
  m-1)}) } \ldots 
\bar \cS\prt{y^m(\tau_{j(1)}) - y^m(\tau_{j(0)})}u_0\\
&= \bar \cS\prt{z^m(\tau_{j(\bar m)}) - z^m(\tau_{j(\bar
  m-1)}) } \ldots 
\bar \cS\prt{z^m(\tau_{j(1)}) - z^m(\tau_{j(0)})} u_0.
\end{split}
\]

Assume $\bar m >1$ (otherwise $z^m|_{[0,\tau_{m-1}]} = 0$ and the
lemma trivially holds). We claim that
for all $k \in \{0,\ldots,\bar m-1\}$ such that $j(k+1) - j(k) \ge 2$,
\begin{equation}\label{eq:zMIter}
\begin{split}
\bar \cS\prt{z^m(\tau_{j(k+1)}) - z^m(\tau_{j(k)}))u^m(\tau_{j(k)}} &=
\bar \cS\prt{\Delta z^m_{j(k+1)-1}} \ldots \bar
\cS\prt{\Delta z^m_{j(k)}}  u^m(\tau_{j(k)})\\
& = u^m(\tau_{j(k+1)}).
\end{split} 
\end{equation}
Define
\[
 k_1 = \min \seq{k \in \{0,\ldots,\bar m-1\} \big| j(k+1) - j(k) \ge 2}
\cup \seq{\bar m} .
\] 
Then, since $j(k+1) - j(k)=1$  for all 
$k <  k_1$, it follows from the solution representations~\eqref{eq:solRep2}
and~\eqref{eq:solRep3} that $v^m(\tau_{j(k_1)}) = u^m(\tau_{j(k_1)})$. 
If $k_1 = \bar m$, we have $v^m(T) = u^m(T)$.
Otherwise, if $k_1 < \bar m$, assumption~\eqref{eq:zMIter} implies that 
$v^m(\tau_{j(k_1+1)}) = u^m(\tau_{j( k_1+1)})$.
Let
\[
k_2 = \min \seq{k \in \{ k_1+1,\ldots,\bar m-1\} | j(k+1) - j(k) \ge 2}
\cup \{\bar m\} ,
\]
and argue as above to conclude that if $k_2 = \bar m$, then 
$v^m(T) = u^m(T)$, and otherwise, $v^m(\tau_{j( k_2+1)}) =
u^m(\tau_{j(k_2+1)})$. The lemma follows by induction once we
have verified the claim~\eqref{eq:zMIter}.

Suppose $k \in \{0,\ldots,\bar m-1\}$ satisfies $j(k+1) - j(k) \ge 2$.
Consider two cases: $\tau_{j(k+1)} \notin
A^\pm[z^m]$ and $\tau_{j(k+1)} \in A^\pm[z^m]$.

The case $\tau_{j(k+1)} \notin A^\pm[z^m]$ is only
possible if $k=\bar m -1$, i.e., if $\tau_{j(k+1)} = T$. 
Then
\[
z^m(\tau_\ell) \in  [M^-[z^m](\tau_{j(k)}), M^+[z^m](\tau_{j(k)})],
\quad \forall \ell \in \{j(\bar m-1), j(\bar m-1)+1, \ldots, j(\bar m)\},
\]
which implies that
\[
z^m([\tau_{j(k)}, T]) \subset  [M^-[z^m](\tau_{j(k)}),
M^+[z^m](\tau_{j(k)})],
\]
and~\eqref{eq:zMIter} follows by Lemma~\ref{lem:cancel}. 

For the second case; $\tau_{j(k+1)} \in A^\pm[z^m]$, then
\[
z^m(\tau_\ell) \in  [M^-[z^m](\tau_{j(k)}), M^+[z^m](\tau_{j(k)})],
\quad \forall \ell \in \{j(k), j(k)+1,
\ldots, j(k+1)-1\},
\]
and there exists a unique $\tau \in
(\tau_{j(k+1)-1}, \tau_{j(k+1)}]$ such that 
\begin{equation}\label{eq:uniqueTau}
	\begin{split}
		& z^m([\tau_{j(k)}, \tau]) 
		\subset [M^-[z^m](\tau_{j(k)}), M^+[z^m](\tau_{j(k)})],
		\\ &
		\tau \in \{\bar A^+[z^m](\tau), \bar A^-[z^m](\tau)\}.
	\end{split}
\end{equation}
To show \eqref{eq:uniqueTau}, if $z^m(\tau_{j(k+1)}) \in  [M^-[z^m](\tau_{j(k)}), M^+[z^m](\tau_{j(k)})]$
then $\tau = \tau_{j(k+1)}$. Otherwise, if
$z^m(\tau_{j(k+1)}) \notin  [M^-[z^m](\tau_{j(k)}), M^+[z^m](\tau_{j(k)})]$,
then\\
$z^m(\tau_{j(k+1)-1}) \in  [M^-[z^m](\tau_{j(k)}), M^+[z^m](\tau_{j(k)})]$
implies that $|\Delta z^m_{j(k+1)-1}| >0$. This observation and 
$\tau_{j(k+1)-1} \notin A^{\pm}[z^m]$ verifies statement~\eqref{eq:uniqueTau}.

By~\eqref{eq:uniqueTau} and Lemma~\ref{lem:cancel},
\[
\begin{split}
& \bar \cS(\Delta z^m_{j(k+1)-1}) \ldots
 \bar\cS(\Delta z^m_{j(k)})  u^m(\tau_{j(k)}) \\
 &= \bar \cS\prt{\Delta z^m(\tau_{j(k+1)}) - z^m(\tau)}
 \bar \cS\prt{z^m(\tau) - z^m(\tau_{j(k+1)-1})}\ldots
 \bar\cS(\Delta z^m_{j(k)})  u^m(\tau_{j(k)})  \\
&= \bar \cS(\Delta z^m(\tau_{j(k+1)}) - z^m(\tau)) \bar \cS(z^m(\tau) -z^m(\tau_{j(k)}))  u^m(\tau_{j(k)}).
\end{split}
\]
If $\tau_{j(k+1)} = \bar A^+[z^m](\tau_{j(k+1)})$, then it must hold
that $\Delta z^m_{j(k+1)-1}\ge0$, hence, $z^m(\tau_{j(k+1)})-
z^m(\tau) \ge 0$. The former inequality and
$\tau \in \{\bar A^+[z^m](\tau), \bar A^-[z^m](\tau)\}$
imply that $z^m(\tau) = M^+[z^m](\tau)$, so that also 
$z^m(\tau) - z^m(\tau_{j(k)}) \ge 0$.
If, on the other hand, $\tau_{j(k+1)} = \bar A^-[z^m](\tau_{j(k+1)})$,
then a similar argument yields 
\[
z^m(\tau_{j(k+1)})- z^m(\tau)\le 0 \quad \text{and}\quad  z^m(\tau)-z^m(\tau_{j(k)}) \le 0.
\] 
We conclude from the above that if $\tau_{j(k+1)} \in  A^\pm[z^m]$, then
\[
(z^m(\tau_{j(k+1)})- z^m(\tau)) (z^m(\tau) -z^m(\tau_{j(k)})) \ge 0.
\]
Since both increments either are non-negative or non-positive, 
\[
\bar \cS\prt{z^m(\tau_{j(k+1)})- z^m(\tau)}
\bar \cS\prt{z^m(\tau) -z^m(\tau_{j(k)})} u^m(\tau_{j(k)})
\]
equals the unique entropy solution at time
\[
t = \abs{z^m(\tau_{j(k+1)} ) - z^m(\tau) + (z^m(\tau)
  -z^m(\tau_{j(k)}))} = \abs{z^m(\tau_{j(k+1)} )  -z^m(\tau_{j(k)})}
\]
 of 
\[
\partial_t \tilde u + \sign{z^m(\tau_{j(k+1)} ) -z^m(\tau_{j(k)})} \partial_x
f(\tilde u) = 0, \quad \tilde u(0) = u^m(\tau_{j(k)}),
\]
cf.~\eqref{eq:map3}.
However,
\[
\bar \cS(z^m(\tau_{j(k+1)})-z^m(\tau_{j(k)})) u^m(\tau_{j(k)})
\]
also equals $\tilde u(|z^m(\tau_{j(k+1)} ) -z^m(\tau_{j(k)})| )$.
Hence,
\[
\begin{split}
u^m(\tau_{j(k+1)})&=\bar \cS(\Delta z^m_{j(k+1)-1}) 
\cdots \bar\cS(\Delta z^m_{j(k)})  u^m(\tau_{j(k)}) \\
&= \bar \cS(\Delta z^m(\tau_{j(k+1)}) - z^m(\tau)) \bar \cS(z^m(\tau)-z^m(\tau_{j(k)}))  u^m(\tau_{j(k)})\\
& = \bar \cS(z^m(\tau_{j(k+1)})-z^m(\tau_{j(k)})) u^m(\tau_{j(k)}).
\end{split}
\]
\end{proof}

The operator $\overline{\orm}_m:C_0([0,T]) \to I_0([0,T]; \{\tau_j\}_{j=0}^m)$ introduced in
Definition~\ref{def:ormLDef} maps every driving path $z \in C_0([0,T])$ to a less 
oscillatory driving path $y^m = \overline{\orm}_m[z]$. In
Theorem~\ref{thm:ormLinEquiv} it is shown that,
provided $u_0 \in (L^1 \cap BV)(\bR)$ and $f \in C^2(\bR)$ is
strictly convex, the paths $z^m = \cI^m[z]$ and $\overline{\orm}_m[z]$ are
equivalent in the sense of preserving final time solutions: 
\[
 v(T; \overline{\orm}_m[z],\{\tau_{k}\}_{k=0}^{m}) = v(T; z^m, \{\tau_{k}\}_{k=0}^{m}) = u^m(T), 
\]
for all $z \in C_0([0,T])$ and meshes $\{\tau_k\}_{k=0}^m$ for $m \ge 2$.

We next introduce an operator $\orm: C_0([0,T]) : \to \cD([0,T])$ 
that may be viewed as the ``imit extension" of the $\overline{\orm}_m$ 
operators in the sense that
\[
v(T; \orm[z^m] , \{\tau_{k}\}_{k=0}^{m}) 
= v(T;\overline{\orm}_m[z], \{\tau_{k}\}_{k=0}^{m}), 
\]
cf.~Theorem~\ref{thm:ormEquiv}, and, under the more restrictive
assumptions of Theorem~\ref{thm:limOrmBound},
\[
\lim_{m \to \infty} v(T; \orm[z],\{\tau_k\}_{k=0}^{m}) = u(T).
\]

\begin{definition}\label{def:orm}
We define the operator $\orm[\cdot]: C_0([0,T]) \to \mathcal{D}([0,T])$ by
\begin{equation*}
	\orm[z](t) \coloneq\begin{cases} M^+[z](t) \1{\bar A^+[z](t)> \bar A^-[z](t)} 
	+ M^-[z](t) \1{ \bar A^-[z](t)>  \bar A^+[z](t)} &
        \text{if } t \in [0,T), \\
          z(T) & \text{if } t = T,
\end{cases}
\end{equation*}
and refer to $\orm[z]$ as the oscillating running 
max/min (orm) function of $z$. 
\end{definition}
See Figure~\ref{fig:orm} for examples 
illustrating the $\orm[z]$ functions. 

\begin{figure}[H]
\includegraphics[width=0.99\textwidth]{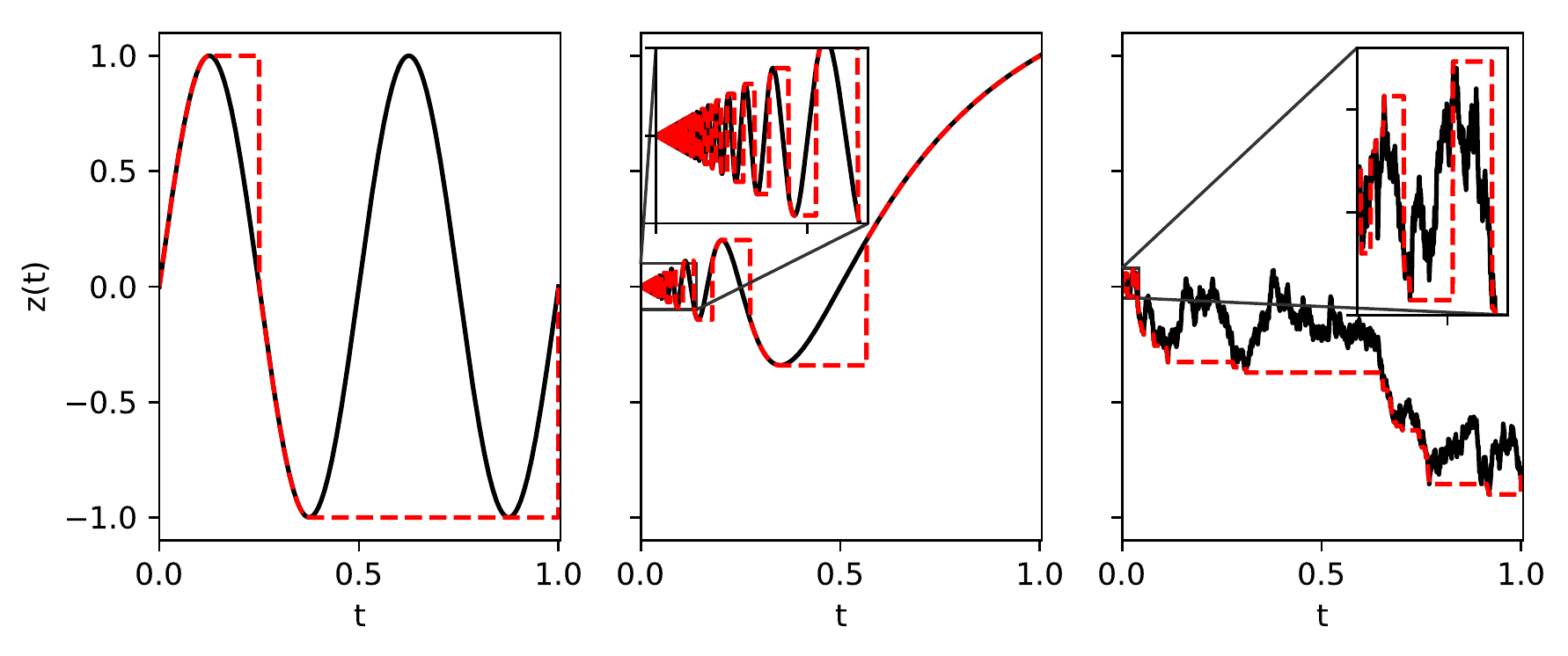}
\caption{Illustration of a rough path $z$ (black line) and
  $\orm[z]$ (red dashed line).
Left plot: $z(t) = \sin(4\pi t)$. 
Middle plot: $z(t) = \1{t>0} t\sin(\pi/(2t))$. 
Right plot: $z$ is a realization of a standard Wiener process.}
\label{fig:orm}
\end{figure}

It remains to verify that the orm functions are c\`{a}dl\`{a}g.
\begin{lemma}\label{lem:ormCadlag}
For all $z \in C_0([0,T])$,  $\orm[z] \in \cD([0,T])$.
\end{lemma}
\begin{proof}
Note first that if $t \in \bar A^+[z]([0,T]) \cap \bar A^-[z]([0,T])$,
then by~\eqref{eq:MPmRelation}, 
\[
z(t) = M^+[z](t) = M^-[z](t) \implies \text{$z(s) =0$ for all $s\in [0,t]$.}
\]
By~\eqref{eq:MPmRelation} and Definition~\ref{def:orm}, 
we have for all $t \in [0,T)$,
\[
\begin{split}
\orm[z](t) &=  M^+[z](t) \1{\bar A^+[z](t)> \bar A^-[z](t)} 
	+ M^-[z](t) \1{ \bar A^-[z](t)>  \bar A^+[z](t)} \\
&= \begin{cases} 
M^+[z](t) = z(\bar A^+[z](t)) & \text{if} \quad \bar A^+[z](t) \ge \bar
A^-[z](t)\\
M^-[z](t) = z(\bar A^-[z](t)) & \text{if} \quad \bar A^-[z](t) > \bar
A^+[z](t),
\end{cases}
\end{split}
\]
and $\orm[z](T) = z(T)$. 
Hence, 
\begin{equation}\label{eq:ormMaxRelation}
\orm[z](t) = \begin{cases} z(\max(\bar A^+[z], \bar A^-[z])(t)) &
  \text{if } t \in [0,T)\\
z(T) & \text{if } t = T,
\end{cases}
\end{equation}
and, since $z \in C_0([0,T])$ and 
\[
g(t) =  \1{t <T} \max(\bar A^+[z], \bar A^-[z])(t) + \1{t=T} T
\]
is c\`{a}dl\`{a}g whenever $\bar A^+[z], \bar A^-[z] \in \cD([0,T])$, 
it follows that also $\orm[z] = z\circ g$ belongs to $\cD([0,T])$.

\end{proof}

\begin{theorem}\label{thm:ormEquiv}
Assume that $f\in C^2(\bR)$ is strictly convex, 
$u_0 \in (L^1\cap BV)(\bR)$, and $z \in C_0([0,T])$. For some $m\ge 2$,
let $\{\tau_k\}_{k=0}^m$ be a set of points satisfying
\[
0=\tau_0< \tau_1 <\ldots< \tau_m =T,
\]
$z^m = \cI^m[z]$, and let
\[
0=\tau_{j(0)}< \tau_{j(1)} <\ldots< \tau_{j(\bar m)} =T
\]
be the associated interpolation points of 
$\overline{\orm}[z](\cdot; \{\tau_{j(k)}\}_{k=0}^{\bar m})$,
cf.~Definition~\ref{def:ormLDef}.
Then 
\begin{equation}\label{eq:ormSol}
  v(T;\orm[z^m], \{\tau_{j(k)}\}_{k=0}^{\bar m}) =
  v(T; \orm[z^m], \{\tau_{k}\}_{k=0}^m) =
   u^m(T).
\end{equation}
\end{theorem}

\begin{proof}
Set $y^m = \overline{\orm}_m[z]$ and $\hat y^m = \orm[z^m]$.
By~\eqref{eq:fixedPointsAPm},~\eqref{eq:ormMaxRelation}
and~\eqref{eq:yDef}, 
it holds for all $k \in \{0,1,\ldots, \bar m-1\}$ that
\begin{equation}\label{eq:pointEquality}
  \begin{split}
    \hat y^m(\tau_{j(k)})  &= z^m\prt{\max(\bar A^+[z^m], \bar A^-[z^m])(\tau_{j(k)})}
     = z^m(\tau_{j(k)})
    = y^m(\tau_{j(k)}),
    \end{split}
  \end{equation}
  and
  \[
  \hat y^m(\tau_{j(\bar m)}) = z^m(T) = y^m(\tau_{j(\bar m)}).
  \]
By Theorem~\ref{thm:ormLinEquiv},
\begin{multline}
\bar \cS(\Delta y_{\tau_{j(\bar m) -1}}^m ) 
\cdots \bar \cS( \Delta y_{\tau_{j(0)}}^m) u_0\\ 
= \bar \cS(y^m(\tau_{j(\bar m)}) - y^m(\tau_{j(\bar m-1)}) ) 
\cdots
\bar \cS(y^m(\tau_{j(1)}) - y^m(\tau_{j(0)}))u_0 =u^m(T);
\end{multline}
so to verify~\eqref{eq:ormSol},
it suffices to show that for all $k \in \{0,1,\ldots, \bar m-1\}$,
\begin{multline}\label{eq:equivGoal}
\bar \cS(\Delta \hat y_{j(k+1) -1}^m) \ldots \bar \cS(\Delta \hat y_{j(k)}^m)
u^m(\tau_{j(k)}) 
\\  =\bar \cS( y^m(\tau_{j(k+1)}) -  y^m(\tau_{j(k)})) u^m(\tau_{j(k)}).
\end{multline}

For $k \in \{0,1,\ldots,\bar m-1\}$ such that $j(k+1) -j(k) =1$, 
\eqref{eq:equivGoal} holds by~\eqref{eq:pointEquality}.

Assume next that $k \in \{0,1,\ldots,\bar m-1\}$ is such that $j(k+1) -
j(k) \ge  2$.
If $k = \bar m -1$ and $\tau_{j(k+1)} = T \notin A^\pm[z^m]$, then 
\[
z^m([\tau_{j(k)}, T]) \subset [M^{-}[z^m](\tau_{j(k)}),
M^{+}[z^m](\tau_{j(k)})].
\] 
It follows that $M^+[z^m], M^{-}[z^m], \bar A^{+}[z^m]$ and $\bar A^-[z^m]$
are constant over $[\tau_{j(\bar m -1)},T)$, which implies that $\hat y^m$ is constant
over $[\tau_{j(k)},T)$.
Hence  
$\hat y^m(\tau_{j(\bar m)-1}) = \hat y^m(\tau_{j(\bar m-1)}) $,
\begin{equation}\label{eq:constantOrm1}
\Delta \hat y_{j(\bar m) -\ell}^m = 0 \quad \forall \ell \in
\{2,\ldots, j(\bar m) - j(\bar m-1)\}, 
\end{equation}
and since $\bar \cS(0) = I$, 
\begin{align*}
&\bar \cS(\Delta \hat y_{j(\bar m) -1}^m) 
\cdots \bar \cS(\Delta \hat y_{j(\bar  m-1)}^m)
u^m(\tau_{j(\bar m -1)})
\\ & \qquad = \bar \cS(\Delta \hat y_{j(\bar m) -1}^m)
u^m(\tau_{j(\bar m-1)})
\\ & \qquad =  \bar \cS(y^m(\tau_{j(\bar m)}) - y^m(\tau_{j(\bar m-1)}))
u^m(\tau_{j(\bar m -1)}).
\end{align*}

Otherwise, if $k \in \{0,1,\ldots,\bar m-1\}$ is such that $j(k+1) -
j(k) \ge  2$ and $\tau_{j(k+1)} \in A^\pm[z]$, then 
we recall from the proof of Theorem~\ref{thm:ormLinEquiv}
that there exists a unique $\tau \in
(\tau_{j(k+1)-1}, \tau_{j(k+1)}]$ such that 
\[
z^m([\tau_{j(k)}, \tau]) \subset
[M^-[z^m](\tau_{j(k)}), M^+[z^m](\tau_{j(k)})], \quad 
\tau \in \{\bar A^+(\tau), \bar A^-(\tau)\}.
\]
This implies that $\hat y^m(\tau) 
= z^m(\tau)$, cf.~\eqref{eq:ormMaxRelation}, and that
$ M^+[z^m], M^{-}[z^m], \bar A^{+}[z^m]$, $\bar A^-[z^m]$ are constant over the
interval $[\tau_{j(k)},\tau)$. Consequently, $\hat y^m$ is constant
over $[\tau_{j(k)},\tau)$, and since $\tau_{j(k+1)-1}< \tau$,
$\hat y^m(\tau_{j(k+1)-1}) = \hat y^m(\tau_{j(k)})$ and
\begin{equation}\label{eq:constantOrm2}
\Delta \hat y_{j(k+1) -\ell}^m = 0 \quad \forall \ell \in
\{2,\ldots, j(k+1) - j(k)\}.
\end{equation}
Hence,
\[
\begin{split}
\bar \cS(\Delta \hat y_{j(k+1) -1}^m) \cdots \bar \cS(\Delta \hat y_{j(k)}^m)
u^m(\tau_{j(k)}) &= \bar \cS(\Delta \hat y_{j(k+1) -1}^m) u^m(\tau_{j(k)})\\
& =  \bar \cS(y^m(\tau_{j(k+1)}) - y^m(\tau_{j(k)})) u^m(\tau_{j(k)}).
\end{split}
\]
\end{proof}

A direct consequence of the preceding proof is that the total
variation of $\orm[\cI^m[z]]$ equals the total variation of 
$\overline{\orm}_m[z]$.

\begin{lemma}\label{lem:bvEquiv}
For any $z \in C_0([0,T])$ and points 
$0=\tau_0< \tau_1 < \ldots < \tau_m =T$, $m\ge 2$, 
it holds that
\[
\abs{\orm[\cI[z](\cdot; \{\tau_{k}\}_{k=0}^m) ]}_{\BV{[0,T]}} 
= \abs{\overline{\orm}[z](\cdot; \{\tau_{k}\}_{k=0}^m)}_{\BV{[0,T]}}.
\]
\end{lemma}
\begin{proof}
Recall that $y^m = \overline{\orm}_m[z] \in I_0([0,T];
\{\tau_k\}_{k=0}^m)$ is piecewise linear, with interpolation points 
$0 = \tau_{j(0)} < \tau_{j(1)} < \ldots < \tau_{j(\bar m)} = T$.
Hence, $y^m$ is monotone over each interval $[\tau_{j(k)}, \tau_{j(k+1)}]$.
Since $z^m = \cI^m[z]$ is linear over every
interval $[\tau_k, \tau_{k+1}]$, it follows by
Definition~\ref{def:orm} that $\hat y^m = \orm[z^m]$ is 
monotone over every interval $[\tau_k, \tau_{k+1}]$,
and by the proof of Theorem~\ref{thm:ormEquiv}, for every 
interval $[\tau_{j(k)}, \tau_{j(k+1)}]$ with
$j(k+1)-j(k) \ge 2$, it holds that $\hat y^m$ is constant over
$[\tau_{j(k)}, \tau_{j(k+1)-1}]$. Consequently, $\hat y^m$ is monotone
over each interval $[\tau_{j(k)}, \tau_{j(k+1)}]$,
whereby
\[
\abs{\hat y^m}_{\BV{[\tau_{j(k)}, \tau_{j(k+1)}]}} 
= \abs{\hat y^m(\tau_{j(k+1)}) - \hat y^m(\tau_{j(k)})}
\quad \forall k \in \{0,1,\ldots, \bar m-1 \}.
\]
By~\eqref{eq:pointEquality}, 
\[
\begin{split}
\abs{\hat y^m}_{\BV{[0,T]}}
&= \sum_{k=0}^{\bar m-1} \abs{\hat y^m}_{\BV{[\tau_{j(k)}, \tau_{j(k+1)}]}} \\
&= \sum_{k=0}^{\bar m-1} \abs{\hat y^m(\tau_{j(k+1)}) - \hat y^m(\tau_{j(k)})}\\
&= \sum_{k=0}^{\bar m-1} \abs{ y^m(\tau_{j(k+1)}) - y^m(\tau_{j(k)})}
= \abs{y^m}_{\BV{[0,T]}}.
\end{split}
\]
\end{proof}

\begin{definition}
For any mesh $\{\tau_{k}\}_{k=0}^m$ such that 
\[
0=\tau_0 < \tau_1 < \cdots < \tau_m = T,
\]
and $g \in \cD([0,T])$ we define the total variation of $g$ restricted
to $\{\tau_{k}\}_{k=0}^m$ by
\[
TV(g; \{\tau_{k}\}_{k=0}^m) \coloneq \sum_{k=0}^{m-1} |g(\tau_{k+1}) - g(\tau_{k})|.
\]
\end{definition}

The next lemma shows that the approximation 
$v(T;\orm[z], \{\tau_j\}_{j=0}^m)$ converges to a limit in $L^1(\bR)$ 
as $m\to \infty$, and it provides 
an upper bound for the approximation error.
The upper bound depends on $\orm[z]$, the mesh $\{\tau_k\}_{k=0}^m$,
$|u_0|_{\BV{(\bR)}}$ and $\|f'\|_\infty$, hence it
differs from the stability result~\eqref{eq:stabilityPes}.

\begin{lemma}\label{lem:trueOrmConvBound}
Let $\{ \{ \tau^m_{k} \}_{k=0}^{\ell_m} \}_{m=0}^\infty$ denote a sequence
of nested meshes fulfilling $\ell_0\ge 2$, and for all $m\ge0$,
\[
0=\tau^m_0 < \tau^m_1 < \cdots < \tau^m_{\ell_m} = T,  \quad
\text{with} \quad \ell_m < \ell_{m+1} \le
2\ell _m,
\]
and the nested mesh property
\[
\{\tau^m_{k} \}_{k=0}^{\ell_m} \subset \{ \tau^{m+1}_{k}
\}_{k=0}^{\ell_{m+1}}
\]
with a strictly increasing mapping $j^m:\{0,1,\ldots,\ell_m\} \to \{0,1,\ldots,\ell_{m+1}\}$
such that
\[
\tau_{j^m(k) }^{m+1}  = \tau^m_k,  \quad \forall k \in \{0,1,\ldots,\ell_m\},
\]
and 
\[
1\le j^m(k+1) - j^m(k) \le 2, \quad \forall k \in \{0,1,\ldots, \ell_m-1\}.
\]
Assume $f\in C^2(\bR)$ is strictly convex, 
$u_0 \in (L^1\cap BV)(\bR)$, $z \in C_0([0,T])$. 
Then, for any $\widehat m, \overline m \in \bN \cup \{0\}$
with $\widehat m \ge \overline m$,
\begin{equation}\label{eq:firstTvVm}
\begin{split}
&\norm{v(T; \orm[z], \{\tau_k^{\widehat m}\}_{k=0}^{\ell_{\widehat m}})  -
v(T; \orm[z],\{\tau_k^{\overline m}\}_{k=0}^{\ell_{\overline m}})}_{L^1(\R)} \\
&\qquad \le 2 \|f'\|_{\infty} |u_0|_{\BV{\bR}} \prt{
  TV(\orm[z];\{\tau_k^{\widehat m}\}_{k=0}^{\ell_{\widehat m}})  - 
  TV(\orm[z];\{\tau_k^{\overline m}\}_{k=0}^{\ell_{\overline m}}) }.
\end{split}
\end{equation}
Moreover, if
\[
\abs{\orm[z]}_{\BV{[0,T]}}  < \infty,
\]
and
\begin{equation}\label{eq:bvAssumption}
\lim_{m \to \infty } \prt{ \abs{\orm[z]}_{\BV{[0,T]}} 
- TV(\orm[z]; \{\tau_k^m\}_{k=0}^{\ell_m})}= 0,
\end{equation}
then 
\[
v_{\orm}(T)  \coloneq \lim_{m \to \infty} v(T; \orm[z],
\{\tau_{k}^m\}_{k=0}^{\ell_m}) \in L^1(\bR)
\]
and
\begin{equation}
\begin{split}
&\|v_{\orm}(T)  - v(T; \orm[z],\{\tau_k^m\}_{k=0}^{\ell_m}) \|_1 \\
&\qquad \le 2 \|f'\|_{\infty} |u_0|_{\BV{\bR}} \prt{
  \abs{\orm[z]}_{\BV{[0,T]}} - TV(\orm[z]; \{\tau_k^m\}_{k=0}^{\ell_m})}.
\end{split}
\end{equation}
\end{lemma}
\begin{proof}
For a fixed $m \ge 0$, set $y = \orm[z]$, $\Delta y^m_k =
y(\tau^m_{k+1}) - y(\tau^m_k)$,  and introduce the shorthand 
\[
v^m(\tau_k^{m}) = v(\tau_{k}^m; y,\{\tau_{k}^{m+1}\}_{k=0}^{\ell_{m}}),
\quad \text{for all } k \in \{0,1,\ldots, \ell_m\}.
\]
We have that 
\begin{equation}\label{eq:vMError}
\begin{split}
& \|v^{m+1}(T)  - v^m(T) \|_{L^1(\R)}\\
 & \quad = \norm{\bar \cS(\Delta y^{m+1}_{\ell_{m+1} -1} ) \cdots
\bar \cS(\Delta y^{m+1}_{0} )u_0 - \bar \cS(\Delta y^{m}_{\ell_{m} -1} ) 
\cdots \bar \cS(\Delta y^{m}_{0} ) u_0}_{L^1(\R)}\\
& \quad \le \sum_{k \in J^m_2 } \|\Delta \bar \cS_k  u_0\|_{L^1(\R)},
\end{split}
\end{equation}
where $J^m_2 = \{k \in \{0,1,\ldots, \ell_m-1\}\mid j^m(k+1) - j^m(k) = 2\}$ and 
\[
\begin{split}
&\Delta \bar \cS_k u_0 \coloneq \\
&\begin{cases}
0 & \text{if }  k\notin J^m_2,\\
\!\begin{aligned} & \bar \cS(\Delta y^{m+1}_{\ell_{m+1} -1} ) \cdots \bar
  \cS(\Delta y^{m+1}_{j^m(k+1)}) \\
& \prt{\bar \cS(\Delta y^{m+1}_{j^m(k)+1}) \bar \cS(\Delta y^{m+1}_{j^m(k)} )
- \bar \cS(\Delta y^m_k) } v^m(\tau^m_{k}) 
\end{aligned} & \text{if }  k \in J^m_2 \cap [0, \ell_m-2],\\
\prt{\bar \cS(\Delta y^{m+1}_{\ell_{m+1} -1} ) \bar \cS(\Delta
y^{m+1}_{\ell_{m+1} -2} ) - \bar \cS(\Delta y^m_{\ell_m-1})}
v^m(\tau^m_{\ell_m-1}) & \text{if } k \in J^m_2 \cap \{ \ell_{m}-1\}.
\end{cases}
\end{split} 
\]

Assume that $k \in J^m_2$. Using~\eqref{eq:stabOp}, we obtain 
the following bound
\[
\begin{split}
\|\Delta \bar \cS_k u_0\|_{L^1(\R)} 
& \le 
\norm{ \prt{\bar \cS(\Delta y^{m+1}_{j^m(k)+1} ) \bar \cS(\Delta y^{m+1}_{j^m(k)} )
- \bar \cS(\Delta y^m_k) }  v^m(\tau^m_{k}) }_{L^1(\R)}.
\end{split}
\]
We recall from~\eqref{eq:map3} that 
$\bar \cS(\Delta y^{m+1}_{j^m(k)+1}) \bar \cS(\Delta y^{m+1}_{j^m(k)}) v^m(\tau^m_{k})$
is recursively defined by being the solution at time $t = |\Delta y^{m+1}_{j^m(k)+1}|$ of 
\[
\partial_t \tilde u + \sign{\Delta y^{m+1}_{j^m(k)+1}} \partial_x f(\tilde u) = 0,
\quad t >0, \quad \tilde u(0) = \bar \cS(\Delta y^{m+1}_{j^m(k)}) v^m(\tau^m_{k}).
\]
If $\Delta y^{m+1}_{j^m(k)+1} \Delta y^{m+1}_{j^m(k)} \ge 0$, then
\begin{equation}\label{eq:sameSignCancel}
\begin{split}
\bar \cS(\Delta y^{m+1}_{j^m(k)+1}) \bar \cS(\Delta y^{m+1}_{j^m(k)})
v^m(\tau^m_{k})
&= \bar \cS(\Delta y^{m+1}_{j^m(k)+1} + \Delta y^{m+1}_{j^m(k)})
v^m(\tau^m_{k})\\
&= \bar \cS(\Delta y^{m}_{k})v^m(\tau^m_{k}).
\end{split}
\end{equation}
This yields
\begin{equation}\label{eq:dsZero}
\|\Delta \bar \cS_k u_0\|_{L^1(\R)} =0,  
\quad   \forall k \in \{\bar k \in  J^m_2 \mid
\Delta y^{m+1}_{j^m(\bar k)+1} \Delta y^{m+1}_{j^m(\bar k)} \ge 0 \}.
\end{equation}
Consider next the case $\Delta y^{m+1}_{j^m(k)+1} \Delta y^{m+1}_{j^m(k)} < 0$.
In view of Definition~\ref{def:orm},
$y(\tau^{m}_{k+1}) \in \{ M^+[z](\tau^{m}_{k+1}),
M^-[z](\tau^{m}_{k+1})\}$.
If $y(\tau^{m}_{k+1}) = M^+[z](\tau^{m}_{k+1})$, then
\[
y(\tau^{m+1}_{j^m(k)+1)}) \le  y(\tau^{m}_{k+1}) = y(\tau^{m+1}_{j^m(k+1))}),
\]
while if $y(\tau^{m}_{k+1}) = M^-[z](\tau^{m}_{k+1})$, then
\[
y(\tau^{m+1}_{j^m(k)+1}) \ge y(\tau^{m}_{k+1}) = y(\tau^{m+1}_{j^m(k+1))}).
\]
We conclude that regardless of whether $y(\tau^{m}_{k+1}) =
M^+[z](\tau^{m}_{k+1})$ or $y(\tau^{m}_{k+1}) = M^-[z](\tau^{m}_{k+1})$,
it holds that
\begin{equation}\label{eq:dySign}
\Delta y^{m+1}_{j^m(k)+1} \Delta y^{m}_{k}  = \prt{y(\tau^m_{k+1}) -
y(\tau^{m+1}_{j^m(k)+1})} \prt{y(\tau^m_{k+1}) - y(\tau^m_{k})} \ge 0.
\end{equation}
And since 
\[
\sign{\Delta y^{m}_{k} \prt{-\Delta y^{m+1}_{j^m(k)}}} =
\sign{\Delta y^{m}_{k} \Delta y^{m+1}_{j^m(k+1)}} \ge 0,
\]
we also have that
\[
\Delta y^{m}_{k} \prt{-\Delta y^{m+1}_{j^m(k)}} \ge 0.
\]
Using the relationship
\[
\Delta y^{m+1}_{j^m(k)+1}  = \Delta y^{m}_{k} - \Delta y^{m+1}_{j^m(k)} 
\]
we obtain
\[
\begin{split}
  &\bar \cS(\Delta y^{m}_{k}) \bar \cS\prt{-\Delta y^{m+1}_{j^m(k)}}
  \bar \cS\prt{\Delta y^{m+1}_{j^m(k)}} v^m(\tau^m_{k})\\
  &\qquad \qquad = \bar \cS\prt{\Delta y^{m}_{k} - \Delta y^{m+1}_{j^m(k)} } \bar \cS\prt{\Delta y^{m+1}_{j^m(k)}} v^m(\tau^m_{k}) \\
&\qquad \qquad = \bar \cS\prt{\Delta y^{m+1}_{j^m(k)+1}} \bar \cS\prt{\Delta y^{m+1}_{j^m(k)}}
v^m(\tau^m_{k}).
\end{split}
\]
By~\eqref{eq:stabOp},~\eqref{eq:stabOp2} and~\eqref{eq:stabOp3}, we
derive the following bound for all $k \in J^m_2$ such that $\Delta y^{m+1}_{j^m(k)+1} \Delta
y^{m+1}_{j^m(k)} < 0$, 
\begin{equation}\label{eq:dsNonZero}
\begin{split}
\|\Delta \bar \cS_k u_0\|_{L^1(\R)} &\le 
\norm{ \bar \cS(\Delta y^{m}_{k}) \prt{ \bar \cS(-\Delta y^{m+1}_{j^m(k)}) \bar
  \cS(\Delta y^{m+1}_{j^m(k)}) - I } v^m(\tau^m_{k})}_{L^1(\R)}\\
& \le \norm{\prt{ \bar \cS(-\Delta y^{m+1}_{j^m(k)}) \bar
  \cS(\Delta y^{m+1}_{j^m(k)}) - I } v^m(\tau^m_{k})}_{L^1(\R)} \\
& \le \norm{\prt{ \bar \cS(-\Delta y^{m+1}_{j^m(k)}) -I} \bar
  \cS(\Delta y^{m+1}_{j^m(k)})  v^m(\tau^m_{k})}_{L^1(\R)} \\
& \qquad +  \norm{\prt{ \bar \cS(\Delta y^{m+1}_{j^m(k)})-I} v^m(\tau^m_{k})}_{L^1(\R)} \\
& \le 2\|f'\|_\infty |u_0|_{\BV{\bR}} |\Delta y^{m+1}_{j^m(k)}|.
\end{split}
\end{equation}
By~\eqref{eq:vMError},~\eqref{eq:dsZero} and~\eqref{eq:dsNonZero},
\begin{equation}\label{eq:vMBvBound}
\begin{split}
\|v^{m+1}(T)  - v^m(T) \|_{L^1(\R)} & 
\le \sum_{k \in J^m_2 } 2\|f'\|_\infty |u_0|_{\BV{\bR}} |\Delta y^{m+1}_{j^m(k)}|
\\ & \le 2\|f'\|_\infty |u_0|_{\BV{\bR}} \Bigl( \sum_{k =0}^{\ell_{m+1}}
  |\Delta y^{m+1}_{k}| - \sum_{k =0}^{\ell_{m}}
  |\Delta y^{m}_{k}|\Bigr).
\end{split}
\end{equation}
For $m_2>m_1$ we get
\begin{align*}
  \norm{v^{m_2}(T)-v^{m_1}(T)}_{L^1(\bR)} &\le \sum_{m=m_1}^{m_2-1}
  \norm{v^{m+1}(T)-v^{m}(T)}_{L^1(\bR)}\\
  &\le 2\|f'\|_\infty |u_0|_{\BV{\bR}} \sum_{m=m_1}^{m_2-1}\Bigl( \sum_{k =0}^{\ell_{m+1}}
  |\Delta y^{m+1}_{k}| - \sum_{k =0}^{\ell_{m}}
  |\Delta y^{m}_{k}|\Bigr)\\
  &=2\|f'\|_\infty |u_0|_{\BV{\bR}} \Bigr(\sum_{k =0}^{\ell_{m_2+1}}
  |\Delta y^{m_2+1}_{k}| - \sum_{k =0}^{\ell_{m_1}}
  |\Delta y^{m_1}_{k}|\Bigr)
\end{align*}
By assumption~\eqref{eq:bvAssumption}, both sums on the right converge
to $\abs{\orm[z]}_{\BV{[0,T]}}$, thus $\{v^m(T)\}_m \subset
L^1(\bR)$ is a Cauchy sequence, and,
\[
\lim_{m \to \infty} v^m(T) = v_{\orm}(T) \in L^1(\bR).
\]
Moreover, by~\eqref{eq:vMBvBound},~\eqref{eq:bvAssumption}
and a telescoping sum argument we obtain
\[
\|v_{\orm}(T) - v^m(T)\|_{L^1(\R)}\le 2\|f'\|_\infty |u_0|_{\BV{\bR}} \Bigl( \abs{\orm[z]}_{\BV{[0,T]}} 
- \sum_{k =0}^{\ell_{ m}} |\Delta y^{ m}_{k}| \Bigr).
\]
Inequality~\eqref{eq:firstTvVm} can be proved using a similar telescoping sum
argument.
\end{proof}

It remains to verify that $v_{\orm}(T)$, under some assumptions, is
equal to the unique pathwise entropy 
solution of~\eqref{eq:sscl} at time $t=T$. 
\begin{theorem}\label{thm:limOrmBound}
Let $u$ denote the unique pathwise entropy solution of~\eqref{eq:sscl} with
initial data $u_0 \in (L^1 \cap BV)(\bR)$, strictly convex flux $f \in
C^2(\bR)$ and driving path $z \in C_0([0,T])$. Assume that 
$\abs{\orm[z]}_{\BV{\bR}} < \infty$ and for 
some $m\ge2$, let $\{ \tau_{j=0} \}_{j=0}^m$ denote a mesh
satisfying
\[
0=\tau_0 < \tau_1 < \ldots < \tau_{m} = T.
\]
Then 
\[
\begin{split}
&\|u(T) - v(T;\orm[z], \{\tau_k\}_{k=0}^m)\|_{L^1(\R)}\\
& \qquad \qquad \le 
2\|f'\|_\infty |u_0|_{\BV{\bR}} \prt{ \abs{\orm[z]}_{\BV{[0,T]}} -
  TV(\orm[z]; \{\tau_{k}\}_{k=0}^m ) }.
\end{split}
\]
\end{theorem}
\begin{proof}
Since $\abs{\orm[z]}_{BV([0,T])}< \infty$, there is a sequence of
meshes $\{\{\tau^r_k\}_{k=1}^{\ell_r}\}_{r=0}^\infty$ such that
\begin{equation*}
  \lim_{r\to \infty} TV(\orm[z]; \{\tau_{k}^r\}_{k=0}^{\ell_r}) =  \abs{\orm[z]}_{\BV{[0,T]}}.
\end{equation*}
If necessary, we can add meshpoints to get \emph{nested} meshes
$\{\{\hat \tau_{k}^{r}\}_{k=0}^{\hat \ell_{r}}\}_{r=0}^{\infty}$ fulfilling 
that $\{\hat \tau_{k}^0\}_{k=0}^{\hat \ell_0} = \{ \tau_k\}_{k=0}^m$, 
$\hat \ell_{r} < \hat \ell_{r+1}$ for all $r\ge 0$, 
\[
\lim_{r\to \infty} \max_{k \in \{0,1,\ldots, \ell_{
    m}-1\}} \hat \tau^{r}_{k+1} -\hat \tau_{k}^{r} = 0,
\]
and 
\[
\lim_{r\to \infty} TV(\orm[z]; \{\hat \tau_{k}^{r}\}_{k=0}^{\hat \ell_{r}} ) 
= \abs{\orm[z]}_{BV([0,T])}.
\]
Let 
\begin{equation}\label{eq:aPmOrmDef}
A^{\pm}_{r}[z]  \coloneq \bar A^+[z]( \{\hat \tau_{k}^{r}\}_{k=0}^{\hat \ell_r} ) 
\cup \bar  A^-[z]( \{\hat \tau_{k}^{r}\}_{k=0}^{\hat \ell_{r}} ),
 \quad r \ge 0.
\end{equation}
As $A^{\pm}_{ r}[z] \subset A^{\pm}_{r+1}[z]$ for
all $r\ge 0$, we may construct a new sequence of nested meshes
$\{\{ \tau_{k}^{r}\}_{k=0}^{\ell_{r}}\}_{r=0}^{\infty}$ defined by
\begin{equation}\label{eq:tauMeshDef}
\{ \tau_{k}^{r}\}_{k=0}^{\ell_{r}} =
\begin{cases}
\{ \hat \tau_{k}^{0}\}_{k=0}^{ \hat \ell_{0}}  = \{\tau_k\}_{k=0}^m, & \text{if
} r
=0,\\
\{ \hat \tau_{k}^{r-1}\}_{k=0}^{ \hat \ell_{r-1}} \cup
A^{\pm}_{r-1}[z], & \text{if } r \ge 1.
\end{cases}
\end{equation}
Since $\{ \tau_{k}^{r}\}_{k=0}^{\ell_{r}}
\supset \{\hat \tau_{k}^{r}\}_{k=0}^{\hat \ell_{r}}$  for all
$r\ge 0$, it also holds that 
\begin{equation}\label{eq:dTauLim}
\lim_{r\to \infty} 
\max_{k \in \{0,1,\ldots, \ell_{r}-1\}} \tau^{r}_{k+1} -\tau_{k}^{r} = 0,
\end{equation}
and 
\[
\lim_{r\to \infty} TV(\orm[z]; \{\tau_{k}^{r}\}_{k=0}^{\ell_{r}}) 
= \abs{\orm[z]}_{BV([0,T])}.
\]

For any $r \ge 1$, we have that
\begin{equation}\label{eq:vRIAndIIBound}
\begin{split}
& \|v(T;\orm[z], \{\tau_k\}_{k=0}^m)- u(T) \|_{L^1(\R)} 
\\ & \qquad 
\le \|v(T;\orm[z], \{\tau_k\}_{k=0}^m) 
- v(T;\orm[z], \{ \tau_k^{r}\}_{k=0}^{\ell_{r}}) \|_{L^1(\R)}
\\ & \qquad \qquad
+ \| v(T;\orm[z], \{ \tau_k^{r}\}_{k=0}^{\ell_{r}}) - u(T) \|_{L^1(\R)}
\\ & \qquad \eqcolon \mathrm{I}_{r} + \mathrm{II}_{r}.
\end{split}
\end{equation}
Lemma~\ref{lem:trueOrmConvBound} implies that for any 
pair of nested meshes $\{\tau_k\}_{k=0}^m \subset \{ \tau_k^{r} \}_{k=0}^{\ell_{r}}$,
\begin{equation}\label{eq:vRIBound}
\begin{split}
\mathrm{I}_{r} &\le   2 \|f'\|_{\infty} |u_0|_{\BV{\bR}} 
\prt{TV(\orm[z];\{\tau_k^{r}\}_{k=0}^{\ell_{r}})  
- TV(\orm[z];\{\tau_k\}_{k=0}^{m}) }\\ & 
\le  2 \|f'\|_{\infty} |u_0|_{\BV{\bR}} 
\prt{ \abs{\orm[z]}_{BV([0,T])} -  TV(\orm[z];\{\tau_k\}_{k=0}^{m}) }.
\end{split}
\end{equation}

To bound $\mathrm{II}_{r}$, note
first that by~\eqref{eq:ormMaxRelation},
\[
\orm[z](\tau_k^{r}) = \begin{cases}
 z(\max(A^{+}[z], A^{-}[z])(\tau_k^{r}) ) & \text{if} \quad k
 \in\{0,1,\ldots, \ell_{r}-1\},\\
z(T) =  z(\tau_{ \ell_{r}}^r)& \text{if} \quad k = \ell_{r}.
\end{cases}
\]
Recalling that $\hat \cS(0) = I$ and that $\max(\bar A^{+}[z],
\bar A^{-}[z])(\cdot)$ is a monotonically increasing function, 
\begin{equation}\label{eq:ormVSol1}
\begin{split}
&v(T; \orm[z], \{\tau_{k}^{r}\}_{k=0}^{\ell_{r}}) \\
& = \bar \cS(\orm[z](\tau^{r}_{\ell_{r}}) - \orm[z](\tau^{r}_{ \ell_{r}-1}) ) \cdots \bar
\cS(\orm[z](\tau^{r}_1) - \orm[z](\tau^{r}_{0}) ) u_0 \\
& = v(T; z, \max(A^{+}[z], A^{-}[z])(\{ \tau_{k}^{r}\}_{k=0}^{\ell_{r}}) \cup \{T\}).
\end{split}
\end{equation}
By~\eqref{eq:fixedPointsA}, 
\[
\max\seq{\bar A^{+}[z], \bar A^{-}[z]}( A^{\pm}_{r}[z] ) = A^{\pm}_{r}[z],
\]
and~\eqref{eq:tauMeshDef} implies that 
\[
\begin{split}
  & \max(\bar A^{+}[z], \bar A^{-}[z])(\{\tau_{k}^{{r}}\}_{k=0}^{\ell_{r}} )  =
  \max(\bar A^{+}[z], \bar A^{-}[z])( A^\pm_{r}[z]  \cup \{\hat \tau_{k}^{{r}}\}_{k=0}^{\hat \ell_{r}} ) \\
&\qquad \qquad \qquad = \max(\bar A^{+}[z], \bar A^{-}[z])( A^\pm_{r}[z] ) \, \cup \,
\max(\bar A^{+}[z], \bar A^{-}[z])(\{\hat \tau_{k}^{r}\}_{k=0}^{\hat \ell_{r}}) \\
&\qquad \qquad  \qquad =A^\pm_{r}[z].
\end{split}
\]
Consequently, 
\[
\begin{split}
  v(T; z, \max(A^{+}[z], A^{-}[z])(\{\tau_{k}^{r}\}_{k=0}^{ \ell_{r}}) \cup \{T\})
  &= v(T; z, A^{\pm}_{r}[z] \cup \{T\})\\
  &=v(T; \tilde z^r, A^{\pm}_{r}[z] \cup \{T\}),
  \end{split}
\]
where $\tilde z^r \coloneq \cI[z](\cdot; A^{\pm}_{r}[z] \cup \{T\})$.
Introducing the function $z^{r} = \cI[z](\cdot; \{\tau_{k}^{r}\}_{k=0}^{\ell_{r}})$,
we may bound the second term as follows
\begin{equation}\label{eq:vRII1AndII2Bound}
\begin{split}
\mathrm{II}_{r} &\le \norm{u(T) - v(T;  z^{r},\{ \tau_{k}^{r}\}_{k=0}^{\ell_{r}}) }_1
+ \norm{v(T; z^{r},\{ \tau_{k}^{r}\}_{k=0}^{\ell_{r}})
  - v(T; \tilde z^{r}, A^{\pm}_{r}[z] \cup \{T\})}_1\\
  & \eqcolon \mathrm{II}_{r,1} + \mathrm{II}_{r,2}.
\end{split}
\end{equation}
Theorem~\ref{thm:wellPosednessPes} and the property $z^{r}(T) = z(T)$ for all $r\ge1$
imply that there exists a constant $C(\|u_0\|_2, \|f''\|_\infty) > 0$
such that
\[
\mathrm{II}_{r,1}  \le C \sqrt{\max_{s \in [0,T]} \abs{z - z^{r}}(s) }, 
\quad \forall r \ge 1.
\]
Since $z$ is uniformly continuous on $[0,T]$, 
and it follows from \eqref{eq:dTauLim} that
\begin{equation}\label{eq:vRII1Bound}
\lim_{r \to \infty} \mathrm{II}_{r,1} = 0.
\end{equation}

The term $\mathrm{II}_{r, 2}$ is bounded by verifying that
\begin{equation}\label{eq:vEquality}
v(T;  z^{r},\{ \tau_{k}^{r}\}_{k=0}^{\ell_{r}})
= v(T; \tilde z^r, A^\pm_r[z] \cup \{T\} ),
\end{equation}
which actually means that $\mathrm{II}_{r,2}=0$.
Note first that since $A^{\pm}_{r}[z] \cup \{T\} \subset \{ \tau_{k}^{r}\}_{k=0}^{\ell_{r}}$,
we may introduce the monotonically increasing function
$h: \{0,1,\ldots, \ell_{r}\} \to \{0,1,\ldots,  \ell_{r}\}$
defined by
\[
h(k)= \begin{cases}
  \{s \in \{0,1,\ldots,k\} \mid \tau_s^{r}
  = \max(\bar A^+[z], \bar A^-[z])(\tau^{r}_k)\} 
  & \text{if } k <  \ell_{r}, \\
  \ell_{r} & \text{if } k = \ell_{r}.
  \end{cases}
\]
and write $\{\tau_{h(k)}\}_{k=0}^{\ell_{r}} =A^{\pm}_{r}[z] \cup \{T\}$.
Using this representation and that
\[
\tilde z^{r}|_{A^{\pm}_{r}[z] \cup \{T\}} = z|_{A^{\pm}_{r}[z] \cup \{T\}} 
=  z^{r}|_{A^{\pm}_{r}[z] \cup \{T\}},
\]
we obtain
\[
v(T; \tilde z^{r}, \{\tau_{h(k)}\}_{k=0}^{\ell_{r}}) =
\bar \cS\prt{ z^{r}( \tau_{h( \ell_{r})}^r) -  z^{r}( \tau_{h( \ell_{r}-1)}^r) } \cdots
\bar \cS\prt{z^{r}( \tau_{h(1)}^r) - z^{r}( \tau_{ h(0)}^r)}u_0.
\]
Recalling that
\[
\begin{split}
v(T; z^{r}, \{\tau_{k}^r\}_{k=0}^{\ell_{r}}) &=
\bar \cS\prt{ z^{r}( \tau_{\ell_{r}}) - z^{r}( \tau_{\ell_{r}-1}) } \cdots
\bar \cS\prt{ z^{r}( \tau_{1}) -  z^{r}( \tau_{0})}u_0,
\end{split}
\] 
equality~\eqref{eq:vEquality} follows by a straightforward induction argument
if the following equality holds
for all $k \in \{0,1,\ldots, \ell_r-1\}$ such that $h(k+1)-h(k) \ge 2$:
\small
\begin{equation}\label{eq:zRCancel}
\begin{split}
&\bar \cS\prt{ z^{r}( \tau_{h(k+1)}^r) -  z^{r}( \tau_{h(k+1)-1}^r) } \cdots
\bar \cS\prt{ z^{r}( \tau_{h(k)+1}^r) -  
z^{r}( \tau_{h(k)}^r) }v(\tau^r_{h(k)};z^r,\{\tau_{h(k)}^r\}_{k=0}^{\ell_r})\\
&= \bar \cS\prt{ z^{r}( \tau_{h(k+1)}^r) - 
z^{r}( \tau_{h(k)}^r) } v(\tau^r_{h(k)};z^r,\{\tau_{h(k)}^r\}_{k=0}^{\ell_r}).
\end{split}
\end{equation}
\normalsize

Assume $k \in \{0,1,\ldots, \ell_r-1\}$ is such that
$h(k+1)-h(k) \ge 2$. Then
\begin{align*}
& \{\bar A^-[z](\tau_{h(k+1)-1}^r), \bar A^+[z](\tau_{h(k+1)-1}^r) \} \subset A^\pm_r[z],
\\ & \max(\bar A^-[z], \bar A^+[z])(\tau_{h(k+1)-1}^r) = \tau_{h(k)}^r,
\end{align*}
which implies that for all $s \in \{h(k), \ldots, h(k+1)-1\}$,
\begin{align*}
& z^r(\tau_{s}^r) \ge M^-[z^r](\tau^r_{h(k+1)-1}) \ge M^-[z](\tau^r_{h(k+1)-1}) 
\\ & \qquad = z(\bar A^-[z](\tau_{h(k+1)-1}^r)) \ge M^-[z^r](\tau_{h(k)}^r)
\end{align*}
and
\begin{align*}
& z^r(\tau_{s}^r) \le M^+[z^r](\tau^r_{h(k+1)-1}) \le M^+[z](\tau^r_{h(k+1)-1}) 
\\ & \qquad 
= z(\bar A^+[z](\tau^r_{h(k+1)-1})) \le M^+[z^r](\tau_{h(k)}^r).
\end{align*}
Consequently,
\begin{equation}\label{eq:zRContained}
z^r([\tau_{h(k)}, \tau_{h(k+1)-1}]) \subset [M^{-}[z^r](\tau_{h(k)}^r), M^{+}[z^r](\tau_{h(k)}^r)].
\end{equation}
Consider the following three cases: 
$z^r(\tau_{h(k+1)}^r) - z^r(\tau_{h(k+1)-1}^r)=0$,
$z^r(\tau_{h(k+1)}^r) - z^r(\tau_{h(k+1)-1}^r)>0$, 
and $z^r(\tau_{h(k+1)}^r) - z^r(\tau_{h(k+1)-1}^r)<0$.

If $z^r(\tau_{h(k+1)}^r) - z^r(\tau_{h(k+1)-1}^r)=0$, then 
\begin{equation}\label{eq:zRContained2}
z^r([\tau_{h(k)}, \tau_{h(k+1)}]) \subset [M^{-}[z^r](\tau_{h(k)}^r), M^{+}[z](\tau_{h(k)}^r)],
\end{equation}
and~\ref{eq:zRCancel} follows from Lemma~\ref{lem:cancel}.

If $z^r(\tau_{h(k+1)}^r) - z^r(\tau_{h(k+1)-1}^r)>0$,
then, since $\tau_{h(k+1)} \in A^{\pm}_r[z] \cup \{T\}$,
either $\tau_{h(k+1)} =T \notin A^{\pm}_r[z]$ or
$\tau_{h(k+1)} \in A^{\pm}_r[z]$. If $\tau_{h(k+1)} =T \notin A^{\pm}_r[z]$,
then $\max(\bar A^{+}[z], \bar A^{-}[z])(T) = \tau_{h(\ell_r -1)}^r$,
and 
\begin{align*}
	& z^r(\tau_{h(\ell_r)}) 
	\ge z(\bar A^-[z](\tau_{h(\ell_r)}^r)) \ge M^-[z^r](\tau_{h(\ell_r-1)}^r),
	\\ &
	z^r(\tau_{h(\ell_r)}) \le 
	z(\bar A^+[z](\tau_{h(\ell_r)}^r)) \le M^+[z^r](\tau_{h(\ell_r-1)}^r).
\end{align*}
Hence~\eqref{eq:zRContained2} holds and~\eqref{eq:zRCancel} follows.
If $z^r(\tau_{h(k+1)}^r) - z^r(\tau_{h(k+1)-1}^r)>0$
and $\tau_{h(k+1)}^r \in A^{\pm}_r[z]$, then for all $t \in (\tau_{h(k+1)-1}^r,\tau_{h(k+1)}^r]$,
\[
\begin{split}
  z(t)&\ge z(\bar A^-(\tau_{h(k+1)-1}^r)) + z^r(t) - z^r(\tau_{h(k+1)-1}^r)\\
  &\ge M^-[z](\tau_{h(k)}^r) + z^r(t) - z^r(\tau_{h(k+1)-1}^r)\\
  & > M^-[z](\tau_{h(k)}^r).
\end{split}
\]
This implies that $\tau_{h(k+1)}^r = \bar A^+[z](\tau_{h(k+1)}^r)$ and
we conclude from
\begin{equation*}
  \begin{split}
  M^+[z^r](\tau_{h(k+1)}^r) &\le M^+[z](\tau_{h(k+1)}^r)\\
  &= z(\bar A^+[z](\tau_{h(k+1)}^r))\\
  &= z(\tau_{h(k+1)}^r)\\
  &\le M^+[z^r](\tau_{h(k+1)}^r),
  \end{split}
\end{equation*}
that $M^+[z^r](\tau_{h(k+1)}^r) = z^r(\tau_{h(k+1)}^r)$. 
Moreover, by~\eqref{eq:zRContained} and $\dot z^r(\tau_{h(k+1)-1}^r+)>0$,
there exists a unique
$\tau \in [\tau_{h(k+1)-1}^r, \tau_{h(k+1)}^r]$ such that
\[
z^r(\tau) = M^+[z^r](\tau_{h(k)}^r) \quad \text{ and } \bar A^+[z](\tau) = \tau.
\]
Consequently,
\[
z^r([\tau_{h(k)}, \tau]) \subset [M^{-}[z^r](\tau_{h(k)}^r), M^{+}[z^r](\tau_{h(k)}^r)],
\]
and Lemma~\ref{lem:cancel} yields
\small
\[
\begin{split}
&\bar \cS\prt{ z^{r}( \tau_{h(k+1)}^r) -  z^{r}( \tau_{h(k+1)-1}^r) } \cdots
\bar \cS\prt{ z^{r}( \tau_{h(k)+1}^r) -  z^{r}( \tau_{h(k)}^r) }
v(\tau^r_{h(k)};z^r,\{\tau_{h(k)}^r\}_{k=0}^{\ell_r})\\
 &= \bar \cS\prt{ z^{r}( \tau_{h(k+1)}^r) - z^r(\tau)} 
 \cS\prt{z^{r}( \tau) - z^r(\tau_{h(k)}^r)}v(\tau^r_{h(k)};z^r,\{\tau_{h(k)}^r\}_{k=0}^{\ell_r})\\
 &= \bar \cS\prt{ z^{r}( \tau_{h(k+1)}^r) -  z^{r}( \tau_{h(k)}^r) } 
 v(\tau^r_{h(k)};z^r,\{\tau_{h(k)}^r\}_{k=0}^{\ell_r}), 
\end{split}
\]
\normalsize
where the last equality follows from 
\begin{multline*}
\prt{z^r(\tau_{h(k+1)}^r) - z^r(\tau)}\prt{z^r(\tau) - z^r(\tau_{h(k)}^r)}\\
= \prt{M^+[z^r](\tau_{h(k+1)}^r) 
- M^+[z^r](\tau_{h(k)}^r)} \prt{M^+[z^r](\tau_{h(k)}^r) - z^r(\tau_{h(k)}^r)} \ge 0
\end{multline*}
and the argument preceding~\eqref{eq:sameSignCancel}.

Verifying~\eqref{eq:zRCancel} for the case
$z^r(\tau_{h(k+1)}^r) - z^r(\tau_{h(k+1)-1}^r)<0$ may be done similarly.
This yields
\[
\mathrm{II}_{r,2} =0 \qquad \forall r \ge 1,
\]
and by~\eqref{eq:vRIAndIIBound},~\eqref{eq:vRIBound},~\eqref{eq:vRII1AndII2Bound}
and~\eqref{eq:vRII1Bound},
\[
\begin{split}
&\|v(T;\orm[z], \{\tau_k\}_{k=0}^m)- u(T) \|_1 
\\ & \quad \le \lim_{r\to \infty} \prt{\mathrm{I}_r + \mathrm{II}_{r,2}}
\\ & \quad \le  2 \|f'\|_{\infty} |u_0|_{\BV{\bR}} 
\prt{ \abs{\orm[z]}_{BV(\bR)} -  TV(\orm[z];\{\tau_k\}_{k=0}^{m}) }.
\end{split}
\]
\end{proof}

For strictly convex fluxes $f \in C^2(\bR)$ and paths $z \in
C_0([0,T])$ with $\orm[z] \in BV([0,T])$,
Theorem~\ref{thm:limOrmBound} provides an error bound for
approximations of pathwise entropy solutions, which is 
different from~\eqref{eq:stabilityPes} and shows a link
between the driving path $z$ and $\orm[z]$. If one were to consider
an extension of \eqref{eq:sscl} with
$\orm[z]$ given as input, then Theorem~\ref{thm:limOrmBound} could be used as basis for a
numerical method that (numerically) solves
\[
v(T;\orm[z], \{\tau_k\}_{k=0}^m).
\]
This would differ from the methods we propose herein, i.e.,
either numerically finding an approximation to 
\[
v(T;z^m, \{\tau_k\}_{k=0}^m) \quad \text{or} 
\quad v(T;\overline{\orm}_m[z], \{\tau_k\}_{k=0}^m),
\]
see the next section for a description of the latter approach.

\subsection{Improved numerical method}\label{sec:framework2}
Theorem~\ref{thm:ormLinEquiv} shows that for strictly convex $f$ 
and $z\in C_0([0,T])$
the entropy solution of~\eqref{eq:consClass} at time $t=T$ with driving path
$z^m =\cI^m[z]$ can be computed by replacing $z^m$ with $y^m = \overline{\orm}_m[z]$
and solving $v(T;y^m, \{\tau_{j(k)}\}_{k=0}^{\bar m})$, cf.~\eqref{eq:solRep3}
(or, equivalently, using $\orm[z^m]$, see Theorem~\ref{thm:ormEquiv}).
One may view $\overline{\orm}_m[z]$ as a version 
of $z^m$ with type (iii) ``oscillatory 
cancellations'' (cf.~ Lemma~\ref{lem:cancel}) removed. 
A further implication of the lemma is that we may also 
remove type (i) and (ii) ``oscillatory cancellations''
from $\overline{\orm}_m[z]$ and still preserve the entropy solution at final time $t=T$.
The following algorithm describes the removal procedure:
\begin{algorithm}[H]
 \caption{Removal of  Lemma~\ref{lem:cancel} type (i) and (ii) ``oscillatory cancellations''}
  \begin{algorithmic}\label{alg:removeAlg}
  \STATE{{\bf Input:} $y^m = \overline{\orm}_m[z]$ and mesh $\{\tau_{j(k)}\}_{k=0}^{\bar m}$.}
  \STATE{{\bf Output:} Reduced mesh $\{ \tilde \tau_{k}\}_{k=0}^{\tilde L(m)} \subset \{\tau_{j(k)}\}_{k=0}^{\bar m}$
    and $\tilde y^m = \cI[y^m]\prt{\cdot; \{ \tilde \tau_{k}\}_{k=0}^{\tilde L(m)}}$.}
  \medskip
  
  \STATE{Set $k=0$ and $\tilde \tau_k = \tau_{j(k)} = 0$.}
  \WHILE{$\tilde \tau_k<T$}
  \STATE{Compute 
    \begin{equation}\label{eq:tauPM}
      \begin{split}
    \tau^+_k &= \max\{t \in \{\tau_{j(i)}\}_{i=0}^{\bar m} \cap [\tilde \tau_k,T] \mid
    \dot y^m(s+) \ge 0 \quad \forall s \in [\tilde \tau_k,t) \},\\
    \tau^-_k &= \max\{t \in \{\tau_{j(i)}\}_{i=0}^{\bar m} \cap [\tilde \tau_k,T] \mid \dot y^m(s+) \le 0 \quad \forall s \in [\tilde \tau_k,t) \},
      \end{split}
    \end{equation}
    
    set 
    \[
    \tilde \tau_{k+1} = \max(\tau^+_k, \tau^-_k),
    \]
    and $k=k+1$.
    }
  \ENDWHILE
  \STATE{Set $\tilde L(m) = k$.}

  \RETURN{$\{ \tilde \tau_{j}\}_{j=0}^{\tilde L(m)}$
    and $\tilde y^m = \cI[y^m](\cdot; \{ \tilde \tau_{j}\}_{j=0}^{\tilde L(m)})$.}
  \end{algorithmic}
\end{algorithm}
Figure~\ref{fig:ormFullCancel} illustrates the 
transition from $y^m = \overline{\orm}_m[z]$
to $\tilde y^m$ computed by Algorithm~\ref{alg:removeAlg}.
\begin{figure}[h!]
\includegraphics[width=0.68\textwidth]{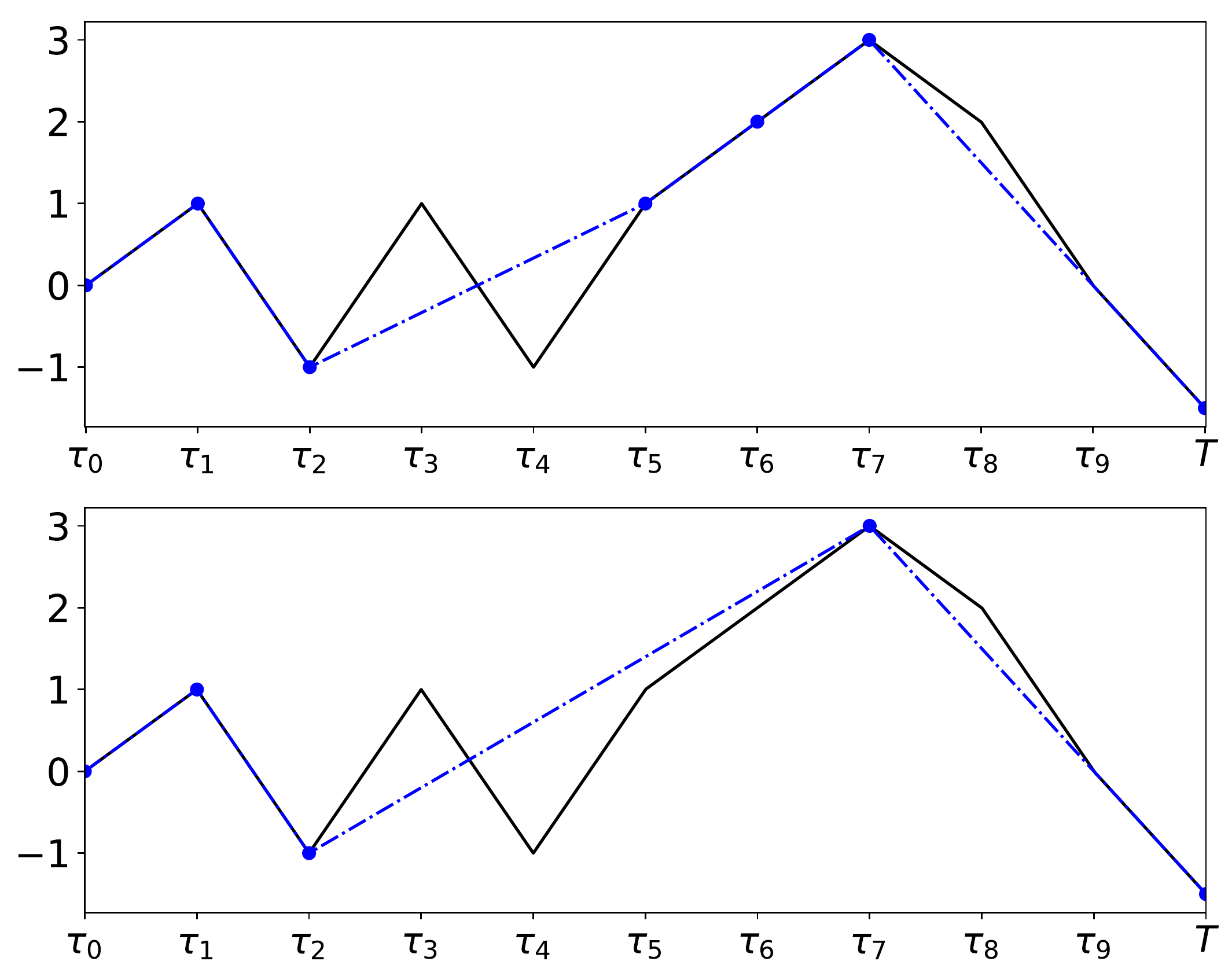}
\caption{Top: The piecewise linear path $z^m$ with $m=10$ (black line) and the
  associated $y^m = \overline{\orm}_m[z]$ (blue dash-dotted line). Blue
  dots mark the value of $\overline{\orm}_m[z]$ at its interpolation points 
  $\{\tau_{j(k)}\}_{k=0}^{\bar m}$.
  Bottom: The piecewise linear path $z^m$ (black line) and the
  associated $\tilde y^m$ computed by 
  Algorithm~\ref{alg:removeAlg} (blue dash-dotted line). Blue
  dots mark the value of $\tilde y^m$ at its interpolation points 
  $\{\tilde \tau_{k}\}_{k=0}^{\tilde L(m)}$.
}
\label{fig:ormFullCancel}
\end{figure}

For later reference, note that the output 
mesh of Algorithm~\ref{alg:removeAlg} satisfies
$\{\tilde \tau_k\}_{k=0}^{\tilde L(m)-1} \subset A^\pm[z^m]$,
and, using that for all $k < \tilde L(m)$,
\[
y^m(\bar A^+[z^m](\tilde \tau_k)) = M^{+}[z^m](\tilde \tau_k)
\quad \text{and} \quad 
y^m(\bar A^-[z^m](\tilde \tau_k)) = M^{-}[z^m](\tilde \tau_k),
\]
cf.~Definition~\eqref{eq:yDef} and~\eqref{eq:MPmRelation},
it follows that 
$\{ y^m(\bar A^+[z^m](\tilde \tau_k))\}_{k=0}^{\tilde L(m)-1}$
and\\
$\{y^m(\bar A^-[z^m](\tilde \tau_k))\}_{k=0}^{\tilde L(m)-1}$ respectively
are monotonically increasing and decreasing sequences.
For any $1 \le k \le \tilde L(m)-2$
such that $\tilde \tau_k = \bar A^+[z^m](\tilde \tau_k)$, it must hold that
$\dot y^m(\tilde \tau_k+) <0$. Consequently, $\tilde \tau_{k+1} = \tau^-_k$, and
since
\[
y^m(\tilde \tau_{k+1}) < y^m(\tilde \tau_k) \le y^m(\bar A^+[z^m](\tilde \tau_{k+1}))  
\]
implies that $\tilde \tau_{k+1} \neq \bar A^+[z^m](\tilde \tau_{k+1})$,
it must hold that $\tilde \tau_{k+1} = \bar A^{-}[z^m](\tilde \tau_{k+1})$.
By similar reasoning, if $1 \le k \le \tilde L(m)-2$ is such that
$\tilde \tau_k = \bar A^-[z^m](\tilde \tau_k)$, then
$\tilde \tau_{k+1} = \bar A^+[z^m](\tilde \tau_k)$.
We obtain that $\tilde y^m_{k} \tilde y^m_{k+1} \le 0$ 
for all $k \le \tilde L(m)-2$,
\[
\tilde y^m_{k} \tilde y^m_{k+2}\ge 0 \quad \text{and} \quad
|\tilde y^m_{k}| \le |\tilde y^m_{k+2}| \qquad \forall k \le \tilde L(m)-3,
\]
and as $\tilde y^m_1 \neq 0$, if $\tilde L(m)>1$,
\begin{equation}\label{eq:dyTilde}
  \Delta \tilde y^m_{k}\Delta \tilde y^m_{k+1}  < 0 \quad \text{and} \quad 
|\Delta \tilde y^m_{k}| \le |\Delta \tilde y^m_{k+1}| \quad \forall k \le \tilde L(m)-3.
\end{equation}
(For the last index, $k=L(m)-2$, the properties~\eqref{eq:dyTilde} hold if
$T \in A^{\pm}[z^m]$, but may not hold if $T \notin A^{\pm}[z^m]$.)

\begin{corollary}\label{cor:bvEquiv2}
Assume $f \in C^2(\bR)$ is 
strictly convex, $u_0 \in (L^1 \cap BV)(\bR)$ and $z\in C_0([0,T])$.
For any mesh 
$0=\tau_0< \tau_1 < \ldots < \tau_m =T$, $m\ge 2$, let
 $y^m = \overline{\orm}_m[z]$ and
$\tilde y^m = \cI[y^m](\cdot; \{ \tilde \tau_{k}\}_{k=0}^{\tilde L(m)})$,
cf.~Algorithm~\ref{alg:removeAlg}. 

Then
\begin{equation}\label{eq:equivOrmAllRemoved}
  v(T;y^m, \{\tau_{k}\}_{k=0}^m )
  = v(T; \tilde y^m, \{ \tilde \tau_{k}\}_{k=0}^{\tilde L(m)})
\end{equation}
and 
\begin{equation}\label{eq:bvEquivOrmAllR}
\abs{y^m}_{\BV{[0,T]}} = \abs{\tilde y^m}_{\BV{[0,T]}}.
\end{equation}
\end{corollary}
\begin{proof}
Equation~\eqref{eq:equivOrmAllRemoved} 
follows directly by Lemma~\ref{lem:cancel}.

To verify~\eqref{eq:bvEquivOrmAllR}, note by Algorithm~\ref{alg:removeAlg} that
for any $k \in \{0,1,\ldots,\tilde L(m)-1\}$, it either holds that
\[
\dot y^m(s+) \ge 0 \quad \text{and} \quad \dot{\tilde y}^m(s+) 
\ge 0 \quad \forall s \in [\tilde \tau_k, \tilde \tau_{k+1}),
\]
or
\[
\dot y^m(s+) \le 0 \quad \text{and} \quad \dot{\tilde y}^m(s+) 
\le 0 \quad \forall s \in [\tilde \tau_k, \tilde \tau_{k+1}),
\]
and that
\[
\tilde{y}^m(\tilde \tau_k) = y^m(\tilde \tau_k) \quad \forall k \in \{0,1,\ldots, \tilde L(m)\}.
\]
Consequently,
\[
\begin{split}
  \abs{y^m}_{\BV{[0,T]}} & = 
  \sum_{k=0}^{\bar m-1} \abs{y^m(\tau_{j(k+1)}) - y^m(\tau_{j(k)})}\\
  & = \sum_{k=0}^{\tilde L(m)-1} \abs{y^m(\tilde \tau_{k+1}) - y^m(\tilde \tau_{k})}
  =\abs{\tilde y^m}_{\BV{[0,T]}}.
  \end{split}
\]

\end{proof}

We now propose a numerical method that makes use of
the piecewise linear orm function of $z$ and $\tilde y^m$
to compute entropy solutions at final time:
\begin{enumerate}

\item[(i)] Approximate the rough path $z \in C_0([0,T])$ by the 
  piecewise linear interpolant $z^m = \cI^m[z]$ on a uniform
  mesh $\{\tau_k\}_{k=0}^m$ with step size $T/m$.

\item[(ii)] Compute $y^m = \overline{\orm}_m[z]$ and its interpolation points
  $\{\tau_{j(k)}\}_{k=0}^{\bar m}$, cf.~Definition~\ref{def:ormLDef}.

\item[(iii)] Compute $\tilde y^m$ and its interpolation points
  $\{\hat \tau_{k=0}\}_{k=0}^{\hat L(m)}$ by Algorithm~\ref{alg:removeAlg}
  with $y^m$ and $\{\tau_{j(k)}\}_{k=0}^{\bar m}$ as input.

\item[(iv)]  Compute a numerical solution of
  $v(T;\tilde y^m, \{\tilde \tau_{k}\}_{k=0}^{\tilde L(m)})$, cf.~\eqref{eq:solRep3}
  using a consistent, conservative and monotone finite volume method.
  
\end{enumerate}  
The numerical solution of $v\prt{T;\tilde y^m, \{\tau_{k}\}_{k=0}^{\tilde L(m)}}$,
which we denote $U(T)$, 
is obtained through initializing $U(0)$ by~\eqref{eq:u0Approx}, 
and iteratively, for $k=0,1,\ldots, \tilde L(m)-1$,
computing the numerical solution of 
\begin{equation}\label{eq:numSolProblemLast}
  \begin{split}
\partial_t \tilde u + \sign{\Delta \tilde y^m_{k}} \partial_x f(\tilde u) &= 0 \quad \text{in} \quad 
(0, |\Delta \tilde y^m_k| ] \times \bR,  \\
\quad \tilde u(0) &= U(\tilde \tau_k),
  \end{split}
\end{equation}
and setting
$U(\tilde \tau_{k+1}) = \hat{\tilde{u}}(|\Delta \tilde y^m_k|; U(\tilde \tau_k))$, where 
$\hat{\tilde{u}}(s; U(\tilde \tau_k))$ denotes the numerical solution
of~\eqref{eq:numSolProblemLast} with $\hat{\tilde u}(0) = U(\tilde \tau_k)$.
We let $\tilde{n}(k) \ge1$ denote the number of uniform
timesteps used in the numerical solution of~\eqref{eq:numSolProblemLast}
over $[0, |\Delta \tilde y^m_k| ]$, and
\[
N = \sum_{k=0}^{\tilde L(m)-1} \tilde{n}(k),
\]
denotes the total number of timesteps the numerical method uses
to obtain the final time solution $U(T)$.
The size of the uniform timesteps used in the numerical
solution of the $k$-th problem~\eqref{eq:numSolProblemLast},
for $k \in \{0,1,\ldots, \tilde L(m)-1\}$,
is determined through the following CFL condition:
\begin{equation}\label{eq:cfl3}
  \begin{cases}
    \widetilde \Delta t_k = \frac{\abs{\Delta \tilde y^m_k}}{\tilde n(k)}, \quad \text{where}\\
    \\
    \tilde{n}(k) =
    \max\prt{\left \lceil \frac{\abs{\Delta \tilde y^m_k} 
    \|f'\|_{\infty}}{C_{\mathrm{CFL}}\Delta x  }\right \rceil, \, 1}.
  \end{cases}
\end{equation}
In other words, 
the numerical solution of the
$k$-th problem~\label{eq:numSolProblem}
is computed on the temporal mesh discretization
\begin{equation}\label{eq:tMeshFinal}
0=t_{k,0} \le t_{k,1} \le \ldots \le t_{k,\tilde{n}(k)} = |\Delta \tilde y^m_k|, 
\end{equation}
where $t_{k,r} = r \widetilde \Delta t_k$ for $0 \le r\le \tilde{n}(k)$.

For a given $z \in C_0^{0,\alpha}(\bR)$, $\alpha \in (0,1]$,  
Theorem~\ref{thm:met1Compl2} shows that 
the factors $\abs{z^m}_{BV([0,T])}^{5}$ and $\abs{z^m}_{BV([0,T])}^{6}$
respectively enter in lower and upper bounds of the computational cost
of solving $v(T;z^m, \{\tau_{k}\}_{k=0}^m)$ by the
adaptive time stepping method in Section~\ref{sec:met1}.
In comparison, the method considered here solves
$v(T; \tilde y^m, \{\tilde \tau_{k}\}_{k=0}^{\tilde L(m)})$,
and since 
\[
|\tilde y^m|_{BV([0,T])} = |\overline{\orm}_m[z^m]|_{BV{([0,T])}} \le
\abs{z^m}_{BV([0,T])} \quad \forall m \ge 2,
\]
cf.~Corollary~\ref{cor:bvEquiv2}, it is to be expected that
the latter numerical method can be more efficient than the former.
The following theorem states conditions under which
efficiency gains are achieved.

\begin{theorem}\label{thm:met2Compl}
Let $u\in C([0,T]; L^1(\bR))$ denote the unique pathwise entropy
solution of \eqref{eq:sscl} for given $u_0 \in (L^1 \cap BV)(\bR)$
with $\Leb{\supp{u_0}}>0$, strictly convex $f\in C^2(\bR)$ and $z \in
C^{0,\alpha}_0([0,T])$ with $\alpha \in (0,1]$.  For any $m \ge 2$,
let $\{\tau_j\}_{j=0}^m\subset [0,T]$ denote the uniform mesh with
step size $\Delta \tau = T/m$ and assume the computational cost of
generating the interpolant $z^m = \cI^m[z]$ is $\Theta(m^\beta)$ for
some $\beta \ge 1$.  Set $y^m = \overline{\orm}_m[z]$, and let $\tilde
y^m = \cI[y^m]\prt{\cdot; \{\tilde \tau_k\}_{k=0}^{\tilde L(m)}}$
denote the function generated by Algorithm~\ref{alg:removeAlg}.
Let $U$ denote the
solution of the numerical method in Section~\ref{sec:framework2}
satisfying the local CFL condition~\eqref{eq:cfl3} and
with the following constraint imposed on the spatial resolution
\begin{equation}\label{eq:dxConstraint}
\Delta x = \tO{\frac{\Delta \tau^\alpha}{\max\prt{|y^m|_{BV([0,T])}^2, \, 1}}}.
\end{equation}
Assume that the spatial support of $U([0,T])$
is covered by an interval $[a_m,b_m] \subset \bR$ that
satisfies
\[
c_1 \le b_m-a_m \le c_2 (1+ N\Delta x),
\]
for some $c_1,c_2>0$, cf.~\eqref{eq:finiteSupp},
and that at least one of the following two conditions hold:
\begin{itemize}
\item[(a)] 
\begin{equation*}
\tilde L(m) = \cO{ m^\alpha \max\prt{|y^m|_{BV([0,T])}^3,1}},
\end{equation*}   

\item[(b)]
  there exists an $\tilde m \ge 2$ and $\tilde c >0$ such that
\begin{equation*}
  \max\prt{M^+[z^m](\tau_{\ceil{m^\alpha}}), \abs{M^-[z^m](\tau_{\ceil{m^\alpha}})}} \ge \tilde c m^{-\alpha}
  \qquad \forall m \ge \tilde m.
\end{equation*}

\end{itemize}
Then
\begin{equation}\label{eq:NLastMethod}
N = \sum_{k=0}^{\tilde L(m)-1} \widetilde n(k) =
\cO{\frac{\max\prt{|y^m|_{BV([0,T])}^3, \, 1}}{\Delta \tau^{\alpha}}},
\end{equation}
and
\[
\|u(T) - U(T)\|_1 = \cO{m^{-\alpha/2}}
\]
is achieved at the computational cost 
\begin{multline*}
\hat c_1 \prt{ \max\prt{ |y^m|_{BV([0,T])}^{5},\, 1} m^{2\alpha}  +m^\beta}\\
\le \mathrm{Cost}(U) \le \\
 \hat c_2 \prt{\max\prt{ |y^m|_{BV([0,T])}^{6},\, 1} m^{2\alpha}  +m^\beta},
\end{multline*}
for some $\hat c_1, \hat c_2 >0$.

\end{theorem}

\begin{proof}

By the CFL condition~\eqref{eq:cfl3}, it holds for all $m\ge 2$
that
\[
\max_{k \in \{0,1,\ldots,\tilde{L}(m)-1\}} \widetilde \Delta t_k 
\le \frac{C_{\mathrm{CFL}}}{\|f'\|_\infty} \Delta x.
\]
Introducing the shorthand
\[
\tilde v^m(\tilde \tau_k) \coloneq v\prt{\tilde \tau_{k}; 
\tilde y^m, \{\tilde \tau_k\}_{k=0}^{\tilde L(m)}}
\qquad \forall k \in \{0,1,\ldots, \tilde L(m)\},
\]
and using Kuznetsov's lemma, cf.~\cite[Example 3.15]{HoldenRisebro2015}, 
the error of the numerical method at time $\tilde \tau_{k+1}$
can be bounded by
\begin{equation}\label{eq:ormErrBound}
	\begin{split}
\|\tilde v^m(\tilde \tau_{k+1}) - U(\tilde \tau_{k+1})\|_1 &
		\leq \|\tilde v^m(\tilde \tau_{k}) - U(\tilde \tau_k)\|_1 
		+ C |\Delta \tilde y^m_k| (\sqrt{\widetilde \Delta t_k} + \sqrt{\Delta x})\\ & 
		\leq  \|\tilde v^m(\tilde \tau_{k}) - U(\tilde \tau_k)\|_1 
		+ C|\Delta \tilde y^m_i| \sqrt{\Delta x},
	\end{split}
\end{equation}
for some $C>0$ that depends on $\|f'\|_{L^\infty}$, $|u_0|_{BV(\bR)}$ and the numerical scheme.
Using that $\tilde v^m(0) = u_0$,
$\tilde v^m(T) = u^m(T)$ and $|\tilde y^m|_{BV([0,T])} = |y^m|_{BV([0,T])}$,
cf.~Corollary~\ref{cor:bvEquiv2} and Theorem~\ref{thm:ormLinEquiv},
we obtain 
\[
\| u^m(T)- U(T)\|_1 \leq   
\|u_0 - U(0)\|_1 + C|y^m|_{BV([0,T])} \sqrt{\Delta x}.
\]
By Theorem~\ref{thm:wellPosednessPes}, \eqref{eq:ormErrBound},
and~\eqref{eq:dxConstraint}, 
\begin{equation}\label{eq:errBLastMethod}
    \begin{split}
      \|u(T) - U(T)\|_1 &\leq \|u(T) - u^m(T)\|_1 + \|u^m(T) - U(T)\|_1  
      \\ & = \cO{\Delta \tau^{\alpha/2} + |y^m|_{BV([0,T])} \sqrt{\Delta x}}\\
  & = \cO{\Delta \tau^{\alpha/2}}.
  \end{split}
\end{equation}
And by~\eqref{eq:cfl3} and~\eqref{eq:dxConstraint}, 
\[
\begin{split}
  N &= \sum_{k=0}^{\tilde L(m)-1} \tilde n(k)
  \le \sum_{k=0}^{\tilde L(m)-1}
  \prt{\frac{ \abs{\Delta \tilde y^m_k} \|f'\|_{L^\infty}}{C_{\mathrm{CFL}}\Delta x  } +1}\\
  &\le \frac{ \|f'\|_{L^\infty}}{C_{\mathrm{CFL}}\Delta x} \abs{y^m}_{BV([0,T])} 
  + \tilde L(m)\\ & = \cO{m^{\alpha} \max\prt{\abs{y^m}_{BV([0,T])}^3 , \, 1} + \tilde L(m) }.
\end{split}
\]
In order to obtain~\eqref{eq:NLastMethod}, it remains to verify that
\begin{equation}\label{eq:tildeLDominated}
\tilde L(m) = \cO{m^{\alpha} \max\prt{|y^m|_{BV([0,T])}^3 , \, 1}}.
\end{equation}
Assume condition (b) holds and that $m\ge \tilde m$.
Since $\{\tilde \tau_k\}_{k=0}^{\tilde L(m)-1} \subset A^{\pm}[z^m]$
and $\{|\Delta \tilde y^m_k|\}_{k=0}^{\tilde L(m)-2}$
is a monotonically increasing sequence,
cf.~\eqref{eq:dyTilde},
it holds for any $\tau_r \in \{\tilde \tau_k\}_{k=0}^{\tilde L(m)-2} \cup [\tau_{\ceil{m^\alpha}},T)$
that
\[
|\Delta \tilde y^m_r| \ge \max\prt{M^+[z^m](\tilde \tau_{r}), \abs{ M^-[z^m](\tilde \tau_{r})}}
\ge \tilde c m^{-\alpha}.
\]
Moreover, the following function 
is well-defined for all $m \ge \tilde m$:
\[
r(m) \coloneq \min \Big\{ k \in \{0,1,\ldots, \ceil{m^{\alpha}} \} \mid
|\Delta \tilde y^m_k| > \tilde c m^{-\alpha} \Big\}.
\]
It is clear that $r(m) = \cO{m^{\alpha}}$ and
\[
\begin{split}
|\tilde y^m|_{BV([0,T])} & = \sum_{k=0}^{\tilde L(m)-1} |\Delta \tilde y^m_k|
\ge \sum_{k=r(m)}^{\tilde L(m)-2} m^{-\alpha}
= (\tilde L(m) - (r(m)+2) )m^{-\alpha}.
\end{split}
\]
We conclude that
\[
\tilde L(m) = \cO{m^{\alpha} \max\prt{|y^m|_{BV([0,T])}, \,1}}
=\cO{m^{\alpha} \max\prt{|y^m|_{BV([0,T])}^3, \,1}},
\]
and thus condition (b) is stronger than condition (a) 
(condition (b) is included in the theorem as it might be easier to verify
than condition (a)).

The computational cost of the numerical method
is equal to the sum of $\tO{m^\beta}$ for generating the
piecewise linear interpolant $z^m$, and
\[
\begin{split}
  \tO{ N \times \frac{b_m - a_m}{\Delta x}}
\end{split}
\]
for solving $U$ over $[0,T]\times [a_m,b_m]$.
\end{proof}

\subsection{Efficiency gains}
By comparing the computational cost versus accuracy results in
Theorems~\ref{thm:met1Compl2} and~\ref{thm:met2Compl}, we see that
if the assumptions of both theorems hold, 
\[
\limsup_{m\to \infty} \frac{m^\beta}{m^{2\alpha}|z^m|^5_{BV([0,T])}} = 0,
\]
and
\begin{equation}\label{eq:finiteLOrm}
\limsup_{m\to \infty} \frac{|\overline{\orm}_m[z]\, |_{BV([0,T])}^6}{\abs{z^m}_{BV([0,T])}^5} =0,
\end{equation}
then it is guaranteed that
the orm based numerical method will 
asymptotically be more efficient than
the adaptive timestep method.

The next two lemmas, Lemmas~\ref{lem:ormBV} and~\ref{lem:maxPropWiener},
verify that 
$\abs{\overline{\orm}_m[z]}_{BV([0,T])} < \infty$, for all $m\ge 2$,
and assumption (b) in Theorem~\ref{thm:met2Compl} hold
for almost all sample paths of a standard Wiener process.

\begin{lemma}\label{lem:ormBV}
Let $(\Omega,\cF, \bP)$ denote a probability space on which the
standard Wiener process $W: [0,\infty) \times \Omega \to \bbR$ with
$W(0) =0$, $\bP$-a.s.~is defined. For every $\omega \in \Omega$, let $z \coloneq
W(\cdot, \omega)$ denote a sample path of the Wiener process.
Then, for every $\alpha \in (0,1/2)$ and $\widehat T>0$ 
\begin{equation}\label{eq:wienerAsCont}
z \in C_0^{0,\alpha}([0,\widehat T]), \qquad \bP\mathrm{-a.s.}
\end{equation}
Furthermore, for a fixed $T>0$, let
\begin{equation}\label{eq:tildeOmega}
\widetilde \Omega = \{ \omega \in \Omega \mid W(\cdot, \omega) \in C_0([0,T]) \},
\end{equation}
and for any $m \ge 2$, let
$\{\tau_k\}_{k=0}^m\subset [0,T]$ denote the uniform mesh with
step size $\Delta \tau = T/m$.
We define
\begin{equation}\label{eq:zMDefWiener}
  z^m \coloneq \begin{cases}
    \cI[z](\cdot; \{\tau_k\}_{k=0}^m) & \text{if} \quad \omega \in \widetilde{\Omega},\\
    0 & \text{if} \quad \omega \in \Omega\setminus \widetilde \Omega,
    \end{cases}
\end{equation}
and
\begin{equation}\label{eq:ormMDefWiener}
  \overline{\orm}_m[z] \coloneq \begin{cases}
    \overline{\orm}[z](\cdot; \{\tau_k\}_{k=0}^m)  & \text{if} \quad \omega \in \widetilde{\Omega},\\
    0 & \text{if} \quad \omega \in \Omega\setminus \widetilde \Omega.
    \end{cases}
\end{equation}
Then
\begin{equation}\label{eq:modCont2}
  \limsup_{m\to\infty}  \sup_{t \in [0,T]} |z - z^m|(t) \sqrt{\frac{m}{\log(m)}}  
  \le \sqrt{2T}, \qquad \bP\mathrm{-a.s.,}
\end{equation}
\begin{equation}\label{eq:wienerBV}
\limsup_{m \to  \infty} \abs{z^m}_{BV([0,T])}
\sqrt{\frac{\log(m)}{m}} 
\ge \sqrt{\frac{T}{2}}, \quad \bP\mathrm{-a.s.,}
\end{equation}
and for $\bP$-almost all paths, there is a 
constant $C(\omega) >0$ such that
\begin{equation}\label{eq:wienerOrmBV1}
\limsup_{m \to  \infty} 
\abs{\overline{\orm}_m[z]}_{BV([0,T])}
 < C(\omega).
\end{equation}
Moreover,
\begin{equation}\label{eq:wienerOrmBV2}
\E{ \sup_{m\ge 2} \abs{\overline{\orm}_m[z]}_{BV([0,T])}} < \infty,
\quad \E{\abs{\orm[z]}_{BV([0,T])}} < \infty.
\end{equation}
\end{lemma}

\begin{proof}
See~\cite{Evans2012} for a proof of~\eqref{eq:wienerAsCont}.

Since $\Prob{\widetilde \Omega}=1$, we restrict ourselves 
to $\omega \in \widetilde \Omega$ in what follows.
By L\'{e}vy's global modulus of continuity~\cite[Theorem 9.25]{Karatzas91}
\footnote{Theorem 9.25 in~\cite{Karatzas91} is formulated for
  standard Wiener processes over the time interval $[0,1]$. 
  However, for any $T>0$, the transform $\widetilde W(t) = \sqrt{T} W(t/T)$
  yields a standard Wiener process $\widetilde W:[0,T]\times \Omega \to \bR$
  and the result extends straightforwardly to the time interval $[0,T]$.},
\[
\limsup_{\delta \downarrow 0} 
\sup_{{\scriptsize \begin{split} 0&\le s\le t\le T \\ &t-s\le \delta \end{split}}}
\frac{\abs{z(t) - z(s)}}{\sqrt{2\delta \log(1/\delta)}}  = T,
\qquad \bP\mathrm{-a.s.,}
\]
and as 
\[
\begin{split}
\sup_{t \in [0,T]} \abs{z(t) - z^m(t)} &\le
\max_{k\in \{0,1,\ldots, m-1\}} \sup_{t \in [\tau_{k},\tau_{k+1}]} \max\prt{\abs{z(t) - z(\tau_k)},\abs{z(t) - z(\tau_{k+1)}}}\\
&\le \sup_{{\scriptsize \begin{split} 0&\le s\le t\le T \\ &t-s\le T/m \end{split}}}
\abs{z(t) - z(s)},
\end{split}
\]
inequality~\eqref{eq:modCont2} follows and so does
\[
\limsup_{m \to \infty} \prt{\max_{k \in \{0,1,\ldots, m-1\}} \abs{\Delta z^m_k} \sqrt{ \frac{m}{ \log(m)}}}
\le \sqrt{2T} , \qquad \bP\mathrm{-a.s.}
\]
We further recall from~\cite{Dudley73} that 
\[
\limsup_{m \to \infty} \sum_{k=1}^{m-1} \abs{\Delta z^m_k}^2 =
\lim_{m \to \infty} \sum_{k=1}^{m-1} \abs{\Delta z^m_k}^2 = T, \qquad \bP\mathrm{-a.s.}
\]
Hence, 
\[
\begin{split}
T &= \limsup_{m \to \infty} \sum_{k=1}^{m-1} \abs{\Delta z^m_k}^2\\
&\le \limsup_{m \to \infty} \prt{ \max_{\ell \in \{0,1,\ldots, m-1\}} \abs{\Delta z^m_\ell}  \sum_{k=1}^{m-1} \abs{\Delta z^m_k}}\\
  & \le \limsup_{m \to \infty} \prt{ \max\prt{\max_{k \in \{0,1,\ldots, m-1\}}\abs{\Delta z^m_k} \sqrt{ \frac{m}{ \log(m)}}, \, \sqrt{2T} }} \\
& \qquad   \times \limsup_{m \to \infty} \prt{\abs{z^m}_{BV([0,T])} \sqrt{ \frac{ \log(m)}{m}}} \\
& = \sqrt{2T}  \limsup_{m \to \infty} \prt{\abs{z^m}_{BV([0,T])} \sqrt{ \frac{\log(m)}{m}}}, \qquad  \qquad \bP\mathrm{-a.s.},
\end{split}
\]
and~\eqref{eq:wienerBV} follows.

Equations~\eqref{eq:wienerOrmBV1} and~\eqref{eq:wienerOrmBV2} are proved in
Appendix~\ref{app:bvWiener}.
\end{proof}  

The next lemma shows that assumption (b) in Theorem~\ref{thm:met2Compl} holds
for almost all sample paths of a standard Wiener process.
\begin{lemma}\label{lem:maxPropWiener}
For any $m \ge 2$, let $\{\tau_k\}_{k=0}^m\subset [0,T]$ denote the
uniform mesh with with step size $\Delta \tau = T/m$.
Let $z = W(\cdot, \omega) \in C_0([0,T])$ denote a path realization of
the standard Wiener process with
$\omega \in \widetilde \Omega$, cf.~Lemma~\ref{lem:ormBV} and~\eqref{eq:tildeOmega},
and let $z^m$ be defined by~\eqref{eq:zMDefWiener}.
Then, for any $\alpha \in (2/5,1]$,
there exists an $\tilde m(\omega) \ge 2$ for almost all
$\omega \in \widetilde \Omega$ 
such that\footnote{By a slight modification of the proof, one may 
  show that for almost all~$\omega \in \widetilde \Omega$, inequality~\eqref{eq:maxPropWiener}
  holds for any $\alpha \in (1/4,1]$.}
\begin{equation}\label{eq:maxPropWiener}
  \max\prt{M^+[z^m](\tau_{\ceil{m^{\alpha}}}), \abs{M^-[z^m](\tau_{\ceil{m^{\alpha}}})}}
  \ge m^{-\alpha}
  \qquad \forall m \ge \tilde m(\omega).
  \end{equation}
\end{lemma}
\begin{proof}
  For $y \in \bR$, let $\floor{y}$ denote the largest $n \in \bZ$ such that
  $y\ge n$ and let $r^m:\bN \to \bN$ be defined by $r^m(k) = \floor{m^{\alpha/2}}k$.
  For any natural number $m \ge \hat m = \ceil{3^{2/\alpha}}$, we introduce
  the set
  \small
  \[
  \begin{split}
  D_m = \left\{ \omega \in \tilde{\Omega}\, \Big{\vert} \,\,
  \abs{z(\tau_{r^m(k+1)}) - z(\tau_{r^m(k)})} \le \frac{2}{m^\alpha} \quad \text{for} \quad  k =0,1,\ldots, \floor{m^{\alpha/2}}-1 \right\}.
  \end{split}
  \]
  \normalsize
 
  We claim that $\omega \in \widetilde \Omega$ for which~\eqref{eq:maxPropWiener} does not
  hold is contained in
  \begin{equation}\label{eq:capcup}
  \bigcap_{m \ge \hat m} \bigcup_{\ell \ge m} D_\ell.
  \end{equation}
  To verify this, observe that if $\omega \notin \bigcap_{m \ge \hat m} \bigcup_{\ell \ge m} D_\ell$,
  then there exists an $\tilde m(\omega)$ such
  that $\omega \notin \bigcup_{\ell \ge \tilde m} D_\ell$.
  Since $\omega \notin D_m$ for every $m\ge \tilde m$,
  there exists a $k_m \in \{0,1,\ldots, \floor{m^{\alpha/2}}-1\}$ such that
  $\abs{z(\tau_{r^m(k_m+1)}) - z(\tau_{r^m(k_m)})} \ge \frac{2}{m^\alpha}$.
  This implies that $\max(|z(\tau_{r^m(k_m+1)})|, |z(\tau_{r^m(k_m)})|)> m^{-\alpha}$,
  and, since $r^m(k_m+1)\le\ceil{m^\alpha}$ by construction,
  we conclude that 
  \[
  \max\prt{M^+[z^m](\tau_{\ceil{m^{\alpha}}}), \abs{M^-[z^m](\tau_{\ceil{m^{\alpha}}})}}
  \ge m^{-\alpha}, \qquad \forall m \ge \tilde m.
  \]

  We will show that~\eqref{eq:capcup} is a zero-measure set.
  Since the increments
  $\{z(\tau_{r^m(k+1)}) - z(\tau_{r^m(k)})\}_k$ are independent
  $N\prt{0, \tau_{\floor{m^{\alpha/2}}}}$ distributed random variables,
  \[
  \Prob{D_m } \le
   \prt{\Prob{\abs{z(\tau_{r^m(1)}) - z(0)}\le \frac{2}{m^\alpha}}}^{\floor{m^{\alpha/2}}},
  \]
  and since $\tau_{\floor{r^m(1)}} = \tO{m^{\alpha/2-1}}$,
  there exists a $C>0$ independent of $m$ such that
  for all $m \ge \hat m$,
  \[
  \begin{split}
    \Prob{\abs{z(\tau_{r^m(1)}) - z(0)}\le \frac{2}{m^\alpha}}&
    = \frac{1}{\sqrt{2 \pi \tau_{r^m(1)}}}
    \int_{-2m^{-\alpha}}^{2m^{-\alpha}} \exp\prt{\frac{-x^2}{2\tau_{r^m(1)}}} \dx\\
    & = \frac{1}{\sqrt{2 \pi}}
    \int_{\frac{-2m^{-\alpha}}{\sqrt{\tau_{r^m(1)}}}}^{\frac{2m^{-\alpha}}{\sqrt{\tau_{r^m(1)}}}}
    e^{-y^2/2} \dy\\
    & \le C m^{1/2 - 5\alpha/4}.
     \end{split}
  \]
  Since $1/2 - 5\alpha/4 < 0$, there exists a $\check m \ge \hat m$,
  such that 
  \[
  \Prob{D_m } \le m^{-2} \quad \text{for all} \quad m \ge \check m,
  \]
  and
  \[
  \Prob{\bigcap_{m \ge \hat m} \bigcup_{\ell \ge m} D_\ell}
  \le \liminf_{m \to \infty} \Prob{\bigcup_{\ell \ge m} D_\ell}
  \le \liminf_{m \to \infty} \sum_{\ell=m}^\infty \Prob{ D_\ell}
  \le \lim_{m \to \infty} \frac{1}{m}
   = 0.
  \]
  Equations~\eqref{eq:wienerAsCont} and~\eqref{eq:tildeOmega} ensures that
  \[
  \Prob{\widetilde \Omega \setminus \bigcap_{m \ge \hat m} \bigcup_{\ell \ge m} D_\ell} =1,
  \]
  and the proof follows.
\end{proof}

By the ``H\"older continuity''~\eqref{eq:modCont2},
equations~\eqref{eq:wienerBV} and~\eqref{eq:wienerOrmBV1},
and the fact that $\beta =1$ for sampling a path of a
standard Wiener processes, we conclude that under the
shared assumptions on $u_0$ in
Theorems~\ref{thm:met1Compl2} and~\ref{thm:met2Compl}
and if $f \in C^2(\bR)$ is strictly convex, then
the computational cost of achieving 
\[
\|u(T) - U(T) \|_{L^1} = \cO{ \prt{\frac{\log(m)}{m}}^{1/4}},
\]
for a sample path $z:[0,T] \to \bR$ of the standard Wiener process 
admits for some $c_1(\omega)>0$ the following lower bound
for the adaptive time stepping method (cf.~Theorem~\ref{thm:met1Compl2}):
\[
 \mathrm{Cost}(U(T)) \ge c_1 \prt{\frac{ m}{\log(m)}}^{7/2} 
 \qquad \bP\mathrm{-a.s.,}
\]
and, for some $c_2(\omega)>0$, the following upper bound
for the orm based method (cf.~Theorem~\ref{thm:met2Compl}):
\[
\mathrm{Cost}(U(T)) \le c_2 m \qquad \bP\mathrm{-a.s.}
\]

\begin{remark}
In many cases (e.g.~Brownian paths), 
$\abs{\orm[z]}_{BV([0,T])} < \abs{z}_{BV([0,T])}$.
That is however not always the case.  
The second example in Figure~\ref{fig:orm} considers
the locally rough path $z(t) = \1{t>0} t \sin(\pi/(2t))$, which is a
member of $C^{0,1/2}_0([0,T])$ for which $\abs{\orm[z]}_{BV([0,T])} 
=\abs{z}_{BV([0,T])}=\infty$.  This shows that even if $f \in C^2(\bR)$
is strictly convex, the orm based method will
not always solve~\eqref{eq:sscl} more efficiently
than the adaptive timestepping method.
\end{remark}

\subsection{Numerical tests}
\begin{example}\label{ex:BVGrowth}
To investigate if Lemma~\ref{lem:ormBV} holds
more generally, we 
approximate
$\E{\abs{z^m}_{BV([0,1])}}$ and $\E{\abs{\overline{\orm}_m[z]}_{BV([0,1])}}$
for four different fBMs with respective
Hurst indices $\alpha = 1/8,1/4,1/2$ and $3/4$ 
on uniform meshes of $[0,1]$ with respective step sizes
$\Delta \tau = 1/m$ for $m = 2^{5}, 2^{6}, \ldots, 2^{16}$.
The expectations are approximated by the
Monte Carlo method using $10^6$ sample realizations of 
$\abs{z^m}_{BV([0,1])}$ and $\abs{\overline{\orm}_m[z]}_{BV([0,1])}$.
\begin{figure}[!h]
\includegraphics[width=0.87\textwidth]{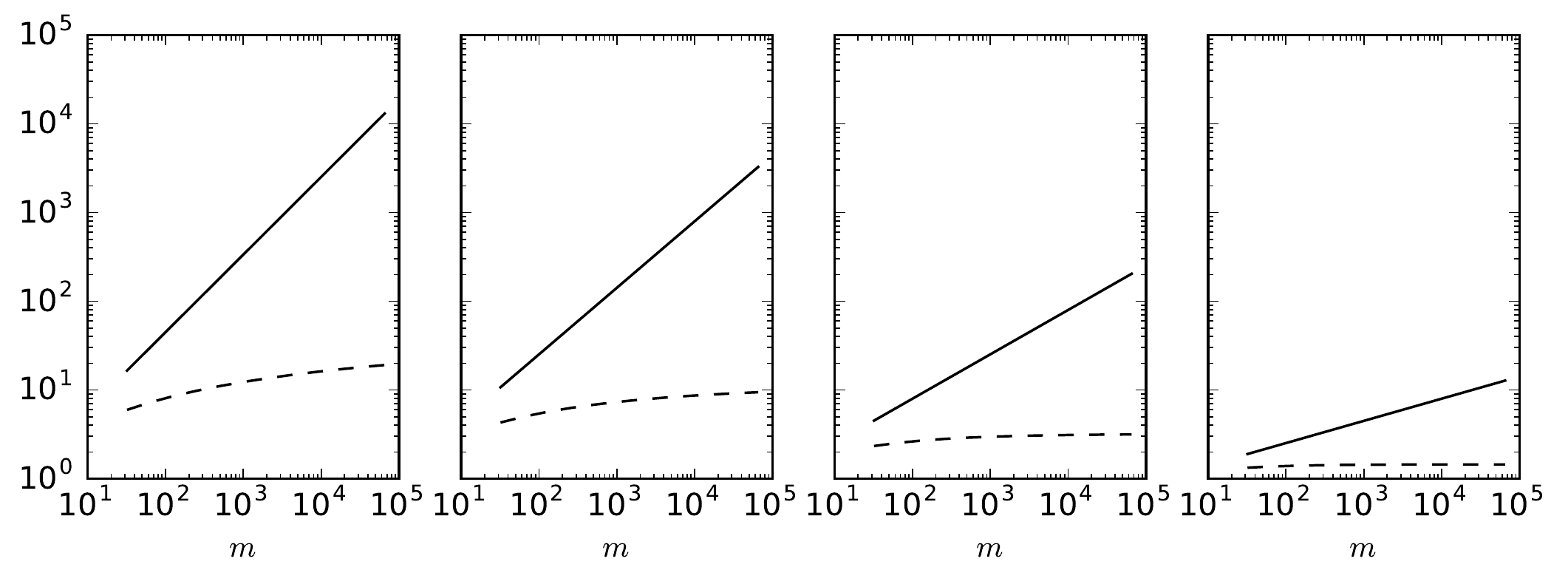}\\
\includegraphics[width=0.87\textwidth]{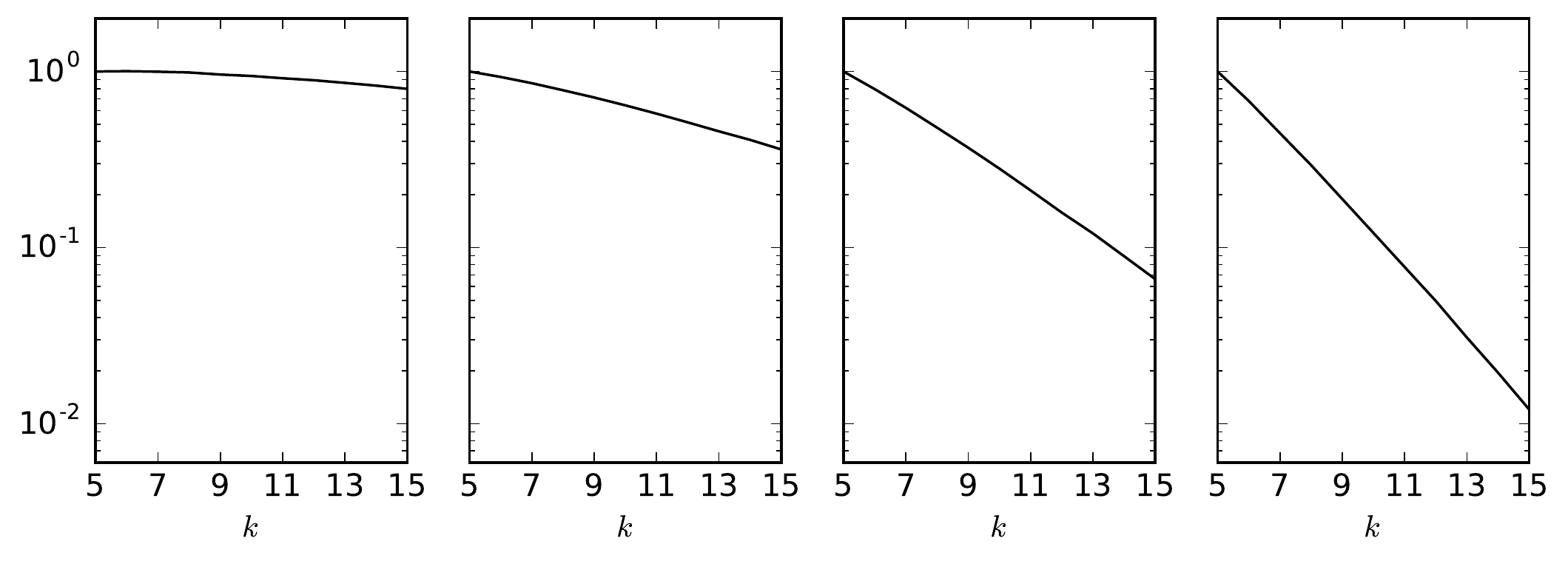}
\caption{Top row: Example~\ref{ex:BVGrowth}, 
Monte Carlo computations of $\E{\abs{z^m}_{BV([0,1])}}$ (solid line) 
and $\E{\abs{\overline{\orm}_m[z]}_{BV([0,1])}}$ (dashed line), 
Hurst indices $\alpha = 0.125$ (left), $\alpha =0.25$ (second left), 
$\alpha = 0.5$ (third left), and $\alpha =0.75$ (right).
Bottom row: plots of the BV increment ratio $G(k)$, 
cf.~\eqref{eq:theRealG}, corresponding to 
the respective top row test cases.}
\label{fig:ex4BV}
\end{figure}
Our Monte Carlo estimates of the expectations are presented in
Figure~\ref{fig:ex4BV}. We observe that while
$\E{\abs{z^m}_{BV([0,1])}} = \tO{ m^{1-\alpha}}$, 
$\E{|\overline{\orm}_{m}[z]|_{BV([0,1])}}$ seems to
stay bounded as $m$ increases for all of the considered fBMs. 
In the bottom row of Figure~\ref{fig:ex4BV}, we have computed the 
``BV increment ratio"
\begin{equation}\label{eq:theRealG}
	G(k) : = \frac{\E{|\overline{\orm}_{^{2^{k+1}}}[z]|_{BV([0,1])}
	-|\overline{\orm}_{2^k}[z]|_{BV([0,1])}}}{\E{|\overline{\orm}_{2^6}[z]|_{BV([0,1])}
	-|\overline{\orm}_{2^5}[z]|_{BV([0,1])}}},
\end{equation}
for $k\ge5$.  The increment ratio seems 
to be geometric of the form $G(k) = \cO{ \rho_\alpha^{k-5}}$ 
with $\rho_\alpha \approx 2^{-\alpha}<1$.
We interpret this as numerical support for 
\[
\E{\lim_{k\to \infty} |\overline{\orm}_{2^k}[z]|_{BV([0,1])}} < \infty, \quad \forall
\alpha \in \{1/8,1/4, 1/2, 3/4\},
\]
since if $G(k) <1$, 
\[
\begin{split}
  & \E{\lim_{k\to \infty} |\overline{\orm}_{2^k}[z]|_{BV([0,1])}}\\
  &= \E{|\overline{\orm}_{2^5}[z]|_{BV([0,1])}}
  + \sum_{k=5}^\infty \E{|\overline{\orm}_{2^{k+1}}[z]|_{BV([0,1])}
    -|\overline{\orm}_{2^k}[z]|_{BV([0,1])}}\\
  &\leq C \sum_{k=0}^\infty \rho_\alpha^k < \infty.
\end{split}
\]
\end{example}


\begin{example}\label{ex:ormNumerics}
  We consider problem~\eqref{eq:ssclPer} over the time interval
  $[0,1]$ with periodic boundary conditions,
  initial data
  \[
  u_0(x) = \sign{x-1/2} \1{ 1/6<5/6} \qquad x \in [0,1],
  \]
  strictly convex flux $f(u) = u/2 + u^2/4$ and
  $z = W(\cdot, \omega)$, where $W:[0,1]\times \widetilde \Omega \to \bR$
  denotes the standard Wiener process and $\widetilde \Omega$ is defined
  in~\eqref{eq:tildeOmega}. In
  Figure~\ref{fig:ex5OrmNum} we compare the performance of the
  adaptive timestep method and the orm based method
  (cf.~Section~\ref{sec:framework2}), both using the
  Lax--Friedrichs scheme.
  The driving path $z$ is piecewise linearly interpolated
  on two uniform mesh resolutions, $m=2^8$ and $m=2^{10}$,
  and the computational cost of the
  respective algorithms are equilibrated through
  \[
  \Delta x_{\text{orm}} = \frac{\Delta x_{\mathrm{adaptive}}}{
  \ceil{\sqrt{\frac{\abs{z^m}_{BV([0,T])}}{|\orm[z^m]|_{BV([0,T])}}}\,\,}},
  \]
  (equilibrating either lower or upper bound costs in both of
  Theorems~\ref{thm:met1Compl2} and~\ref{thm:ormEquiv}).  The
  approximate reference solution is computed by the orm based method
  using the Lax--Friedrichs scheme with piecewise
  linear interpolation of $z$ on
  a uniform mesh with much higher resolution ($m=2^{14}$). We observe that
  at comparable computational budget, the orm
  based method approximates the reference solution with better accuracy
  and produces less artificial diffusion than the adaptive timestep
  method.
\begin{figure}[h!]
\includegraphics[width=0.7\textwidth]{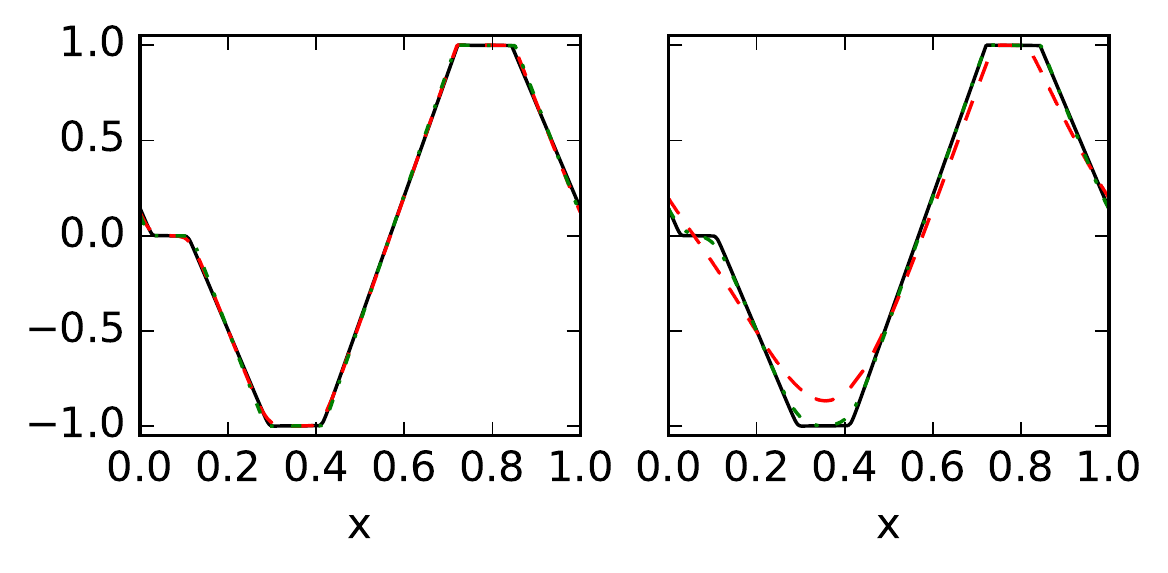}\\
\includegraphics[width=0.68\textwidth]{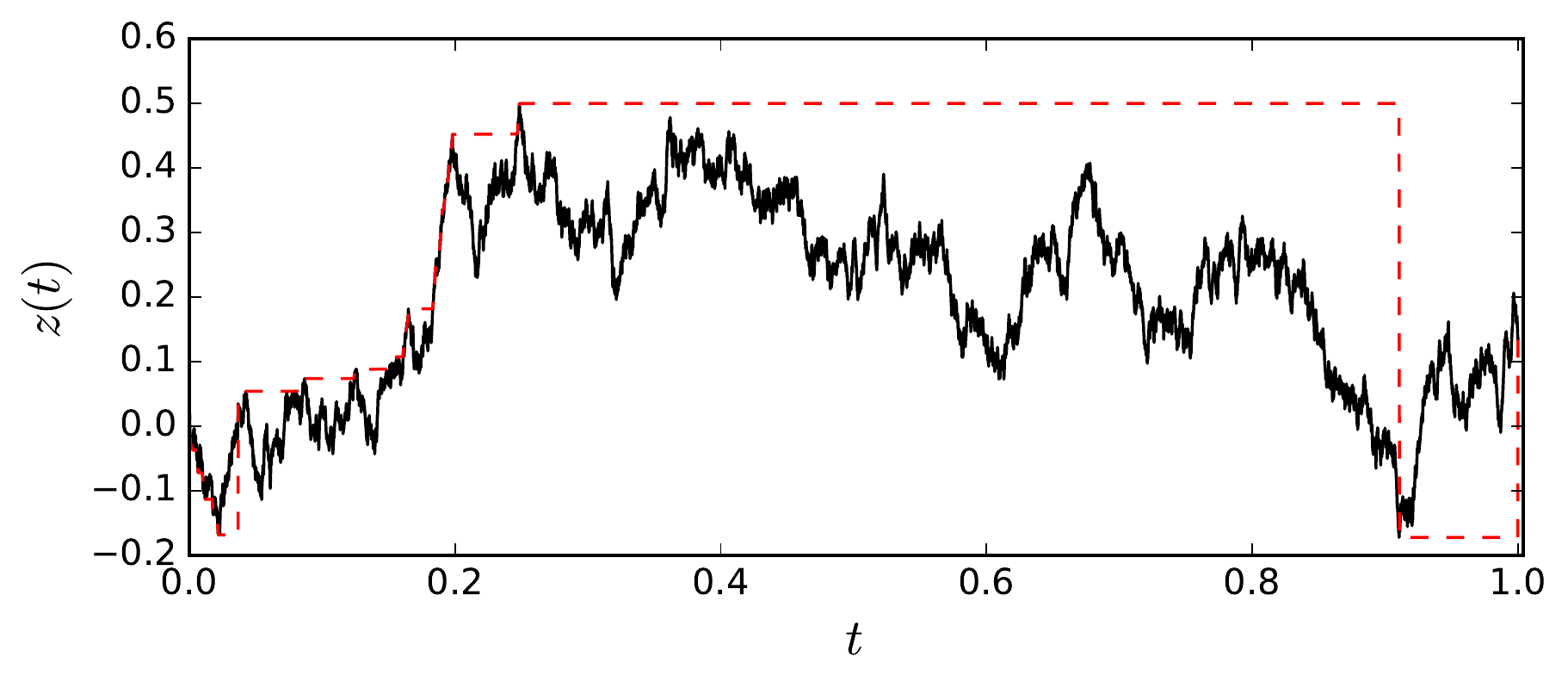}\\
\caption{ Top row: Final time solutions for
  Example~\ref{ex:ormNumerics} computed with the orm based method
  (left figure) and adaptive timestep method (right figure).  The
  black solid line is the approximate reference solution $u(1)$, while
  the red dashed and the green dash-dotted lines are the numerical
  solutions at the rough path resolutions $m=2^8$ and $m=2^{10}$,
  respectively.  Bottom: The rough path (black solid line) and the
  corresponding orm (red dashed line).}
\label{fig:ex5OrmNum}
\end{figure}
\end{example}

\section{Conclusion}\label{sec:conclude}
In this work we have developed fully discrete and thus computable
numerical methods for solving conservation laws with rough paths. 
For strictly convex flux functions, we have identified a
class of ``oscillatory cancellations'' that can be removed from the
rough path to produce numerical methods with improved efficiency. If
the rough path is a realization of a Wiener process, for instance,
the asymptotic efficiency gain can be of orders of magnitude.
An in-depth study of the rough path cancellation property
is found in the theoretical companion article \cite{Hoel:2017}.

\appendix

\section{Regularity of solutions}\label{app:regProof}

\begin{lemma}\label{lem:regLem}
For a given set of points 
\[
0=\tau_0 < \tau_1 < \ldots< \tau_m=T, \qquad m\ge 2,
\]
let $u^m \in C_0([0,T]; L^1(\bR))$ 
denote the c\`{a}dl\`{a}g version of the entropy solution of~\eqref{eq:consClass} 
with initial data $u_0 \in (L^1 \cap BV)(\bR)$, driving noise
$z^m \in I_0([0,T]; \{\tau_j\}_{j=0}^m)$ and strictly convex $f\in C^2(\bR)$.
For any $x \in \bbR$ and $t \in [0,T]$, let  
\[
u^m(t,x+) \coloneq \lim_{\delta \downarrow 0} u^m(t,x+\delta)\, ( = u(t,x) )\quad \text{and} \quad 
u^m(t,x-) \coloneq \lim_{\delta \downarrow 0} u^m(t,x-\delta),
\]
Then, using the convention $0^{-1} = \infty$,  it holds for all 
$x < y$ and $t \in [0,T]$ that
\begin{equation}\label{eq:lipReg1}
  \begin{split}
    -\frac{1}{M^+[z^m](t)-z^m(t)} &\leq \frac{f'(u^m(t,y\pm))-f'(u^m(t,x\pm))}{y-x} \leq \frac{1}{z^m(t)-M^-[z^m](t)},
  \end{split}
\end{equation}
where the functions $M^-[z^m]$ and $M^+[z^m]$ are defined in~\eqref{eq:runMaxMin}.
\end{lemma}
To prove Corollary~\ref{cor:regularity} we will need the following intermediate result,
which is an adaptation of~\cite[Lemma~3.3]{Hoel:2017}.
\begin{lemma}\label{lem:regStep}
For a given set of points 
\[
0=\tau_0 < \tau_1 < \ldots< \tau_m=T, \qquad m\ge 2,
\]
let $u^m \in C_0([0,T]; L^1(\bR))$
denote the c\`{a}dl\`{a}g version of the entropy solution of~\eqref{eq:consClass} 
with initial data $u_0 \in (L^1 \cap BV)(\bR)$, driving noise
$z^m \in I_0([0,T]; \{\tau_j\}_{j=0}^m)$ and strictly convex $f\in C^2(\bR)$.
Assume that \eqref{eq:lipReg1} holds at some time $\tau_k \in \{\tau_j\}_{j=0}^{m-1}$.
Then, for all $x<y$ and $s \in [\tau_k, \tau_{k+1}]$, 
\begin{equation}\label{eq:lipReg2}
  \begin{split}
   -\frac{1}{ M^+[z^m](s)- z^m(s)} &\leq \frac{f'(u^m(s ,y\pm))-f'(u(s,x\pm))}{y-x} \leq \frac{1}{z^m(s) - M^-[z^m](s)}.
   \end{split}
 \end{equation}
\end{lemma}
\begin{proof}
For some $s \in (\tau_k, \tau_{k+1}]$, let $\zeta_{\pm}$ be the maximal/minimal backward
generalized characteristic emanating from $(s, y)$ and $\xi_{\pm}$ be the maximal/minimal backward
generalized characteristic emanating from $(s, x)$, cf.~\cite[\S~X]{Dafermos:2010fk}.
The solution representation
\[
u^m(s) =\bar \cS(z^m(s) - z^m(\tau_k)) u^m(\tau_k) 
= \cS^{\dot z^m_k}(s-\tau_k) u^m(\tau_k),
\]
cf.~\eqref{eq:map1} and~\eqref{eq:map3}, 
and~\cite[Theorem 11.1.1]{Dafermos:2010fk}
implies that the generalized characteristics 
satisfies the following equations
\small
 \begin{align*}
  x &= \xi_{\pm}(\tau_k) + (s-\tau_k)\dot{z}^m_k f'(u^m(s,x\pm))
  = \xi_\pm(\tau_k) + (z^m(s)-z^m(\tau_k)) f'(u^m(s,x\pm)), \\
  y &= \zeta_{\pm}(\tau_k) + (s-\tau_k) \dot{z}^m_k f'(u^m(s,y\pm)) =
  \zeta_\pm(\tau_k) + (z^m(s)-z^m(\tau_k)) f'(u^m(s,y\pm)),
 \end{align*}
\normalsize
 where $\xi_{\pm}(\tau_k) \le \zeta_{\pm}(\tau_k)$ and each equation holds
 using either the limit sign $+$ or $-$ consistently in all terms with
 the appendage $\pm$ (i.e., $x+$ and $\xi_+(\tau_k)$ etc.); and we recall
 that $\dot{z}^m_k = \dot{z}^m(\tau_k+)$.
  We will treat the cases $\dot{z}^m_k >0$ and $\dot{z}^m_k <0$ separately
 (the case $\dot{z}^m_k = 0$ is trivial since then
  $u^m(t) = \cS^{0}(t-\tau_k) u^m(\tau_k) = u^m(\tau_k)$ for all $t \in [\tau_k,\tau_{k+1}]$).
 
Assume $\dot{z}^m_k >0$. By~\cite[Theorem 11.1.3]{Dafermos:2010fk},
 \begin{equation}\label{eq:dafProp}
   \begin{split}
 u^m(\tau_k, \xi_\pm(\tau_k)-) &\le u^m(s, x\pm) \le u^m(\tau_k, \xi_\pm(\tau_k)+), \\
 u^m(\tau_k, \zeta_\pm(\tau_k)-)& \le u^m(s, y\pm) \le u^m(\tau_k, \zeta_\pm(\tau_k)+).
 \end{split}
 \end{equation}
  Under the assumption that $f'(u^m(s,y\pm))-f'(u^m(s,x\pm)) \neq 0$
  (otherwise~\eqref{eq:lipReg2} holds trivially), we have
  \begin{equation}\label{eq:regAux}
   \begin{split}
    \frac{f'(u^m(s,y\pm))-f'(u(s,x\pm))}{y-x}=   \frac{1}{\frac{\zeta_{\pm}(\tau_k) - \xi_{\pm}(\tau_k)}{f'(u^m(s,y\pm))-f'(u(s,x\pm))}
       +(z^m(s)-z^m(\tau_k))}.
   \end{split}
  \end{equation}
  If $f'(u^m(s,y\pm))-f'(u^m(s,x\pm))< 0$, then the upper bound of~\eqref{eq:lipReg2} holds
  trivially, and inequalities~\eqref{eq:dafProp} and $f'>0$ implies that 
 \[
 0>f'(u^m(s,y\pm))-f'(u^m(s,x\pm))\ge f'(u^m(\tau_k,\zeta_\pm(\tau_k)-))- f'(u^m(s,\xi_\pm(\tau_k)+)).
 \]
 Hence, either $\zeta_\pm(\tau_k) = \xi_\pm(\tau_k)$ or $\zeta_\pm(\tau_k) > \xi_\pm(\tau_k)$. In the former case,
 the lower bound of~\eqref{eq:lipReg2} holds since
 \[
 \frac{\zeta_{\pm}(\tau_k) - \xi_{\pm}(\tau_k)}{f'(u^m(s,y\pm))-f'(u(s,x\pm))} + (z^m(s)-z^m(\tau_k))
   = (z^m(s)-z^m(\tau_k)) > 0,
 \]
 (as we assume $\dot{z}^m_k>0$). In the latter case, since $u^m(\tau_k)$ is c\`{a}dl\`{a}g, there exists a
 sequence 
 $\{\xi_{\pm,j} (\tau_k)\}_{j \in \bbN} \subset [\xi_{\pm}(\tau_k), \zeta_{\pm}(\tau_k))$ such that
   $\xi_{\pm,j} (\tau_k) \downarrow \xi_{\pm}(\tau_k)$,
   \[
   f'(u^m(\tau_k,\xi_{\pm,j}(\tau_k)-) > f'(u^m(\tau_k,\zeta_{\pm}(\tau_k)-) \quad \forall j \in \bbN,
   \]
   and
   \[
   \lim_{j\to \infty} f'(u^m(\tau_k,\xi_{\pm,j}(\tau_k)-) = f'(u^m(\tau_k,\xi_{\pm}(\tau_k)+). 
   \]
   By~\eqref{eq:lipReg1},
   \[
   \begin{split}
     \frac{\zeta_{\pm}(\tau_k) - \xi_{\pm}(\tau_k)}{f'(u^m(s,y\pm))-f'(u(s,x\pm))} & \le
     \lim_{j\to \infty}
     \frac{\zeta_{\pm}(\tau_k) 
     - \xi_{\pm,j}(\tau_k)}{f'(u^m(\tau_k, \zeta_{\pm}(\tau_k)-))-f'(u(\tau_k,\xi_{\pm,j}(\tau_k)-))}\\
     &\le z^m(\tau_k) - M^+[z^m](\tau_k),
   \end{split}
   \]
   which in combination with~\eqref{eq:regAux} shows that the 
   lower bound of~\eqref{eq:lipReg2} holds.

   So far, we have verified the lemma in the following situations: 
   \begin{enumerate}
   \item[(i)] $\dot{z}^m_k =0$,
   \item[(ii)] $f'(u^m(s,y\pm))-f'(u(s,x\pm)) =0$,
   \item[(iii)] $f'(u^m(s,y\pm))-f'(u(s,x\pm)) <0$ and $\dot{z}^m_k >0$.
   \end{enumerate}
   To complete the proof it remains to verify the lemma for
   \begin{enumerate}
   \item[(iv)] $f'(u^m(s,y\pm))-f'(u(s,x\pm)) <0$ and $\dot{z}^m_k <0$,
   \item[(v)]  $f'(u^m(s,y\pm))-f'(u(s,x\pm)) >0$ and $\dot{z}^m_k \neq 0$.
   \end{enumerate}
   These cases can be proved in a similar fashion as (iii). 
   We refer to~~\cite[Lemma~3.3]{Hoel:2017} for further details.
\end{proof}

\begin{proof}[Proof of Corollary~\ref{cor:regularity}]
For an arbitrary $t \in [0,T]$, let us assume that $t \in [\tau_{k},\tau_{k +1}]$ for some $0 \leq k \leq m-1$. 
For any $0 \leq j \le m$, let $P_j$ be the statement: for all $ -\infty
< x<y< \infty$,
\begin{multline}\tag{$P_j$}
 -\frac{1}{M^+[z^m](\tau_j)-z(\tau_j)} \leq \\
 \frac{f'(u^m(\tau_j,y\pm))-f'(u^m(\tau_j,x\pm))}{y-x} \\
 \leq \frac{1}{z(\tau_j)-M^-[z^m](\tau_j)}.
\end{multline}
The statement $P_0$,
\begin{equation*}
 -\infty \leq \frac{f'(u_0(y\pm))-f'(u_0(x\pm))}{y-x} \leq \infty,
 \quad \text{for all} \quad -\infty < x<y< \infty,
\end{equation*}
is obviously true. Furthermore, if $P_j$ is true, Lemma~\ref{lem:regStep}
implies that also $P_{j+1}$ is true.
By induction, we conclude that $P_{k}$ is true and, using
Lemma~\ref{lem:regStep} once more,
we conclude that~\eqref{eq:lipReg1} holds for the considered $t \in [\tau_{k}, \tau_{k+1}]$.
\end{proof}

\section{Orm, Wiener paths and bounded total variation}\label{app:bvWiener}
Before proceeding with the proof of equations~\eqref{eq:wienerOrmBV1}
and~\eqref{eq:wienerOrmBV2} of Lemma~\ref{lem:ormBV}, we recall a few
useful properties on downcrossings for standard Wiener processes.

\begin{theorem}\label{thm:hitting}
  Let $a<m<b$ and consider a standard Wiener process
  $W: [0, \infty) \times \Omega \to \bR$
  with $W(0) = m$, $\bP$-a.s.
  Set $t^* = \min\{t\ge 0 \mid W(t) \in \{a, b\} \}$. Then
\[
\Prob{W(t^*) = a} = \frac{b-m}{b-a}, \qquad
\Prob{W(t^*) = b} = \frac{m-a}{b-a},
\]
and 
\[
\E{t^*} = (m-a)(b-m).
\]
\end{theorem}
For a proof of Theorem~\ref{thm:hitting}, see \cite[Theorem 2.49]{Morters2010}.

\begin{definition}[Downcrossing function]\label{def:downcross}
Let $m>a$ and consider a standard Wiener process $W: [0, \infty) \times \Omega \to \bR$
with $W(0) = m$. Introduce the stopping times $\hat \nu_0 = 0$ and
for $j \ge 1$,
\[
\check \nu_j = \inf \set{t > \hat \nu_{j-1}}{W(t) = a}, 
\qquad  
\hat \nu_j = \inf\set{t > \check \nu_{j}}{W(t) = m}.
\]

The function $W^{(j)}_\downarrow(\cdot,\omega) :[0, \check \nu_j -\hat \nu_{j-1}] \to \bR$ defined by
\[
W^{(j)}_{\downarrow} (s) =W(\hat \nu_{j-1}(\omega) +s, \omega),
\]
thus represents the $j$th downcrossing of $[a,m]$ for the Wiener path
$W(\cdot,\omega)$\footnote{The time $\hat \nu_0$ is the first time
$W(\cdot,\omega)$ equals $m$ and $\check \nu_1$
is the first time after $\hat \nu_0$ that $W(\cdot, \omega)$
equals $a$. Thus $W([\hat \nu_0, \check \nu_1],\omega)$ represents
the first downcrossing of $[a,m]$.
The time $\hat \nu_1$ is the first time after $\check \nu_1$
that $W(\cdot,\omega)$ equals $m$ and $\check \nu_2$ is
the first time after $\hat \nu_1$ that $W(\cdot, \omega)$
equals $a$. Thus $W([\hat \nu_1, \check \nu_2],\omega)$ represents
the second downcrossing of $[a,m]$ \ldots}
For $t>0$, we denote the number of 
downcrossings of $[a,m]$ completed before time $t$ by
\[
D(a,m,t) \coloneq \max\set{j \in \bN}{ \check \nu_j \leq t}.
\]
\end{definition}


See~\cite[Section 6]{Morters2010} for details
on downcrossings for standard Wiener processes.

For $a<m<b$, a standard Wiener process $W: [0, \infty) \times \Omega \to \bR$
  with $W(0) = m$, $\bP$-a.s.~and the stopping time
  \[
  \mathfrak{t}^b \coloneq \inf \set{t\ge 0}{W(t) = b},
  \]
  it follows from Theorem~\ref{thm:hitting} and
  Definition~\ref{def:downcross} that
\begin{equation}\label{eq:geoDist}
  D(a,m, \mathfrak{t}^b) \sim \mathrm{Geo}\prt{\frac{m-a}{b-a}}.
\end{equation}
Here, $\mathrm{Geo}(p)$ denotes the geometric distribution with
parameter $p \in (0,1]$, which for $X\sim \mathrm{Geo}(p)$
has probability mass function
\[
P(X = k ) = p (1-p)^{k}, \quad k = 0,1, \ldots,
\]
and
\begin{equation}\label{eq:geoMoments}
\E{X} = (1-p)/p, \qquad \E{X^2} = (1-p)(2-p)/p^2,
\end{equation}
cf.~\cite{Gut:2013}.

\begin{proof}[Proof of equations~(\ref{eq:wienerOrmBV1})
  and~(\ref{eq:wienerOrmBV2})]
Recalling \eqref{eq:tildeOmega}, it suffices to consider 
Wiener paths $z = W(\cdot, \omega)$ for $\omega \in \widetilde \Omega$. 
Moreover, since
\[
\abs{\orm[z]}_{BV([0,T])} \le \abs{\max(\orm[z],0)}_{BV([0,T])} 
+ \abs{\min(\orm[z],0)}_{BV([0,T])},
\]
and, by symmetry, since the sample path $z$ has the same probability as the sample path $-z$ and
$\orm[z]= -\orm[-z]$, cf.~\eqref{eq:ormMaxRelation},
\[
\E{|\max(\orm[z],0)|_{BV([0,T])}} = \E{|\min(\orm[z],0)|_{BV([0,T])}},
\]
it suffices to verify that
\[
\E{  \abs{y^\dagger}_{BV([0,T])}} < \infty,
\]
where $y^\dagger \coloneq \1{\omega \in \widetilde \Omega} \max(\orm[z],0)$.
Introduce the stopping times
\[
\check \xi_j = \inf \set{t>0}{z(t) = 2^{-j}}, \qquad j \in \bZ
\]
and note that $D(0, 2^{-j}, \check \xi_{j-1} -\check \xi_{j})$ 
equals the number of $z$-downcrossings of $[0, 2^{-{j}}]$
completed in the time interval $[\check \xi_{j}, \check \xi_{j-1}]$
(i.e., between the first time $z$ equals
$2^{-j}$ and the first time it equals $2^{-j+1}$).
By~\eqref{eq:geoDist}, it follows that
\[
D(0, 2^{-j}, \check \xi_{j-1} -\check \xi_{j}) \sim \mathrm{Geo}\prt{\frac{1}{2}}.
\]

Suppose that 
\begin{equation}\label{eq:jumpDDown}
\bar t \in \{t \in [0,T) \mid y^\dagger (t-)>y^\dagger( t) \}.
\end{equation}
Then, since $z \in C_0([0,T])$, there must hold that 
\[
\bar A^+[z](\bar t-)> \bar A^-[z](\bar t-), \qquad
\bar A^+[z](\bar t-) \le \bar A^+[z](\bar t) < \bar A^-[z](\bar t),
\]
and 
\[
y^\dagger (\bar t-)-y^\dagger(\bar t) = M^+[z](\bar t),
\]
cf.~Definition~\ref{def:orm}.
Hence,
\[
y^\dagger(\bar t-)-y^\dagger(\bar t) < 2^{\floor{\log_2(M^+[z](\bar t))}+1},
\]
and
\[
\orm[z](\bar t-) = z(\bar A^+[z](\bar t-))>0  \quad \text{and} \quad
\orm[z](\bar t) = z(\bar A^-[z](\bar t))\le 0.
\]
Consequently, any jump-discontinuity of the form~\eqref{eq:jumpDDown}
with $j = \floor{\log_2(M^+[z](t))}$ must be 
preceded by a $z$-downcrossing of $[0, 2^{-j}]$
within the time interval $[\check \xi_{j}, \check \xi_{j-1}]$
and $\bar t \in [\check \xi_{j}, \check \xi_{j-1}]$.
(For $t<\check \xi_{j}$,
jump-discontinuities in $y^\dagger$ of magnitude greater or equal to
$2^{-j}$ cannot happen,
and at later times, $t > \check \xi_{j-1}$,
all jump-discontinuities of $y^\dagger$ will have
magnitude greater than $2^{-j+1}>M^+[z](t)$.)
Consequently, the jump-discontinuity $\bar t$
may be associated uniquely to the latest $z$-downcrossing
of $[0, 2^{-j}]$ in the time interval $[\check \xi_{j}, \check \xi_{j-1}]$
that precedes $t=\bar t$, and the mapping constituting this
association, from the set~\eqref{eq:jumpDDown} to
the set
\begin{equation}\label{eq:zDownCSet}
\bigcup_{k \in \bZ} \{z\text{-downcrossings of } [0,2^{-k}] \text{ in time interval }
[\check \xi_k, \check \xi_{k-1}] \}
\end{equation}
is thus injective. 

Suppose next that 
\begin{equation}\label{eq:jumpDUp}
\bar t \in \{t \in [0,T) \mid y^\dagger (t-)<y^\dagger( t) \}.
\end{equation}
Then there must hold that
\[
\bar A^+[z](\bar t-)< \bar A^-[z](\bar t-), \qquad  \bar A^+[z](\bar t) > \bar A^-[z](\bar t) \ge \bar A^-[z](\bar t).
\]
and
\[
y^\dagger (\bar t)-y^\dagger(\bar t-) = M^+[z](\bar t).
\]
Hence,
\[
|y^\dagger(\bar t-)-y^\dagger(\bar t)| < 2^{\floor{\log_2(M^+[z](\bar t))}+1},
\]
and
\[
\orm[z](\bar t-) = z(\bar A^-[z](\bar t-))\le0,
\qquad \orm[z](\bar t) = z(\bar A^+[z](\bar t))> 0.
\]
Since $z \in C_0([0,T])$ it follows by~\eqref{eq:MPmRelation}  
\[
M^{+}[z](\bar t-) = z(\bar A^+[z](\bar t-)) = M^+[z](\bar t),
\]
so there exists a time $s=\bar A^+[z](\bar t-) < \bar A^-[z](\bar t-)$ such that
$z(s) = M^+[z](\bar t)$. By the same reasoning as above,
this implies that any jump-discontinuity of the
form~\eqref{eq:jumpDUp} with $j = \floor{\log_2(M^+[z](\bar t))}$
is preceded by a $z$-downcrossing of $[0, 2^{-j}]$
in the time interval $[\check \xi_{j}, \check \xi_{j-1}]$,
and, in fact, $\bar t \in [\check \xi_{j}, \check \xi_{j-1}]$.
Consequently, the jump-discontinuity $\bar t$ may be 
associated uniquely to the latest $z$-downcrossing
of $[0, 2^{-j}]$ in the time interval $[\check \xi_{j}, \check \xi_{j-1}]$
that precedes $t=\bar t$, and the mapping constituting this
association, from the set~\eqref{eq:jumpDUp} to the
set~\eqref{eq:zDownCSet} is thus injective.

For any $j \in \bZ$, let
\[
\mathfrak{B}^-_j
\coloneq \{t \in [0,T) \mid  y^\dagger( t-) - y^\dagger( t) \in[2^{-j}, 2^{-j+1}) \}
\]
and
\[
\mathfrak{B}^+_j
\coloneq \{t \in [0,T) \mid   y^\dagger( t)- y^\dagger( t-) \in[2^{-j}, 2^{-j+1}) \}.
\]  
It then follows that for any $ j \in \bZ$, 
\[
\sum_{ \bar t \in \mathfrak{B}_j^-} 
  |y^\dagger(\bar t-) - y^\dagger(\bar t) |
  \le 2^{-j+1} D(0,2^{-j},  \check \xi_{j-1}-\check \xi_{j})\1{\check \xi_{j}< T}
\]
and
\[
\sum_{ \bar t \in \mathfrak{B}_j^+} |y^\dagger(\bar t-) - y^\dagger(\bar t) |
  \le 2^{-j+1} D(0,2^{-j},  \check \xi_{j-1}-\check \xi_{j})\1{\check \xi_{j}< T}
\]
so that for $\mathfrak{B}_j = \mathfrak{B}_j^- \cup \mathfrak{B}_j^+$,
\[
\sum_{ \bar t \in \mathfrak{B}_j} |y^\dagger(\bar t-) - y^\dagger(\bar t) |
  \le 2^{-j+2} D(0,2^{-j},  \check \xi_{j-1}-\check \xi_{j})\1{\check \xi_{j}< T}.
\]
Including the possible jump-discontinuity of $y^\dagger$ at $\bar t =T$,
and the contribution to the total variation of $y^\dagger$ from
$[0,T] \setminus \cup_{j \in \bZ} \mathfrak{B}_j$,  
we obtain
\[
\begin{split}
  \abs{y^\dagger}_{BV([0,T])} &\le
  \sum_{j \in \bZ} \Big\{ 2^{-j+2} D(0,2^{-j},  \check \xi_{j-1}-\check \xi_{j}) \1{\check \xi_{j} <T}\Big\} \\
  &  +   y^\dagger(T-) + |y^\dagger(T-) - y^\dagger(T)| \\
  &\le \sum_{j \in \bZ} \Big\{ 2^{-j+2} D(0,2^{-j},  \check \xi_{j-1}-\check \xi_{j}) \1{\check \xi_{-j} <T}\Big\} + 2M^+[z](T).
\end{split}
\]
Observe that for all $j \in \bZ$,
\[
Z_j = D(0,2^{-j},  \check \xi_{j-1}-\check \xi_{j}) \sim \mathrm{Geo}\prt{\frac{1}{2}},
\]
$\E{Z_j} = 1$ and $\E{Z_j^2 }=3$, cf.~\eqref{eq:geoMoments}.
By~\cite[eq.~(8.3)]{Karatzas91},
\[
\Prob{M^+[z](T) \in \dx } = \sqrt{\frac{2}{\pi T}} e^{-x^2/(2T)} \dx, \quad x >0.
\]
It therefore holds for all $j \ge j^* \coloneq \ceil{\log_2(T)}+1$
that 
\[
\begin{split}
  \Prob{\hat \xi_{-j} <T} &\le \Prob{M^+[z](T) \ge 2^j}
   = \sqrt{\frac{2}{\pi T}} \int_{2^j}^\infty  e^{-x^2/(2T)} \dx 
   = \sqrt{\frac{2}{\pi T}} e^{-2^j}.
\end{split}
\]
and
\begin{equation}\label{eq:runMaxExp}
\E{M^+[z](T)} = \sqrt{\frac{2}{\pi T}} \int_0^\infty x e^{-x^2/(2T)} \dx = \sqrt{\frac{2T}{\pi}}.
\end{equation}
By H\"older's inequality, 
\[
\begin{split}
&  \sum_{j \in \bZ}^\infty \E{2^{-j+2} D(0,2^{-j},  \check \xi_{j-1}-\check \xi_{j}) \1{\check \xi_{j} <T} } \\
&\le \mspace{-5mu}\sum_{j > -j^*} \E{2^{-j+2} D(0,2^{-j},  \check \xi_{j-1}-\check \xi_{j}) } 
   + \mspace{-5mu}\sum_{j\le -j^*} \E{2^{-j+2} D(0,2^{-j},  \check \xi_{j-1}-\check \xi_{j}) \1{\check \xi_{j} <T} } \\
  & = \sum_{j > -j^*} 2^{-j+2} + \sum_{j\le -j^*} 2^{-j+2}
  \sqrt{\E{\abs{D(0,2^{-j},  \check \xi_{j-1}-\check \xi_{j})}^2}}\sqrt{\Prob{\hat \xi_{j} <T}}\\
  & \le 8T +  \sqrt{3} \prt{\frac{2}{\pi T}}^{1/4} \sum_{j\ge j^*}2^{j+2} e^{-2^{j-1}}
   \le 8T +  \prt{\frac{18}{\pi T}}^{1/4} \sum_{j\ge j^*}2^{-j+5} \\
   &\le 8T +  \prt{\frac{18}{\pi T}}^{1/4} 2^{-j^* + 6} \le
   8T +  \frac{64}{T^{5/4}} .
\end{split}
\]
By~\eqref{eq:runMaxExp} and the preceding inequality, 
\[
\begin{split}
\E{\abs{y^\dagger}_{BV([0,T])}} & \le 
\E{ \sum_{j \in \bZ}^\infty \Big\{ 2^{-j+2} D(0,2^{-j},  
\check \xi_{j-1}-\check \xi_{j}) \1{\check \xi_{-j} <T}\Big\} + 2M^+[z](T)}  \\
& = \sum_{j \in \bZ}^\infty \E{ 2^{-j+2} D(0,2^{-j},  \check \xi_{j-1}-\check \xi_{j}) \1{\check \xi_{-j} <T}} + \sqrt{\frac{8T}{\pi}} \\
& \le 8T + \frac{64}{T^{5/4}}   + \sqrt{\frac{8T}{\pi}}.
\end{split}
\]
This verifies that
\[
\E{\abs{\orm[z]}_{BV([0,T])} } \le 2 \E{\abs{y^\dagger}_{BV([0,T])}} < \infty,
\]
and thus, $\abs{\orm[z]}_{BV([0,T])} < \infty$, $\bP$-a.s.

A similar argument may be employed to verify that
$\E{\sup_{m\ge2} |\orm[z^m]|_{BV([0,T])}} < \infty$.
A short sketch of such an argument follows.

By symmetry, it suffices to verify that
$y^{\dagger, m} = \1{\omega \in \widetilde \Omega}\max(\orm[z^m],0)$ has bounded
total variation, uniformly in $m \ge 2$, $\bP$-a.s.
The only differing
technicality from the preceding argument is that
for fixed $m\ge 2$, any positive/negative jump-discontinuity of $y^{\dagger,m}$ at
time $\bar t$ may be associated uniquely
to a $z$-downcrossing of
$[0, 2^{k}]$ for some $k \ge j = \floor{\log_2(M^+[z^m](\bar t))}$, 
in the time interval $[\check \xi_k, \check \xi_{k-1}]$
(i.e., through an injective mapping from the set of
positive/negative jump-discontinuities of $y^{\dagger,m}$
to the set~\eqref{eq:zDownCSet}).
Moreover $|y^{\dagger,m}(\bar t-) - y^{\dagger,m}(\bar t)| < 2^{j+1} \le 2^{k+1}$.
(The reason for $k\ge j$ in the association is that $z$-downcrossings 
may be more frequent than $z^m$-downcrossings, and they may also
happen at other times.)
It consequently holds that
\[
\sup_{m \ge 2} \abs{y^{\dagger,m}}_{BV([0,T])} \le
  \sum_{j \in \bZ} \Big\{ 2^{-j+2} D(0,2^{-j},  \check \xi_{j-1}-\check \xi_{j}) \1{\check \xi_{j} <T}\Big\}+ 2M^+[z](T),
\]
and the result follows.
\end{proof}

%
%


\end{document}